\documentclass[11pt,reqno]{amsart}

\usepackage{amssymb,amsmath,amsthm,amsfonts}
\usepackage{hyperref, multirow, bm, color, marginnote, graphicx, wrapfig, subcaption, esint, stmaryrd}
\usepackage{diagbox}

\makeatletter
\renewcommand{\tocsection}[3]{%
  \indentlabel{\@ifnotempty{#2}{\bfseries\ignorespaces#1 #2\quad}}\bfseries#3}
\renewcommand{\tocsubsection}[3]{%
  \indentlabel{\@ifnotempty{#2}{\ignorespaces#1 #2\quad}}#3}

\newcommand\@dotsep{4.5}
\def\@tocline#1#2#3#4#5#6#7{\relax
  \ifnum #1>\c@tocdepth 
  \else
    \par \addpenalty\@secpenalty\addvspace{#2}%
    \begingroup \hyphenpenalty\@M
    \@ifempty{#4}{%
      \@tempdima\csname r@tocindent\number#1\endcsname\relax
    }{%
      \@tempdima#4\relax
    }%
    \parindent\z@ \leftskip#3\relax \advance\leftskip\@tempdima\relax
    \rightskip\@pnumwidth plus1em \parfillskip-\@pnumwidth
    #5\leavevmode\hskip-\@tempdima{#6}\nobreak
    \leaders\hbox{$\m@th\mkern \@dotsep mu\hbox{.}\mkern \@dotsep mu$}\hfill
    \nobreak
    \hbox to\@pnumwidth{\@tocpagenum{\ifnum#1=1\bfseries\fi#7}}\par
    \nobreak
    \endgroup
  \fi}
\AtBeginDocument{%
\expandafter\renewcommand\csname r@tocindent0\endcsname{0pt}
}
\def\l@subsection{\@tocline{2}{0pt}{2.5pc}{5pc}{}}
\makeatother

\usepackage[margin=1in]{geometry}

\newcommand{\R}{\mathbb{R}}
\newcommand{\Z}{\mathbb{Z}}

\newcommand{\T}{\mathbb{T}}
\renewcommand{\P}{\mathbb{P}}

\newcommand{\bu}{\bm{u}}

\newcommand{\bx}{\bm{x}}

\newcommand{\X}{\bm{X}}
\newcommand{\Y}{\bm{Y}}
\newcommand{\be}{\bm{e}}

\newcommand{\bv}{\bm{v}}
\newcommand{\p}{\partial}
\renewcommand{\div}{{\rm{div}\,}}

\newcommand{\abs}[1]{\left\lvert #1 \right\rvert}
\newcommand{\norm}[1]{\left\lVert #1 \right\rVert}
\newcommand{\floor}[1]{\left\lfloor #1 \right\rfloor}

\newcommand{\wh}[1]{\widehat{#1}}

\newcommand{\mc}[1]{\mathcal{#1}}

\numberwithin{equation}{section}
\newtheorem{theorem}{Theorem}[section]
\newtheorem{lemma}[theorem]{Lemma}
\newtheorem{proposition}[theorem]{Proposition}
\newtheorem{corollary}[theorem]{Corollary}
\newtheorem{definition}[theorem]{Definition}
\theoremstyle{definition}

\allowdisplaybreaks

\begin{document}
\title[A nonlocal curve evolution for an immersed elastic filament]{A nonlocal curve evolution for an immersed elastic filament: global existence and convergence to resistive force theory}

\author{Laurel Ohm}
\address{Department of Mathematics, University of Wisconsin - Madison, Madison, WI 53706}
\email{lohm2@wisc.edu}

\begin{abstract} 
We consider a nonlocal curve evolution belonging to a hierarchy of models for the dynamics of an inextensible elastic filament in a 3D Stokes fluid. This model captures the principal part of a full free boundary problem for an elastic filament in Stokes flow. The fluid effects on the filament evolution are encoded in a pseudodifferential force-to-velocity operator which may be regarded as an interpolation between resistive force theory at low wavenumbers and a Stokes boundary value problem at high wavenumbers. Here the curve is considered to be the centerline of a 3D filament with constant cross sectional radius $\epsilon>0$. We show global well-posedness for the curve evolution in the natural energy space. This loosely suggests that the full evolution may be globally well-posed if the large-scale geometry is controlled. Furthermore, we prove convergence to resistive force theory dynamics as $\epsilon\to 0$, which illustrates how resistive force theory emerges from more detailed models.
\end{abstract}

\maketitle

\tableofcontents

\setlength{\parskip}{4pt}
\section{Introduction}
We consider a nonlocal curve evolution approximating the motion of an inextensible elastic filament immersed in a Stokes fluid in $\R^3$. 
Our model belongs to a family of curve evolution equations arising in the following way. The thin filament is treated as a rod with rigid cross sections of constant radius $0<\epsilon\ll1$ whose centerline $\X(s,t):\T\times[0,T]\to \R^3$ for $\T:=\R/\Z$ deforms subject to a 1D elasticity law (Figure \ref{fig:filament}). Here the elastic force density
\begin{equation}
  \bm{f}(s) = (\X_{sss}-\tau\X_s)_s\,, \qquad (\cdot)_s:= \frac{\p\cdot}{\p s}\,,
\end{equation}
along the filament centerline comes from Euler-Bernoulli beam theory \cite{camalet2000generic, camalet1999self, hines1978bend, tornberg2004simulating, wiggins1998flexive, wiggins1998trapping}, with the tension $\tau(s,t)$ serving as a Lagrange multiplier to enforce the local inextensibility constraint $\abs{\X_s}^2=1$. The elastic forcing is coupled to the filament motion via a force-to-velocity map 
\begin{equation}
  \mc{L}_\epsilon(\X) : \, \bm{f}(s) \mapsto \bv(s)
\end{equation}
which incorporates the effects of the surrounding Stokes fluid. Models in this family are essentially distinguished by the level of detail with which the fluid effects are included in the map $\mc{L}_\epsilon(\X)$. The position of the curve then evolves according to an equation of the form
\begin{equation}\label{eq:evolution0}
\frac{\p\X}{\p t} = -\mc{L}_\epsilon(\X)\big[(\X_{sss}-\tau\X_s)_s\big] \,, \quad \abs{\X_s}^2=1\,.
\end{equation}

\begin{figure}[!ht]
\centering
\includegraphics[scale=0.3]{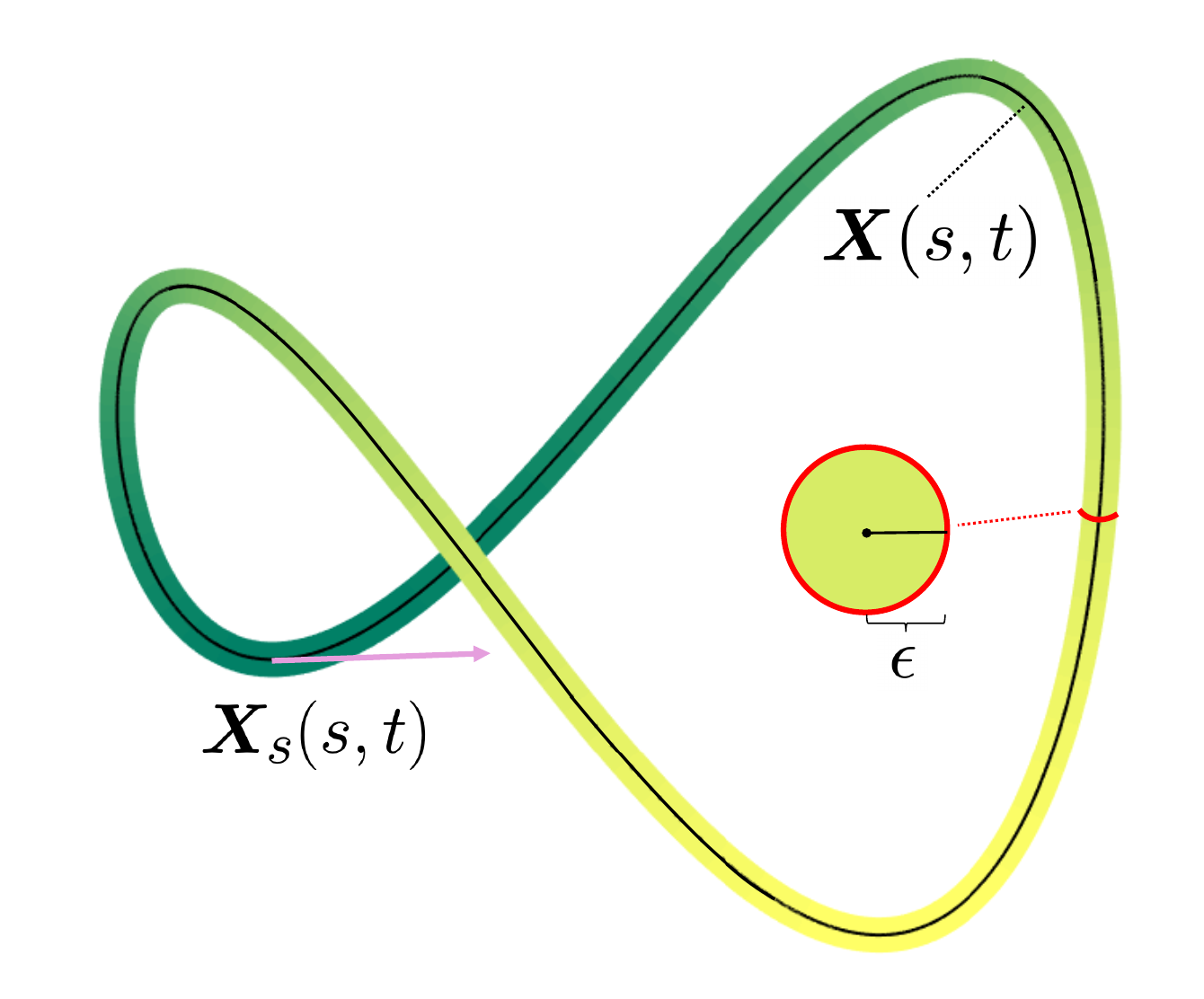}
\caption{We consider $\X(s,t)$ as the centerline of an immersed 3D filament with constant cross sectional radius $0<\epsilon\ll1$. }
\label{fig:filament}
\end{figure}

The simplest choice of force-to-velocity map which still incorporates meaningful hydrodynamic effects is the \emph{resistive force theory} (local slender body theory) \cite{gray1955propulsion, johnson1979flagellar, keller1976swimming, pironneau1974optimal} approximation
\begin{equation}\label{eq:LRFT}
  \mc{L}_{\epsilon,{\rm RFT}}(\X) = \frac{\abs{\log\epsilon}}{4\pi}({\bf I}+\X_s\otimes\X_s)\,.
\end{equation}
The map $\mc{L}_{\epsilon,{\rm RFT}}$ includes only the leading order $O(\abs{\log\epsilon})$ effect of the surrounding viscous fluid on the filament, resulting in a local force-to-velocity operator. The main feature of $\mc{L}_{\epsilon,{\rm RFT}}$ is a drag anisotropy weighting the viscous drag in the tangential ($\X_s$) direction along the filament twice as much as in the normal directions. This drag anisotropy is a simple but important nonlinear effect that, among other consequences, allows for net displacement (swimming) of the filament. Resistive force theory models of elastohydrodynamics are therefore useful for understanding undulatory swimming at low Reynolds number \cite{friedrich2010high, el2020optimal, gadelha2010nonlinear, gadelha2019flagellar, hu2022enhanced, montenegro2015spermatozoa, lauga2013shape, lauga2007floppy,  spagnolie2010optimal, lauga2009hydrodynamics, mori2023well,ohm2024well}. These local dynamics also serve as a useful testing ground for efficient numerical implementations of inextensibility \cite{moreau2018asymptotic,hall2019efficient,walker2020efficient}.

Resistive force theory is an important member of the immersed filament modeling hierarchy, but for certain applications, its treatment of hydrodynamic effects is overly simplistic. A natural, more detailed candidate for a force-to-velocity map $\mc{L}_\epsilon$ is nonlocal slender body theory\footnote{In nonlocal slender body theory, an approximation of the fluid velocity away from the filament centerline is obtained by distributing Stokeslets--the free-space Green's function for the Stokes equations in $\R^3$--and higher order Stokes doublet corrections along the filament centerline with density $\bm{f}(s)$. The expression for $\mc{L}_{\epsilon,\rm nloc}$ comes from evaluating this expression on the surface of the filament (${\rm dist}(\bx,\X)=\epsilon$) and expanding about $\epsilon=0$. } \cite{keller1976slender,johnson1980improved,tornberg2004simulating}:
\begin{equation}
\begin{aligned}
  \mc{L}_{\epsilon,\rm nloc}[\bm{f}](s) &= \frac{1}{8\pi}\big[({\bf I}-3\X_s\otimes\X_s)-2({\bf I}+\X_s\otimes\X_s)\log(\pi\epsilon/4) \big]\bm{f}(s)\\
  &\quad +\frac{1}{8\pi}\int_\T\bigg[\bigg(\frac{{\bf I}}{|\wh{\X}|}+ \frac{\wh{\X}\otimes\wh{\X}}{|\wh{\X}|^3}\bigg)\bm{f}(s')- \frac{{\bf I}+\X_s\otimes\X_s}{\abs{\sin(\pi(s-s'))/\pi}}\bm{f}(s)\bigg]\,ds'
\end{aligned}
\end{equation}
where $\wh{\X} = \X(s)-\X(s')$. We mention this operator due to its numerical utility; however, due to well-known issues \cite{gotz2000interactions,tornberg2004simulating,shelley2000stokesian,inverse} at high wavenumbers $\sim \frac{1}{\epsilon}$, an evolution of the form \eqref{eq:evolution0} using $\mc{L}_{\epsilon,\rm nloc}$ is ill-posed, even at analytic regularity.

We thus turn to arguably the most detailed choice of force-to-velocity map suitable for \eqref{eq:evolution0}: the \emph{slender body Neumann-to-Dirichlet (NtD) map}, introduced by the author together with Mori and Spirn in \cite{closed_loop,free_ends}. The slender body NtD map comes from solving the \emph{slender body boundary value problem} for 3D Stokes flow about a filament of radius $\epsilon>0$ using force data defined on a curve. The full definition of the slender body NtD map, which we will denote by $\mc{L}_{\epsilon,\rm NtD}$, appears in section \ref{subsec:SB_NtD}. The evolution \eqref{eq:evolution0} using $\mc{L}_{\epsilon,\rm NtD}$ gives rise to the slender body free boundary problem, a curve evolution in which the fluid surrounding the 3D filament exactly satisfies a Stokes boundary value problem at each moment in time. In \cite{ohm2025angle,ohm2024free}, we showed that the principal behavior of $\mc{L}_{\epsilon,\rm NtD}$ about a curved filament is captured by the slender body NtD map about a straight cylinder, for which we derived an explicit symbol in \cite{inverse}. From this decomposition, in \cite{ohm2025slender}, we showed that the slender body free boundary problem is locally well-posed in $C^{4,\alpha}(\T)$. Questions of global well-posedness and long-time convergence to equilibrium remain completely open.

This brings us to the curve evolution studied in this paper. Here we consider a natural and novel intermediate model of the form \eqref{eq:evolution0} between the full slender body free boundary problem \eqref{eq:SB_FBP} and the classical resistive force theory model \eqref{eq:RFT_FBP}. Our choice of the force-to-velocity map, which we will denote by $\overline{\mc{L}}_\epsilon$, is a pseudodifferential operator whose symbol in the tangential and normal directions to the filament centerline is the explicit Fourier multiplier corresponding to the exact slender body NtD map about a straight cylinder. In particular, we simply apply the straight map to a curved filament (see Definition \ref{def:Lepsbar}). At low wavenumbers, the map behaves like multiplication by $\abs{\log\epsilon}$, with the tangential direction along the filament weighted twice as much as the normal directions. At high wavenumbers $|k|\gtrsim \frac{1}{\epsilon}$, the map behaves like multiplication by $(\epsilon\abs{k})^{-1}$ on the Fourier side, i.e., a (scaled) inverse derivative.
In a sense, the force-to-velocity map $\overline{{\mc L}}_\epsilon$ interpolates between resistive force theory at low wavenumbers and the full slender body NtD map at high wavenumbers.

Our aim is twofold: 
(I.) We seek to probe the global well-posedness question for the full slender body free boundary problem \eqref{eq:SB_FBP} by replacing $\mc{L}_{\epsilon,\rm NtD}$ with its principal part. Here we prove global well-posedness for \eqref{eq:evolution0} using the map $\overline{\mc{L}}_\epsilon$ in the natural energy space. One caveat is that we treat \eqref{eq:evolution0} with $\overline{\mc{L}}_\epsilon$ as a pure curve evolution and do not worry about self-intersection of the filament. While self-intersection would lead to breakdown of the full free boundary problem, we can continue to make sense of the model with $\overline{\mc{L}}_\epsilon$ even if the filament passes through itself. Self-intersection is inherently a low-wavenumber phenomenon, and at low wavenumbers, $\overline{\mc{L}}_\epsilon$ looks like resistive force theory rather than $\mc{L}_{\epsilon,\rm NtD}$. However, our global well-posedness result for \eqref{eq:evolution0} using $\overline{\mc{L}}_\epsilon$ loosely suggests that if the large-scale geometry of the full slender body free boundary evolution remains under control, then the full free boundary problem should be globally well-posed.

(II.) In addition, we seek to illustrate how resistive force theory dynamics using $\mc{L}_{\epsilon,{\rm RFT}}$ emerge from more detailed models in the singular limit $\epsilon\to 0^+$. We prove that the evolution \eqref{eq:evolution0} with $\overline{\mc{L}}_\epsilon$ converges to the resistive force theory evolution with $\mc{L}_{\epsilon,{\rm RFT}}$ as the radius parameter $\epsilon\to 0$, thereby establishing the first rigorous dynamical connection between different levels of the curve evolution hierarchy. The convergence is at an expected but very slow logarithmic-in-$\epsilon$ rate, since \eqref{eq:LRFT} only captures the very leading order local effects of the Stokes fluid on the filament. The slow convergence does make the proof delicate, as there is no room for waste.

In \cite{albritton2025rods}, we recognize \eqref{eq:evolution0} with $\mc{L}_{\epsilon,{\rm RFT}}$ as a gradient flow of the filament bending energy $\frac{1}{2}\int_\T\abs{\X_{ss}}^2\,ds$, which allows us to show long-time convergence to 3D Euler elasticae, the critical points of the bending energy. The evolution with $\overline{\mc{L}}_\epsilon$ is also a gradient flow of the bending energy with respect to the metric 
\begin{equation}
  \langle \bm{V},\bm{U}\rangle := \langle \overline{\mc{L}}_\epsilon^{-1}\bm{V},\bm{U}\rangle_{L^2(\T)}
\end{equation}
defined on perturbations $\bm{V},\bm{U}$ of an inextensible curve $\X$. We may thus frame our convergence result as that of two different gradient flows for the same energy whose metrics converge as $\epsilon \to 0$.

\subsection{The force-to-velocity map}
We now define the force-to-velocity map $\overline{\mc{L}}_\epsilon$ used in \eqref{eq:evolution0}. Given $\bm{h}(s):\T\to\R^3$ defined along a filament centerline $\X$, we denote the projection of $\bm{h}$ onto the tangential and normal directions about $\X$ as 
\begin{equation}\label{eq:projections}
  \P_{\X_s}\bm{h} = (\X_s\cdot\bm{h}) \X_s\,, \qquad
  \P_{\X_s}^\perp\bm{h} = \bm{h} - \P_{\X_s}\bm{h}\,.
\end{equation}
In addition, we denote the Fourier transform of $\bm{h}$ by
\begin{equation}
  \mc{F}[\bm{h}(s)](k) = \int_\T \bm{h}(s)e^{i2\pi ks}\,ds\,.
\end{equation}
Given a scalar-valued Fourier multiplier $m(k)$, we will use the notation $T_m$ to denote the operator 
\begin{equation}\label{eq:Tmh}
  T_m\bm{h} = \mc{F}^{-1}[m\,\mc{F}[\bm{h}]] \,.
\end{equation}

\begin{definition}[Force to velocity map]\label{def:Lepsbar}
Given $\X\in H^2(\T)$ and a line force density $\bm{f}\in L^2(\T)$, the force-to-velocity map $\overline{\mc{L}}_\epsilon$ is given by 
\begin{equation}
  \overline{\mc{L}}_\epsilon[\bm{f}]= \P_{\X_s}\,T_{m_\epsilon^{\rm t}}(\P_{\X_s}\bm{f}) + \P_{\X_s}^\perp\,T_{m_\epsilon^{\rm n}}(\P_{\X_s}^\perp\bm{f})\,,
\end{equation}
where, for $k\neq 0$, the multipliers $m_\epsilon^{\rm t}(k)$, $m_\epsilon^{\rm n}(k)$ are given by
\begin{align}
\label{eq:multT}
m_\epsilon^{\rm t}(k) &= \frac{2K_0K_1 + 2\pi\epsilon \abs{k} \big( K_0^2 - K_1^2 \big) }{ 8\pi^2\epsilon \abs{k} K_1^2} \\
\label{eq:multN}
m_\epsilon^{\rm n}(k) &= 
\frac{2K_0K_1K_2 + 2\pi\epsilon \abs{k} \big(K_1^2(K_0+K_2)-2K_0^2K_2 \big)}{4\pi^2\epsilon \abs{k}\big(4K_1^2K_2+2\pi\epsilon \abs{k} K_1(K_1^2-K_0K_2)\big)}\,,
\end{align}
where $K_j=K_j(2\pi\epsilon \abs{k})$, $j=0,1,2$, are $j^{\text th}$ order modified Bessel functions of the second kind. 

For $k=0$, we define 
\begin{equation}\label{eq:multN0}
m_\epsilon^{\rm t}(0) = \frac{\abs{\log\epsilon}}{2\pi} \,,
\qquad
m_\epsilon^{\rm n}(0) = \frac{\abs{\log\epsilon}}{4\pi}\,.
\end{equation}
\end{definition}

We provide plots of the \emph{inverse} multipliers $m_\epsilon^{\rm t}(k)^{-1}$ and $m_\epsilon^{\rm n}(k)^{-1}$ in Figure \ref{fig:multipliers} to more clearly display the transition from logarithmic behavior ($m_\epsilon^{\rm t}(k)^{-1},m_\epsilon^{\rm n}(k)^{-1}\sim \abs{\log(\epsilon |k|)}^{-1}$) at low wavenumbers to linear growth $(m_\epsilon^{\rm t}(k)^{-1},m_\epsilon^{\rm n}(k)^{-1}\sim \epsilon |k|)$ at high wavenumbers.

\begin{figure}[!ht]
\centering
\includegraphics[scale=0.4]{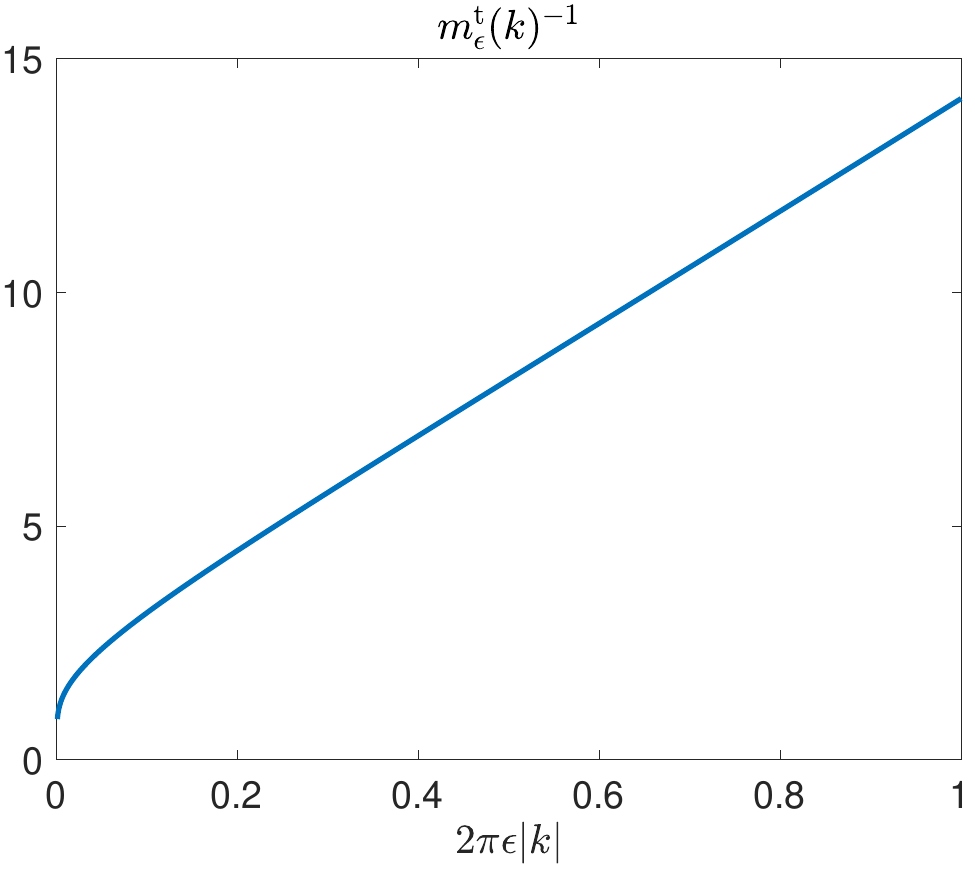}
\includegraphics[scale=0.4]{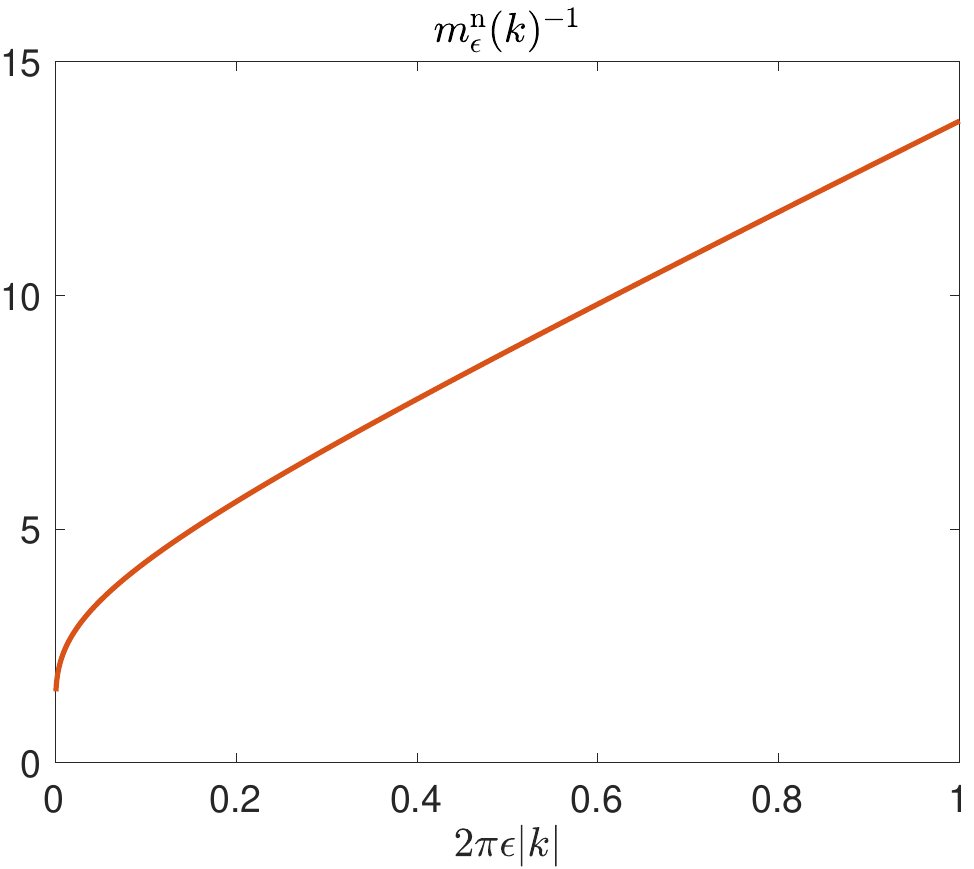}
\caption{Plots of the inverse tangential (left) and normal (right) direction multipliers $m_\epsilon^{\rm t}(k)^{-1}$ and $m_\epsilon^{\rm n}(k)^{-1}$ versus $2\pi\epsilon |k|$. Here the inverses are used to more clearly display the linear ($\sim \epsilon |k|$) behavior at high wavenumbers $|k|\gtrsim \frac{1}{\epsilon}$.}
\label{fig:multipliers}
\end{figure}

For $k\neq 0$, the multipliers $m_\epsilon^{\rm t}(k)$ and $m_\epsilon^{\rm n}(k)$ are precisely the tangential and normal eigenvalues, respectively, of the slender body Neumann-to-Dirichlet map in the special filament geometry of a straight cylinder. 

The regularization \eqref{eq:multN0} for the $k=0$ mode is used since the expressions \eqref{eq:multT}, \eqref{eq:multN} blow up like $\abs{\log(\epsilon |k|)}$ as $k\to 0$. This stems from Stokes' paradox for the straight (periodic/infinite) cylinder, for which these multipliers are calculated, and is not expected to be a physically realistic or meaningful approximation of the zero-mode behavior for a curved, closed filament. In particular, the sense in which $\overline{\mc{L}}_\epsilon$ serves as an approximation for the `full' slender body NtD map (see section \ref{subsec:SB_NtD}) is primarily through high wavenumber behavior. 
Here the choice of regularization for the $k=0$ mode is chosen to match with resistive force theory \eqref{eq:LRFT} and thus facilitate the convergence result of Theorem \ref{thm:RFTconv}. We emphasize that $\overline{\mc{L}}_\epsilon$ in a sense interpolates between resistive force theory at low wavenumbers and the full slender body NtD map at high wavenumbers.

\subsection{Main theorem statements}
Through multiplier and tension bounds outlined in the following sections, we show that the evolution 
\begin{equation}\label{eq:evolution}
\frac{\p\X}{\p t} = -\overline{\mc{L}}_\epsilon(\X)\big[(\X_{sss}-\tau\X_s)_s\big] \,, \quad \abs{\X_s}^2=1\,.
\end{equation}
is a third-order, semilinear, parabolic equation for $\X$, coupled with an elliptic equation for the filament tension $\tau$. It therefore shares features of both resistive force theory and the full slender body free boundary problem. 

In section \ref{sec:GWP}, we show:
\begin{theorem}[Global well-posedness]\label{thm:GWP}
  Given $\epsilon>0$ and an initial curve $\X^{\rm in}\in H^2(\T)$, the evolution \eqref{eq:evolution} with $\overline{\mc{L}}_\epsilon$ as in Definition \ref{def:Lepsbar} admits a unique global solution 
\begin{equation}
  \X\in C([0,+\infty);H^2_s)\cap L^2_{t,{\rm loc}}H^{7/2}_s([0,+\infty) \times \T)
\end{equation}
with initial data $\X^{\rm in}$.
\end{theorem} 

The starting point for Theorem \ref{thm:GWP} is the energy identity 
\begin{equation}\label{eq:energy0}
\frac{1}{2}\p_t\int_\T \abs{\X_{ss}}^2\,ds = -\int_{\T}(\X_{sss}-\tau\X_s)_s\cdot\overline{\mc{L}}_\epsilon[(\X_{sss}-\tau\X_s)_s]\,ds\,,
\end{equation}
which is explored in greater detail in sections \ref{subsec:energetics} and \ref{subsec:GWP}. The main difficulty is that it is not immediately clear that the right hand side dissipation controls the curve $\X_{ssss}$ alone, due to the presence of the tension. In addition, necessary ingredients in the proof of local well-posedness in the energy space are \emph{a priori} estimates for the tension, which in section \ref{subsec:tension} is shown to satisfy an elliptic equation at each time. Interestingly, and somewhat surprisingly, the tension is better behaved than its counterpart in resistive force theory, which is critical in a certain sense.\footnote{In particular, the tension $\tau_Y$ enforcing the inextensibility constraint in \eqref{eq:RFT_FBP} may be estimated (see \cite[Lemma 3.3]{albritton2025rods}) as
\begin{equation}\label{eq:RFTtension1}
  \norm{\tau_Y}_{H^1(\T)} \le c\norm{\Y_{ss}}_{H^1(\T)}^2\le c\norm{\Y_{ss}}_{L^2(\T)}\norm{\Y_{ss}}_{H^2(\T)}\,,
\end{equation}
where we emphasize that the right hand side involves a full copy of $\Y\in H^4_s$ and thus requires the full smoothing power of the semigroup $e^{-t\p_s^4}$ to estimate in the energy space. See \cite[Remark 2.4]{mori2023well} for additional discussion.} This may be due to the more physically realistic behavior of $\overline{\mc{L}}_\epsilon$ at high wavenumbers for capturing fluid effects on immersed filament dynamics. A final key element is our decomposition of the main operator $\overline{\mc{L}}_\epsilon\p_s^4$ in section \ref{subsec:mainOpDecomp}, which exploits the inextensibility in a crucial way.

Our next contribution is to compare the dynamics of the curve evolution \eqref{eq:evolution} directly with the resistive force theory evolution 
\begin{equation}\label{eq:RFT_FBP}
  \frac{\p\Y}{\p t} = -\mc{L}_{\epsilon,{\rm RFT}}(\Y)\big[(\Y_{sss}-\tau_Y\Y_s)_s\big] \,, \quad \abs{\Y_s}^2=1\,,
\end{equation}
where $\mc{L}_{\epsilon,{\rm RFT}}$ is given by the local operator \eqref{eq:LRFT}. Here we use $\Y$ to distinguish from the evolution \eqref{eq:evolution}. Since $\mc{L}_{\epsilon,{\rm RFT}}$ is local, \eqref{eq:RFT_FBP} is a fourth-order semilinear parabolic equation for $\Y(s,t)$, so the natural energy space $L^\infty_tH^2_s\cap L^2_tH^4_s$ is more regular than that of the evolution \eqref{eq:evolution}. The PDE theory of \eqref{eq:RFT_FBP} has been developed by the author and others in \cite{albritton2025rods,mori2023well,moreau2025n}. By \cite[Theorem 1.1]{albritton2025rods}, starting from initial data $\X^{\rm in}\in H^2(\T)$, equation \eqref{eq:RFT_FBP} admits a unique global solution $\Y\in C_tH^2_s\cap L^2_{t,{\rm loc}}H^4_s$, and starting from $\X^{\rm in}\in H^4(\T)$, the unique global solution may be upgraded to $\Y\in C_tH^4_s\cap L^2_{t,{\rm loc}}H^6_s$.

Noting that the behavior of both $\overline{\mc{L}}_\epsilon$ (see Lemma \ref{lem:multiplierbds}) and $\mc{L}_{\epsilon,\rm RFT}$ blows up logarithmically in $\epsilon$ as $\epsilon\to 0$, we consider convergence under the time rescaling 
\begin{equation}\label{eq:tunder1}
  \underline{t} = \abs{\log\epsilon}t\,, \qquad \abs{\log\epsilon}\frac{\p}{\p\underline{t}} = \frac{\p}{\p t}\,.
\end{equation}
In section \ref{sec:RFTconverge}, we then show the following convergence result:
\begin{theorem}[Convergence to RFT dynamics]\label{thm:RFTconv}
  Given any $0<T<\infty$ and an initial curve $\X^{\rm in}\in H^4(\T)$, let $\X\in C_tH^2_s\cap L^2_tH^{7/2}_s([0,T]\times \T)$ and $\Y\in C_tH^4_s\cap L^2_tH^6_s([0,T]\times \T)$ be the unique solutions to \eqref{eq:evolution} and the resistive force theory evolution \eqref{eq:RFT_FBP}, respectively, under the time rescaling \eqref{eq:tunder1} and with initial data $\X^{\rm in}$. For $\epsilon>0$ sufficiently small, there exists an $\epsilon$-independent constant $C_{\rm RFT}=C_{\rm RFT}(T,\norm{\X}_{L^\infty_tH^2_s},\norm{\Y}_{L^\infty_tH^4_s\cap L^2_tH^6_s})$ such that
\begin{equation}
  \norm{\X-\Y}_{L^\infty_t H^2_s} + \abs{\log\epsilon}^{1/2}\norm{\X-\Y}_{L^2_t\dot H^{7/2}_s} 
  \le \abs{\log\epsilon}^{-1/2}C_{\rm RFT}\,.
\end{equation}
In particular, on any fixed time interval, the curve $\X(s,t)$ converges to $\Y(s,t)$ as $\epsilon\to 0$.
\end{theorem}
Note that the very slow convergence rate is expected, and, for $L^2_t\dot H^{7/2}_s$, is likely optimal.

\subsection{Connection with the slender body Neumann-to-Dirichlet map}\label{subsec:SB_NtD}
As mentioned, arguably the most detailed method for incorporating 3D fluid effects into the force-to-velocity map $\mc{L}_\epsilon$ is by solving the following \emph{slender body boundary value problem}, introduced by the author together with Mori and Spirn in \cite{closed_loop,free_ends}. Given $0<\epsilon\ll1$ sufficiently small, we may define a `fattened' version of the curve $\X$ as 
\begin{equation}\label{eq:SigmaEps}
  \Sigma_\epsilon(t) = \big\{ \bx\in \X_s(s,t)^\perp\,:\, {\rm dist}(\bx,\X(s,t))<\epsilon\,, \, s\in \T \big\}\,.
\end{equation}
The 3D body $\Sigma_\epsilon$ may be considered as a regularization of the filament centerline which allows us to seek $(\bu,p): \R^3\backslash \overline{\Sigma_\epsilon} \to \R^3\times\R$ satisfying the boundary value problem
\begin{equation}\label{eq:SB_PDE}
\begin{aligned}
-\Delta \bu +\nabla p &= 0\,, \quad \div\bu=0 \qquad \text{in }\R^3\backslash \overline{\Sigma_\epsilon} \\
\int_0^{2\pi}(\bm{\sigma}[\bu]\bm{n}) \, \mc{J}_\epsilon(s,\theta)\,d\theta &= \bm{f}(s) \qquad\qquad\qquad\; \text{on }\p\Sigma_\epsilon \\
\bu\big|_{\p\Sigma_\epsilon} &= \bv(s)\,, \qquad\qquad\quad\;\;\, \text{unknown but independent of }\theta\,,   
\end{aligned}
\end{equation}
and $\abs{\bu}\to 0$ as $\abs{\bx}\to\infty$. In this boundary value problem, we make sense of the force data $\bm{f}(s)$ defined along a 1D curve as the total surface stress $\bm{\sigma}[\bu]\bm{n}= (\nabla\bu+\nabla\bu^{\rm T}-p{\bf I})\bm{n}$ per cross section of the 3D body $\Sigma_\epsilon$, weighted by the body's surface area through a Jacobian factor $\mc{J}_\epsilon$. Since this is only a partial, `angle-averaged' Neumann boundary condition, it must be supplemented with an additional partial boundary condition in the form of a geometric constraint on the Dirichlet boundary value of $\bu$ known as a fiber integrity condition. In particular, $\bu\big|_{\p\Sigma_\epsilon}$ is unknown but constrained to be a function of arclength $s$ only.

The \emph{slender body Neumann-to-Dirichlet (NtD) map} $\mc{L}_{\epsilon,\rm NtD}$ is then defined as the map
\begin{equation}\label{eq:SB_NtD}
\mc{L}_{\epsilon,\rm NtD}(\X) : \bm{f}(s) \mapsto \bm{v}(s)
\end{equation}
which comes from solving the (Neumann) boundary value problem \eqref{eq:SB_PDE} for $\bu$ and evaluating the Dirichlet boundary value on $\p\Sigma_\epsilon$. The solution to this boundary value problem depends on both the filament radius $\epsilon$ and the centerline geometry $\X(s)$ in a complicated way. To isolate these complexities, we can consider the slender body NtD map in the simplest possible geometry, which here is the straight, constant-radius filament with periodic boundary conditions at the ends. 

 In particular, let $\mc{C}_\epsilon$ denote the body \eqref{eq:SigmaEps} with straight centerline $\X(s)=s\be_z$, $s\in \T$. The behavior of the slender body NtD map \eqref{eq:SB_NtD} about $\mc{C}_\epsilon$ conveniently diagonalizes, with differing behavior in directions tangent and normal to the filament centerline. In particular, for $k\neq 0$, it may be shown (see \cite{inverse}) that
 \begin{equation}\label{eq:eval_prob}
 \begin{aligned}
\mc{L}_{\epsilon,\rm NtD}(s\be_z)[e^{2\pi iks}\be_z] &= m_{\epsilon}^{\rm t}(k)e^{2\pi iks}\be_z\,, \\
\mc{L}_{\epsilon,\rm NtD}(s\be_z)[e^{2\pi iks}\be_j] &= m_{\epsilon}^{\rm n}(k)e^{2\pi iks}\be_j\,, \quad j=x,y \,,
\end{aligned}
\end{equation}
where the eigenvalues $m_{\epsilon}^{\rm t}(k)$ and $m_{\epsilon}^{\rm n}(k)$ are as in Definition \ref{def:Lepsbar}. More specifically, for the straight filament, we have that the slender body NtD map $\mc{L}_{\epsilon,\rm NtD}(s\be_z)[\cdot]$ is exactly the operator $\overline{\mc{L}}_\epsilon(s\be_z)[\cdot]$ of Definition \ref{def:Lepsbar} (after removing the $k=0$ mode).

For a general curved filament, the map $\mc{L}_{\epsilon,\rm NtD}(\X)$ no longer admits an explicit Fourier multiplier representation. However, in \cite[Theorem 1.5 \& Corollary 1.6]{ohm2024free}, it is shown that the principal part of the general operator $\mc{L}_{\epsilon,\rm NtD}(\X)$ is given by the straight NtD map $\mc{L}_{\epsilon,\rm NtD}(s\be_z)$ in the sense that
\begin{equation}\label{eq:NtD_decomp}
  \mc{L}_{\epsilon,\rm NtD}(\X)[\p_s^4\X] = \mc{L}_{\epsilon,\rm NtD}(s\be_z)[\p_s^4\X] + \mc{R}_{\rm err}\,,
\end{equation}
where the remainder terms $\mc{R}_{\rm err}$ are lower order with respect to regularity or size in $\epsilon$. In \cite[Theorem 1.1]{ohm2025slender}, this decomposition is leveraged to show qualitative local well-posedness for the full \emph{slender body free boundary problem}
\begin{equation}\label{eq:SB_FBP}
  \frac{\p\X}{\p t} = -\mc{L}_{\epsilon,\rm NtD}(\X)\big[(\X_{sss}-\tau\X_s)_s\big] \,, \quad \abs{\X_s}^2=1\,,
\end{equation}
where $\mc{L}_{\epsilon,\rm NtD}(\X)$ is given by \eqref{eq:SB_NtD} for the instantaneous curve shape $\X$. More precisely, local well-posedness is shown in $C([0,T];h^{4,\alpha}(\T))$ for an initial curve belonging to the little H\"older space $h^{4,\alpha}(\T)$. Note that, by \eqref{eq:NtD_decomp}, the evolution \eqref{eq:SB_FBP} is a third order, quasilinear parabolic equation for $\X$. In particular, this regularity is shared by the evolution \eqref{eq:evolution} considered here.  

The setting of H\"older spaces is more convenient for the operator decomposition \eqref{eq:NtD_decomp}, but less convenient for obtaining insight into long-time dynamics. Here our choice of operator $\overline{\mc{L}}_\epsilon$ as in Definition \ref{def:Lepsbar} is inspired by the principal straight part of the slender body NtD map, adapted slightly for application directly to a curved filament. This allows us to develop the solution theory of the principal dynamics in a natural energy space while avoiding the additional difficulties associated with (a.) decomposing the full operator as in \eqref{eq:NtD_decomp} in lower regularity spaces, and (b.) accounting for filament self-intersection.

\section{Multiplier bounds and energetics}
Before proceeding to the proof of Theorem \ref{thm:GWP}, we make note of some useful bounds for the multipliers $m_\epsilon^{\rm t}$ and $m_\epsilon^{\rm n}$ comprising the operator $\overline{\mc{L}}_\epsilon$. These may be used to show that the energy identity satisfied by solutions to \eqref{eq:evolution} is in fact signed, which will be crucial for the global-in-time behavior of the filament.

\subsection{Multiplier bounds}
We begin with a collection of lemmas bounding the behavior of the multipliers $m_\epsilon^{\rm t}(k)$ and $m_\epsilon^{\rm n}(k)$ in both $k$ and $\epsilon$. 
First, from \cite[Lemmas 3.4 \& 3.5]{ohm2024free}, we note the following general bounds for the behavior of $m_\epsilon^{\rm t}$ and $m_\epsilon^{\rm n}$, which differs at high versus low wavenumbers. 
\begin{lemma}[Multiplier behavior \cite{ohm2024free}]\label{lem:multiplierbds}
Given $\epsilon>0$ and $m_\epsilon^{\rm t}(k)$ and $m_\epsilon^{\rm n}(k)$ as in \eqref{eq:multT}, \eqref{eq:multN}, there exist constants $c>0$ such that 
  \begin{equation}
  \begin{aligned}
    m_\epsilon^{\rm t}(k)\,,\;m_\epsilon^{\rm n}(k) &\le \begin{cases}
      c\abs{\log\epsilon} \,, & \abs{k}< \frac{1}{2\pi\epsilon}\\
      c\,\epsilon^{-1}\abs{k}^{-1} \,, & \abs{k}\ge \frac{1}{2\pi\epsilon}
    \end{cases} \,,\\
    m_\epsilon^{\rm t}(k)^{-1}\,,\; m_\epsilon^{\rm n}(k)^{-1} &\le \begin{cases}
      c\abs{\log\epsilon}^{-1} \,, & \abs{k}< \frac{1}{2\pi\epsilon}\\
      c\,\epsilon\abs{k} \,, & \abs{k}\ge \frac{1}{2\pi\epsilon}\,.
    \end{cases}
  \end{aligned}
  \end{equation}
\end{lemma}
This behavior is displayed in Figure \ref{fig:multipliers}.
In addition, we will require some refinements to Lemma \ref{lem:multiplierbds} at both high and low wavenumbers. For high wavenumbers, by \cite[Proposition 1.4]{inverse}, we may obtain a more precise linear growth rate for $(m_\epsilon^{\rm t})^{-1}$ and $(m_\epsilon^{\rm n})^{-1}$ as follows.  
\begin{lemma}[High wavenumber refinement \cite{inverse}]\label{lem:highk_refinement}
There exist constants $c_{\rm t}$, $c_{\rm n}>0$ such that for all $k>0$, the multipliers $m_\epsilon^{\rm t}$ and $m_\epsilon^{\rm n}$ in \eqref{eq:multT}, \eqref{eq:multN} satisfy 
   \begin{equation}
   \begin{aligned}
    c_{\rm t}^{-1}\epsilon\abs{k} &\le m_\epsilon^{\rm t}(k)^{-1} \le c_{\rm t}^{-1}(\epsilon\abs{k}+1)
    \qquad
    c_{\rm n}^{-1}\epsilon\abs{k} &\le m_\epsilon^{\rm n}(k)^{-1} \le c_{\rm n}^{-1}(\epsilon\abs{k}+1)\,.
    \end{aligned}
  \end{equation}
\end{lemma}
For low wavenumbers, we note the following refined asymptotics, which will allow us to prove the convergence result of Theorem \ref{thm:RFTconv}. 
\begin{lemma}[Low wavenumber refinement]\label{lem:lowk_refine}
  For $0<\epsilon<1$ and for $0< \abs{k}<\frac{1}{2\pi\epsilon}$, we have the refined low wavenumber bounds
\begin{equation}\label{eq:lowbd}
  m_\epsilon^{\rm t}(k) - \frac{\abs{\log\epsilon}}{2\pi} \le c(1+\abs{\log \abs{k}})\,,
  \quad
  m_\epsilon^{\rm n}(k) - \frac{\abs{\log\epsilon}}{4\pi} \le c(1+\abs{\log \abs{k}})\,,
\end{equation}
along with $m_\epsilon^{\rm t}(0) - \frac{\abs{\log\epsilon}}{2\pi}=m_\epsilon^{\rm n}(0) - \frac{\abs{\log\epsilon}}{4\pi}=0$.
\end{lemma}
In particular, this difference is bounded independent of $\epsilon$ at low wavenumbers.

\begin{proof}[Proof of Lemma \ref{lem:lowk_refine}]
Expanding $m_\epsilon^{\rm t}(k)$ and $m_\epsilon^{\rm n}(k)$ in a Taylor series about $\abs{k}=0$ yields 
\begin{equation}
\begin{aligned}
  m_\epsilon^{\rm t}(k) &= \frac{-1-2\gamma +2\log\pi -2\log(\epsilon \abs{k})}{4\pi} +O(\epsilon^2k^2\log^2(\epsilon |k|)) \\
  m_\epsilon^{\rm n}(k) &= \frac{1 - 2\gamma - 2\log\pi - 2\log(\epsilon \abs{k})}{8\pi} +O(\epsilon^2k^2\log^2(\epsilon |k|))\,.
\end{aligned}
\end{equation}
Here $\gamma\approx 0.5772$ is the Euler gamma. Subtracting $\frac{\abs{\log\epsilon}}{2\pi}$ from $m_\epsilon^{\rm t}(k)$ and $\frac{\abs{\log\epsilon}}{4\pi}$ from $m_\epsilon^{\rm n}(k)$ yields remainder terms satisfying \eqref{eq:lowbd}.
\end{proof}

\subsection{Energetics}\label{subsec:energetics}
The filament evolution \eqref{eq:evolution} admits the following energy identity: 
\begin{equation}\label{eq:energyID}
\frac{1}{2}\p_t\int_\T \abs{\X_{ss}}^2\,ds = -\int_{\T}(\X_{sss}-\tau\X_s)_s\cdot\overline{\mc{L}}_\epsilon[(\X_{sss}-\tau\X_s)_s]\,ds\,,
\end{equation}
which may be obtained by dotting equation \eqref{eq:evolution} with $(\X_{sss}-\tau\X_s)_s$ and integrating over $\T$. Noting that the inextensibility constraint $\abs{\X_s}^2=1$ implies $\p_t\X_s\cdot\X_s=0$, we have  
\begin{equation}
\begin{aligned}
  \int_\T\frac{\p\X}{\p t}\cdot(\X_{sss}-\tau\X_s)_s\,ds &= -\int_\T\frac{\p\X_s}{\p t}\cdot(\X_{sss}-\tau\X_s)\,ds\\
  &= -\int_\T\frac{\p\X_s}{\p t}\cdot\X_{sss}\,ds = \frac{1}{2}\p_t\int_\T \abs{\X_{ss}}^2\,ds\,,
\end{aligned}
\end{equation}
yielding the left hand side of \eqref{eq:energyID}.

To see that the right hand side of \eqref{eq:energyID} is signed, we make note of the following proposition. 
\begin{proposition}\label{prop:positivity}
  Given $\bm{f}\in H^{-1/2}(\T)$, we have 
  \begin{equation}
  \begin{aligned}
    \int_\T \bm{f}\cdot\overline{\mc{L}}_\epsilon[\bm{f}]\,ds &\ge c\abs{\log\epsilon}\big(\norm{\P_{\X_s}\bm{f}}_{H^{-1/2}(\T)}^2 + \|\P_{\X_s}^\perp\bm{f}\|_{H^{-1/2}(\T)}^2\big)\\
    &\ge c\abs{\log\epsilon}\norm{\bm{f}}_{H^{-1/2}(\T)}^2 \,.
  \end{aligned}
  \end{equation}
\end{proposition}
We note that this lower bound is a bit pessimistic (in $\epsilon$) due to the differing $\epsilon$-scaling of the multipliers at low versus high wavenumbers (Lemma \ref{lem:multiplierbds}); however, it will be sufficient for establishing local and global well-posedness.

\begin{proof}
Using Definition \ref{def:Lepsbar} for $\overline{\mc{L}}_\epsilon$ along with Parseval's theorem, we may write 
\begin{equation}
\begin{aligned}
  \int_\T \bm{f}\cdot\overline{\mc{L}}_\epsilon[\bm{f}]\,ds &= \int_\T \bigg((\P_{\X_s}\bm{f})\cdot T_{m_\epsilon^{\rm t}}(\P_{\X_s}\bm{f}) + (\P_{\X_s}^\perp\bm{f})\cdot T_{m_\epsilon^{\rm n}}(\P_{\X_s}^\perp\bm{f}) \bigg)\, ds  \\
  &= \sum_{\abs{k}=0}^\infty \abs{\mc{F}[\P_{\X_s}\bm{f}](k)}^2m_\epsilon^{\rm t}(k) + \abs{\mc{F}[\P_{\X_s}^\perp\bm{f}](k)}^2m_\epsilon^{\rm n}(k) \\
  &\ge c\abs{\log\epsilon}\sum_{\abs{k}=0}^{\floor{\frac{1}{2\pi\epsilon}}} \abs{\mc{F}[\P_{\X_s}\bm{f}](k)}^2 + \abs{\mc{F}[\P_{\X_s}^\perp\bm{f}](k)}^2 \\
  &\quad + c\abs{\log\epsilon}\sum_{\abs{k}=\floor{\frac{1}{2\pi\epsilon}}+1}^{\infty} \abs{k}^{-1}\bigg(\abs{\mc{F}[\P_{\X_s}\bm{f}](k)}^2 + \abs{\mc{F}[\P_{\X_s}^\perp\bm{f}](k)}^2\bigg)\\
  &\ge c\abs{\log\epsilon}\big(\norm{\P_{\X_s}\bm{f}}_{H^{-1/2}(\T)}^2 + \|\P_{\X_s}^\perp\bm{f}\|_{H^{-1/2}(\T)}^2\big)\\
  &\ge c\abs{\log\epsilon}\norm{\bm{f}}_{H^{-1/2}(\T)}^2 \,,
\end{aligned}
\end{equation}
by the triangle inequality. Here in the third line we have used the behavior of the multipliers $m_\epsilon^{\rm t}$ and $m_\epsilon^{\rm n}$ from Lemma \ref{lem:multiplierbds}.
\end{proof}

Using Proposition \ref{prop:positivity}, the filament bending energy is non-increasing in time; in particular, 
\begin{equation}\label{eq:energy_ineq}
\begin{aligned}
\frac{1}{2}\p_t\int_\T \abs{\X_{ss}}^2\,ds &\le -c\abs{\log\epsilon}\norm{(\X_{sss}-\tau\X_s)_s}_{H^{-1/2}(\T)}^2 \\
&\le -c\abs{\log\epsilon}\norm{\X_{sss}-\tau\X_s}_{\dot H^{1/2}(\T)}^2\,.
\end{aligned}
\end{equation}
The energy inequality \eqref{eq:energy_ineq} will form the basis for our proof of global well-posedness.


\section{Global well-posedness}\label{sec:GWP}
This section is devoted to the proof of Theorem \ref{thm:GWP}. We begin with bounds for the tension $\tau(s,t)$ appearing as a Lagrange multiplier in \eqref{eq:evolution}. These bounds are leveraged to show local well-posedness in the energy space, which is then upgraded to global well-posedness using the energy inequality~\eqref{eq:energy_ineq}.

\subsection{Preliminaries}
Here we collect some identities and bounds that will be used throughout the following sections. We begin with the observation that, by differentiating the inextensibility constraint $\abs{\X_s}^2=1$, we obtain the following identities:
\begin{equation}\label{eq:inex_ids}
  \X_s\cdot\X_{ss}=0\,, \quad 
  \abs{\X_{ss}}^2 =-\X_s\cdot\X_{sss}\,, \quad
  3\X_{sss}\cdot\X_{ss} = -\X_s\cdot\X_{ssss}\,.
\end{equation}

Next, we make note of some useful interpolation inequalities on $\T$. Given $h:\T\to \R$ with $\int_\T h\,ds=0$, we may bound
\begin{equation}\label{eq:interps}
\begin{aligned}
  \norm{h}_{\dot H^{1/2}}&\le c\norm{\p_sh}_{L^2}^{1/2}\norm{h}_{L^2}^{1/2}\,,\\
  \norm{h_s}_{L^2}&\le c\norm{h_s}_{\dot H^{1/2}}^{2/3}\norm{h}_{L^2}^{1/3}\,,\\
  \norm{h}_{L^\infty} &\le c\norm{h_s}_{\dot H^{1/2}}^{1/3}\norm{h}_{L^2}^{2/3}\,.
\end{aligned}
\end{equation}
The proof of these may be found, for example, in \cite{brezis2018gagliardo}.

Finally, we will require the following proposition. 
\begin{proposition}\label{prop:H_minus_halfbd}
Given scalar-valued $h\in H^{-1/2}(\T)$ and vector-valued $\bv\in H^1(\T)$ with $\abs{\bv}=1$, we may bound
  \begin{equation}\label{eq:Hminushalf1}
  \norm{h}_{H^{-1/2}(\T)} \le c\norm{\bv}_{H^1(\T)}\norm{h \bv}_{H^{-1/2}(\T)}\,.
\end{equation}
In addition, we may bound
  \begin{equation}\label{eq:Hminushalf2}
  \norm{h\bm{v}}_{H^{-1/2}(\T)} \le c\norm{\bv}_{H^1(\T)}\norm{h}_{H^{-1/2}(\T)}\,.
\end{equation}
\end{proposition}

\begin{proof}
For any scalar-valued $\varphi\in H^{1/2}(\T)$ with $\norm{\varphi}_{H^{1/2}}=1$, we may calculate
\begin{equation}
\begin{aligned}
  \abs{\int_\T h\,\varphi\,ds} &= \abs{\int_\T (h\bv)\cdot(\bv\varphi)\,ds}
  \le c\norm{h\bv}_{H^{-1/2}}\norm{\bv\varphi}_{H^{1/2}}\\
  &\le c\norm{h\bv}_{H^{-1/2}}\norm{\bv}_{H^1}\norm{\varphi}_{H^{1/2}}
  = c\norm{h\bv}_{H^{-1/2}}\norm{\bv}_{H^1}\,.
\end{aligned}
\end{equation}
Taking the supremum over all such $\varphi$ yields \eqref{eq:Hminushalf1}.

Similarly, taking any vector-valued $\bm{\varphi}\in H^{1/2}(\T)$ with $\norm{\bm{\varphi}}_{H^{1/2}}=1$, we may calculate 
\begin{equation}
\begin{aligned}
  \abs{\int_\T h\bv\cdot\bm{\varphi} \,ds}
  &\le \norm{h}_{H^{-1/2}}\norm{\bv\cdot\bm{\varphi}}_{H^{1/2}}\\
  &\le c\norm{h}_{H^{-1/2}}\norm{\bv}_{H^1}\norm{\bm{\varphi}}_{H^{1/2}}
  = c\norm{h}_{H^{-1/2}}\norm{\bv}_{H^1}\,.
\end{aligned}
\end{equation}
Again taking the supremum over all such $\bm{\varphi}$ yields \eqref{eq:Hminushalf2}.
\end{proof}

\subsection{Tension estimates}\label{subsec:tension}
As in the analogous tension determination problems of \cite{mori2023well,kuo2023tension,garcia2025immersed,albritton2025rods,ohm2025slender}, we will use that $\p_t\X_s\cdot\X_s=0$ to obtain an equation and subsequent estimates for the tension $\tau$. Differentiating \eqref{eq:evolution} in $s$ and dotting with $\X_s$, we have that at each time $t$ (which we will omit from our notation throughout this section), $\tau$ satisfies the equation 
\begin{equation}\label{eq:TDP}
\big(\overline{\mc{L}}_\epsilon[(\tau\X_s)_s]\big)_s\cdot\X_s = \big(\overline{\mc{L}}_\epsilon[\X_{ssss}]\big)_s\cdot\X_s\,.
\end{equation}
We show the following.
\begin{lemma}[Tension bounds]\label{lem:TDP}
  Given $\epsilon>0$ and $\X\in H^{7/2}(\T)$, the tension equation \eqref{eq:TDP} admits a unique solution $\tau\in H^{1/2}(\T)$. Furthermore, $\tau$ satisfies the bound 
\begin{equation}\label{eq:tau_bound}
  \norm{\tau}_{H^{1/2}}\le c\,\epsilon^{-1} (1+\norm{\X_s}_{H^1}^4)\big(\norm{\X_{sss}}_{\dot H^{1/2}}^{3/4}\norm{\X_s}_{H^1}^{5/4}+ \norm{\X_{sss}}_{\dot H^{1/2}}^{2/3}\norm{\X_s}_{H^1}^{10/3}\big)\,.
\end{equation}
Given two nearby curves $\X_1,\X_2\in H^{7/2}(\T)$, the solutions $\tau_1$, $\tau_2$ to the corresponding tension determination problems \eqref{eq:TDP} satisfy 
\begin{equation}\label{eq:tau_lip}
\begin{aligned}
  \norm{\tau_1-\tau_2}_{H^{1/2}} 
  &\le 
  c\,\epsilon^{-1}(1+\norm{(\X_1)_s}_{H^1}^9+\norm{(\X_2)_s}_{H^1}^9)\times\\
  &\quad \times \bigg[\norm{(\X_1-\X_2)_s}_{H^1}\bigg(\norm{(\X_1)_{sss}}_{\dot H^{1/2}}^{2/3} + \norm{(\X_2)_{sss}}_{\dot H^{1/2}}^{2/3} + \norm{(\X_1)_{sss}}_{\dot H^{1/2}}^{3/4}\bigg) \\
  &\qquad + \norm{(\X_1-\X_2)_{sss}}_{\dot H^{1/2}}^{1/3}\norm{(\X_1-\X_2)_s}_{H^1}^{2/3}(\norm{(\X_1)_{sss}}_{\dot H^{1/2}}^{1/3}+\norm{(\X_2)_{sss}}_{\dot H^{1/2}}^{1/3}) \\
  &\qquad  + \norm{(\X_1-\X_2)_{sss}}_{\dot H^{1/2}}^{3/4}\norm{(\X_1-\X_2)_s}_{H^1}^{1/4} \bigg] \,.
\end{aligned}
\end{equation}
\end{lemma}

Note that the right hand side of \eqref{eq:tau_bound} involves less than one copy of the highest derivative of $\X$ (here we use $\norm{\X_{sss}}_{\dot H^{1/2}}^{3/4}$). In contrast, the analogous bound \eqref{eq:RFTtension1} in the resistive force theory setting involves a full copy of the highest derivative of the curve in the corresponding energy space, which makes the tension behavior there critical.

\begin{proof}
For $(\tau,\varphi)\in H^{1/2}(\T)\times H^{1/2}(\T)$, we may define the bilinear form
\begin{equation}\label{eq:Beps}
  \mc{B}_\epsilon(\tau,\varphi) := \int_\T \overline{\mc{L}}_\epsilon[(\tau\X_s)_s]\cdot (\varphi\X_s)_s\,ds\,,
\end{equation}
which, using Parseval's theorem, may be seen to satisfy 
\begin{equation}
\begin{aligned}
  \mc{B}_\epsilon(\tau,\varphi) 
  &= \int_\T \bigg( (\varphi_s\X_s)\cdot T_{m_\epsilon^{\rm t}}(\tau_s\X_s)  + (\varphi\X_{ss})\cdot T_{m_\epsilon^{\rm n}}(\tau\X_{ss}) \bigg)\,ds \\
  &\quad= \sum_{\abs{k}=0}^\infty \bigg( m_\epsilon^{\rm t}(k) \mc{F}[\varphi_s\X_s]\cdot \overline{\mc{F}[\tau_s\X_s]}  + m_\epsilon^{\rm n}(k) \mc{F}[\varphi\X_{ss}]\cdot\overline{\mc{F}[\tau\X_{ss}]} \bigg) \\
  &\quad= \int_\T \bigg( T_{m_\epsilon^{\rm t}}(\varphi_s\X_s)\cdot(\tau_s\X_s) + T_{m_\epsilon^{\rm n}}(\varphi\X_{ss})\cdot(\tau\X_{ss}) \bigg)\,ds \\
  &\quad = \mc{B}_\epsilon(\varphi,\tau)\,. 
\end{aligned}
\end{equation}
In addition, using the multiplier bounds of Lemma \ref{lem:multiplierbds} and Parseval's theorem, we may bound 
\begin{equation}
\begin{aligned}
  \abs{\mc{B}_\epsilon(\tau,\varphi) } &\le 
  c\abs{\log\epsilon}\sum_{\abs{k}=0}^{\floor{\frac{1}{2\pi\epsilon}}} \bigg(\abs{\mc{F}[\varphi_s\X_s]}\abs{\mc{F}[\tau_s\X_s]} + \abs{\mc{F}[\varphi\X_{ss}]}\abs{\mc{F}[\tau\X_{ss}]} \bigg) \\
  &\quad + c\,\epsilon^{-1}\sum_{\abs{k}=\floor{\frac{1}{2\pi\epsilon}}+1}^\infty \abs{k}^{-1}\bigg(\abs{\mc{F}[\varphi_s\X_s]}\abs{\mc{F}[\tau_s\X_s]} + \abs{\mc{F}[\varphi\X_{ss}]}\abs{\mc{F}[\tau\X_{ss}]} \bigg)\\
  &\le c\,\epsilon^{-1}\abs{\log\epsilon}\bigg(\norm{\varphi_s\X_s}_{H^{-1/2}}\norm{\tau_s\X_s}_{H^{-1/2}} + \norm{\varphi\X_{ss}}_{H^{-1/2}}\norm{\tau\X_{ss}}_{H^{-1/2}} \bigg) \\
  &\le c\,\epsilon^{-1}\abs{\log\epsilon}\norm{\X_s}_{H^1}^2\norm{\tau}_{H^{1/2}}\norm{\varphi}_{H^{1/2}}\,.
\end{aligned}
\end{equation}
Finally, we note that, by Proposition \ref{prop:positivity}, we may bound $\mc{B}_\epsilon(\tau,\tau)$ below as 
\begin{equation}\label{eq:Btautau_below}
  \mc{B}_\epsilon(\tau,\tau) 
  \ge c\abs{\log\epsilon}\big(\norm{\tau_s\X_s}_{H^{-1/2}(\T)}^2 + \norm{\tau\X_{ss}}_{H^{-1/2}(\T)}^2\big) \,.
\end{equation}

Using Proposition \ref{prop:H_minus_halfbd}, bound \eqref{eq:Hminushalf1}, we may estimate
\begin{equation}
  \norm{\tau_s}_{H^{-1/2}}\le c\norm{\X_s}_{H^1}\norm{\tau_s\X_s}_{H^{-1/2}}\,.
\end{equation}
To obtain a bound for the $L^2$ part of $\tau$, as in \cite{albritton2025rods}, we will need Fenchel's Theorem \cite{fenchel1951differential}: A closed $C^2$ curve $\X$ must satisfy
\begin{equation}\label{eq:fenchel}
  \int_\T\abs{\X_{ss}}^2\,ds \ge \int_\T \abs{\X_{ss}}\,ds\ge 2\pi\,.
\end{equation} 
Letting $\tau_0$ denote the mean of $\tau$, 
\begin{equation}
  \tau_0 = \int_\T\tau\,ds\,,
\end{equation}
on $\T$, we have 
\begin{equation}
\begin{aligned}
  \norm{\tau}_{L^2}^2&\le \norm{\tau-\tau_0}_{L^2}^2+\norm{\tau_0}_{L^2}^2 \le c(\norm{\tau}_{\dot H^{1/2}}^2 + \abs{\tau_0}^2) \\
  &\le c(\norm{\X_s}_{H^1}^2\norm{\tau_s\X_s}_{H^{-1/2}}^2+\abs{\tau_0}^2)\,.
\end{aligned}
\end{equation}
We may estimate $\tau_0$ as 
\begin{equation}
\begin{aligned}
  \abs{\tau_0} &= \abs{\tau_0\frac{\int_\T\abs{\X_{ss}}^2\,ds}{\norm{\X_{ss}}_{L^2}^2}} = \frac{1}{\norm{\X_{ss}}_{L^2}^2}\abs{\int_\T \tau_0 \X_{ss}\cdot\X_{ss}\,ds} \\
  &\le \frac{1}{2\pi}\abs{\int_\T \tau_0\X_{ss}\cdot\frac{\X_{ss}}{\norm{\X_{ss}}_{H^{1/2}}}\,ds} \le c\norm{\tau_0\X_{ss}}_{H^{-1/2}}\,.
\end{aligned}
\end{equation}
Noting that for any $\bv\in H^{1/2}(\T)$ with $\norm{\bv}_{H^{1/2}}=1$ we may calculate
\begin{equation}
\begin{aligned}
  \abs{\int_\T(\tau-\tau_0)\X_{ss}\cdot\bv\,ds} &\le \norm{\tau-\tau_0}_{L^2}\norm{\X_{ss}}_{L^2}\norm{\bv}_{L^\infty}
  \le c\norm{\tau}_{\dot H^{1/2}}\norm{\X_{ss}}_{L^2}\\
  &\le c\norm{\X_s}_{H^1}^2\norm{\tau_s\X_s}_{H^{-1/2}}\,,
\end{aligned}
\end{equation}
we may bound
\begin{equation}
\begin{aligned}
  \norm{\tau_0\X_{ss}}_{H^{-1/2}} &\le \norm{(\tau-\tau_0)\X_{ss}}_{H^{-1/2}}+\norm{\tau\X_{ss}}_{H^{-1/2}}\\
  &\le  c\norm{\X_s}_{H^1}^2\norm{\tau_s\X_s}_{H^{-1/2}}+\norm{\tau\X_{ss}}_{H^{-1/2}}\,.
\end{aligned}
\end{equation}
Thus the mean of $\tau$ satisfies
\begin{equation}
  \abs{\tau_0}\le c\big(\norm{\X_s}_{H^1}^2\norm{\tau_s\X_s}_{H^{-1/2}}+\norm{\tau\X_{ss}}_{H^{-1/2}}\big)\,,
\end{equation}
and altogether we may bound 
\begin{equation}\label{eq:tau_Hhalf}
\begin{aligned}
  \norm{\tau}_{H^{1/2}}&\le 
  c\big((\norm{\X_s}_{H^1}+\norm{\X_s}_{H^1}^2)\norm{\tau_s\X_s}_{H^{-1/2}}+\norm{\tau\X_{ss}}_{H^{-1/2}}\big) \\
  &\le c\abs{\log\epsilon}^{-1/2}(1+\norm{\X_s}_{H^1}^2)\mc{B}_\epsilon(\tau,\tau)^{1/2}\,.
\end{aligned}
\end{equation}
By the Lax-Milgram lemma, given $g\in H^{-1/2}(\T)$, there is then a unique $\tau\in H^{1/2}(\T)$ satisfying 
\begin{equation}
  \mc{B}_\epsilon(\tau,\varphi) = \int_\T g\,\varphi\,ds \qquad \forall\;\varphi\in H^{1/2}(\T)\,.
\end{equation}

We consider $g$ given by the right hand side of \eqref{eq:TDP}. In particular, upon multiplying both sides of \eqref{eq:TDP} by $\tau$ and integrating by parts over $\T$, we have 
\begin{equation}\label{eq:tau_eq}
  \mc{B}_\epsilon(\tau,\tau) = \int_\T\overline{\mc{L}}_\epsilon[\X_{ssss}]\cdot(\tau\X_s)_s\,ds\,.
\end{equation}
Using Parseval's theorem and the multiplier behavior of Lemma \ref{lem:multiplierbds}, the right hand side of \eqref{eq:tau_eq} may be estimated as 
\begin{equation}
\begin{aligned}
  &\abs{\int_\T\overline{\mc{L}}_\epsilon[\X_{ssss}]\cdot(\tau\X_s)_s\,ds} \le \abs{\int_\T \tau_s\X_s\cdot T_{m_\epsilon^{\rm t}}\P_{\X_s}\X_{ssss}\,ds} + \abs{\int_\T \tau\X_{ss}\cdot T_{m_\epsilon^{\rm n}}\P_{\X_s}^\perp\X_{ssss}\,ds}\\
  &\quad =\abs{\sum_{\abs{k}=0}^\infty m_\epsilon^{\rm t}(k)\mc{F}[\tau_s\X_s]\cdot \overline{\mc{F}[\P_{\X_s}\X_{ssss}]} }
   + \abs{\sum_{\abs{k}=0}^\infty m_\epsilon^{\rm n}(k)\mc{F}[\tau\X_{ss}]\cdot \overline{\mc{F}[\P_{\X_s}^\perp\X_{ssss}]} }\\
  &\quad \le c\abs{\log\epsilon}\sum_{\abs{k}=0}^{\floor{\frac{1}{2\pi\epsilon}}} \abs{\mc{F}[\tau_s\X_s]\cdot \overline{\mc{F}[\P_{\X_s}\X_{ssss}]} } + \abs{\mc{F}[\tau\X_{ss}]\cdot \overline{\mc{F}[\P_{\X_s}^\perp\X_{ssss}]} }\\
  &\qquad + c\,\epsilon^{-1}\sum_{\floor{\frac{1}{2\pi\epsilon}}+1}^\infty\abs{k}^{-1}\bigg( \abs{\mc{F}[\tau_s\X_s]\cdot \overline{\mc{F}[\P_{\X_s}\X_{ssss}]} } + \abs{\mc{F}[\tau\X_{ss}]\cdot \overline{\mc{F}[\P_{\X_s}^\perp\X_{ssss}]} }\bigg)\\
  &\quad \le c\,\epsilon^{-1}\abs{\log\epsilon}\bigg(\norm{\tau_s\X_s}_{H^{-1/2}}\norm{\P_{\X_s}\X_{ssss}}_{H^{-1/2}} + \norm{\tau\X_{ss}}_{H^{-1/8}}\norm{\P_{\X_s}^\perp\X_{ssss}}_{H^{-7/8}}\bigg)\,.
\end{aligned}
\end{equation}
Note that in the final inequality, for the term corresponding to the normal directions along the filament, we crucially make use of the fact that we can share the $\abs{k}^{-1}$ unequally between $\tau\X_{ss}$ and $\P_{\X_s}^\perp\X_{ssss}$. In particular, we can leverage the additional regularity of $\tau\X_{ss}$ (compared to $\tau_s\X_s$) and require less from the normal direction term $\P_{\X_s}^\perp\X_{ssss}$. This will be useful for closing a fixed point argument later on.

We may estimate
\begin{equation}\label{eq:Lp_bd}
  \norm{\tau\X_{ss}}_{H^{-1/8}} \le c\norm{\X_{ss}}_{L^2}\norm{\tau}_{L^p}
\end{equation}
for any $p>2$; in particular, we have
\begin{equation}\label{eq:Hminusm}
  \norm{\tau\X_{ss}}_{H^{-1/8}} \le c\norm{\X_{ss}}_{L^2}\norm{\tau}_{H^{1/2}}\,.
\end{equation}
Furthermore, by the triangle inequality and interpolation, we may bound 
\begin{equation}
\begin{aligned}
  \norm{\P_{\X_s}^\perp\X_{ssss}}_{H^{-7/8}} &\le \norm{\X_{ss}}_{\dot H^{9/8}} + \norm{\P_{\X_s}\X_{ssss}}_{H^{-7/8}} \\
  &\le c\norm{\X_{ss}}_{\dot H^{3/2}}^{3/4}\norm{\X_{ss}}_{L^2}^{1/4} + c\norm{\P_{\X_s}\X_{ssss}}_{H^{-1/2}}\,.
\end{aligned}
\end{equation}
%
So far, we may thus bound $\mc{B}_\epsilon(\tau,\tau)$ as
\begin{equation}
\begin{aligned}
  \mc{B}_\epsilon(\tau,\tau)\le c\,\epsilon^{-1}\abs{\log\epsilon} \norm{\tau}_{H^{1/2}}\bigg(\norm{\X_{ss}}_{L^2}\norm{\P_{\X_s}\X_{ssss}}_{H^{-1/2}} +\norm{\X_{sss}}_{\dot H^{1/2}}^{3/4}\norm{\X_{ss}}_{L^2}^{5/4} \bigg)\,.
\end{aligned}
\end{equation}
Using the identities \eqref{eq:inex_ids} coming from the inextensibility constraint, we may write the tangential part of $\X_{ssss}$ as
\begin{equation}
  \P_{\X_s}\X_{ssss} = -\frac{3}{2}\p_s\abs{\X_{ss}}^2\X_s\,.
\end{equation}
By Proposition \ref{prop:H_minus_halfbd}, bound \eqref{eq:Hminushalf2}, along with the interpolation inequalities \eqref{eq:interps}, we may then estimate $\P_{\X_s}\X_{ssss}$ as 
\begin{equation}
\begin{aligned}
  \norm{\P_{\X_s}\X_{ssss}}_{H^{-1/2}(\T)} &\le c\norm{\abs{\X_{ss}}^2}_{\dot H^{1/2}}\norm{\X_s}_{H^1}
  \le c\norm{\p_s\abs{\X_{ss}}^2}_{L^2}^{1/2}\norm{\abs{\X_{ss}}^2}_{L^2}^{1/2}\norm{\X_s}_{H^1}\\
  &\le c\norm{\X_{sss}}_{L^2}^{1/2}\norm{\X_{ss}}_{L^\infty}\norm{\X_s}_{H^1}^{3/2}
  \le c\norm{\X_{sss}}_{\dot H^{1/2}}^{2/3}\norm{\X_s}_{H^1}^{7/3}\,.
\end{aligned}
\end{equation}
We thus obtain 
\begin{equation}
\begin{aligned}
  \mc{B}_\epsilon(\tau,\tau)\le c\,\epsilon^{-1}\abs{\log\epsilon} \norm{\tau}_{H^{1/2}}\big(\norm{\X_{sss}}_{\dot H^{1/2}}^{3/4}\norm{\X_s}_{H^1}^{5/4}+ \norm{\X_{sss}}_{\dot H^{1/2}}^{2/3}\norm{\X_s}_{H^1}^{10/3}\big)\,.
\end{aligned}
\end{equation}
The lower bound \eqref{eq:tau_Hhalf} then yields
\begin{equation}
  \norm{\tau}_{H^{1/2}}^2\le c\,\epsilon^{-1} \norm{\tau}_{H^{1/2}}(1+\norm{\X_s}_{H^1}^4)\big(\norm{\X_{sss}}_{\dot H^{1/2}}^{3/4}\norm{\X_s}_{H^1}^{5/4}+ \norm{\X_{sss}}_{\dot H^{1/2}}^{2/3}\norm{\X_s}_{H^1}^{10/3}\big)\,,
\end{equation}
and an application of Young's inequality yields \eqref{eq:tau_bound}.

We next turn to the Lipschitz bound \eqref{eq:tau_lip}. 
Let $\X_1$, $\X_2$ denote two nearby curves\footnote{We say that two curves are ``nearby" if $\norm{\X_1-\X_2}_{H^2(\T)}=\delta$ for some $\delta$ sufficiently small.}, and let $\tau_1$, $\tau_2$ satisfy the corresponding tension determination problems \eqref{eq:TDP}. Letting $\mc{B}_{\epsilon,1}$, $\mc{B}_{\epsilon,2}$ denote the associated bilinear forms \eqref{eq:Beps}, we consider the weak form of the tension equations satisfied by $\tau_1$ and $\tau_2$. In particular, for any $\varphi\in H^{1/2}(\T)$, we have that $\tau_1$ and $\tau_2$ satisfy
\begin{equation}\label{eq:B1B2eqn}
\begin{aligned}
  &\mc{B}_{\epsilon,1}(\tau_1,\varphi) - \mc{B}_{\epsilon,2}(\tau_2,\varphi)\\
  &\qquad= \underbrace{\int_\T\overline{\mc{L}}_{\epsilon,1}[(\X_1)_{ssss}]\cdot(\varphi(\X_1)_s)_s\,ds
  - \int_\T\overline{\mc{L}}_{\epsilon,2}[(\X_2)_{ssss}]\cdot(\varphi(\X_2)_s)_s\,ds}_{\rm{RHS}}\,.
\end{aligned}
\end{equation}

Note that, using the form \eqref{eq:Beps} of $\mc{B}_\epsilon$, we may write
\begin{equation}
\begin{aligned}
   &\mc{B}_{\epsilon,1}(\tau_1,\varphi) - \mc{B}_{\epsilon,2}(\tau_2,\varphi)
   %
   %
   = \mc{B}_{\epsilon,2}(\tau_1-\tau_2,\varphi) + {\rm LHS}\,,\\
   &\quad {\rm LHS} =\int_\T\bigg( \varphi_s(\X_1-\X_2)_s\cdot T_{m_\epsilon^{\rm t}}((\tau_1)_s(\X_1)_s) + \varphi(\X_1-\X_2)_{ss}\cdot T_{m_\epsilon^{\rm n}}(\tau_1(\X_1)_{ss})\bigg)\,ds\\
  &\qquad\quad +\int_\T\bigg( \varphi_s(\X_2)_s\cdot T_{m_\epsilon^{\rm t}}((\tau_1)_s(\X_1-\X_2)_s) + \varphi(\X_2)_{ss}\cdot T_{m_\epsilon^{\rm n}}(\tau_1(\X_1-\X_2)_{ss})\bigg)\,ds \,.
\end{aligned}
\end{equation}
Following the estimates above for a single curve, we may bound the terms in LHS as
\begin{equation}
\begin{aligned}
  \abs{{\rm LHS}} &\le c\,\epsilon^{-1}\abs{\log\epsilon}\bigg( \norm{\varphi_s(\X_1-\X_2)_s}_{H^{-1/2}}\norm{(\tau_1)_s(\X_1)_s}_{H^{-1/2}} \\
  &\hspace{-.5cm}+ \norm{\varphi(\X_1-\X_2)_{ss}}_{H^{-1/2}}\norm{\tau_1(\X_1)_{ss}}_{H^{-1/2}}+\norm{\varphi_s(\X_2)_s}_{H^{-1/2}}\norm{(\tau_1)_s(\X_1-\X_2)_s}_{H^{-1/2}}\\
  &\qquad  + \norm{\varphi(\X_2)_{ss}}_{H^{-1/2}}\norm{\tau_1(\X_1-\X_2)_{ss}}_{H^{-1/2}}\bigg)\\
  &\le c\,\epsilon^{-1}\abs{\log\epsilon} \norm{\varphi}_{H^{1/2}}\norm{(\X_1-\X_2)_s}_{H^1}\norm{\tau_1}_{H^{1/2}}\big(\norm{(\X_1)_s}_{H^1}+\norm{(\X_2)_s}_{H^1}\big) \\
  &\le c\,\epsilon^{-1}\abs{\log\epsilon} \norm{\varphi}_{H^{1/2}}\norm{(\X_1-\X_2)_s}_{H^1}\big(\norm{(\X_1)_s}_{H^1}+\norm{(\X_2)_s}_{H^1}\big)\times\\
  &\qquad
  \times (1+\norm{(\X_1)_s}_{H^1}^{22/3})\big(\norm{(\X_1)_{sss}}_{\dot H^{1/2}}^{3/4}+ \norm{(\X_1)_{sss}}_{\dot H^{1/2}}^{2/3}\big)\,.
\end{aligned}
\end{equation}
Here in the final line, we have used the tension bound \eqref{eq:tau_bound} (with some simplification).


In addition, we may bound the terms RHS in \eqref{eq:B1B2eqn} as follows. Using the identity \eqref{eq:tang_ID} for the tangential direction along $\X_1$ and $\X_2$, we may write out RHS as
\begin{equation}
\begin{aligned}
  {\rm RHS} &= -\frac{3}{2}\int_\T\bigg( \varphi_s(\X_1-\X_2)_s\cdot T_{m_\epsilon^{\rm t}} \big(\p_s\abs{(\X_1)_{ss}}^2(\X_1)_s \big) \\
  &\quad + \varphi_s(\X_2)_s\cdot T_{m_\epsilon^{\rm t}} \big[\p_s(\abs{(\X_1)_{ss}}^2-\abs{(\X_2)_{ss}}^2)(\X_1)_s +\p_s\abs{(\X_2)_{ss}}^2(\X_1-\X_2)_s \big] \bigg)\,ds \\
  &\qquad +\int_\T\bigg( \varphi(\X_1-\X_2)_{ss}\cdot T_{m_\epsilon^{\rm n}} \big(\P_{(\X_1)_s}^\perp(\X_1)_{ssss}\big)\\
  &\quad + \varphi(\X_2)_{ss}\cdot T_{m_\epsilon^{\rm n}}\big( (\P_{(\X_1)_s}^\perp-\P_{(\X_2)_s}^\perp)(\X_1)_{ssss}+ \P_{(\X_2)_s}^\perp(\X_1-\X_2)_{ssss}\big)\bigg)\,ds\,.
\end{aligned}
\end{equation}
Then, again following the estimates above for a single curve, we may bound
\begin{equation}
\begin{aligned}
  \abs{{\rm RHS}} &\le c\,\epsilon^{-1}\abs{\log\epsilon}\bigg( \norm{\varphi_s(\X_1-\X_2)_s}_{H^{-1/2}}\norm{\p_s\abs{(\X_1)_{ss}}^2(\X_1)_s}_{H^{-1/2}} \\
  &\quad + \norm{\varphi_s(\X_2)_s}_{H^{-1/2}} \norm{\p_s\big((\X_1-\X_2)_{ss}\cdot(\X_1+\X_2)_{ss}\big)(\X_1)_s}_{H^{-1/2}}  \\
  &\quad + \norm{\varphi_s(\X_2)_s}_{H^{-1/2}} \norm{\p_s\abs{(\X_2)_{ss}}^2(\X_1-\X_2)_s }_{H^{-1/2}}  \\
  &\qquad + \norm{\varphi(\X_1-\X_2)_{ss}}_{H^{-1/8}}\norm{\P_{(\X_1)_s}^\perp(\X_1)_{ssss}}_{H^{-7/8}}\\
  &\quad + \norm{\varphi(\X_2)_{ss}}_{H^{-1/8}} \norm{(\P_{(\X_1)_s}^\perp-\P_{(\X_2)_s}^\perp)(\X_1)_{ssss}}_{H^{-7/8}} \\
  &\quad + \norm{\varphi(\X_2)_{ss}}_{H^{-1/8}} \norm{\P_{(\X_2)_s}^\perp(\X_1-\X_2)_{ssss}}_{H^{-7/8}}\bigg) \\
  %
  %
  %
  %
  %
  %
   &\le c\,\epsilon^{-1}\abs{\log\epsilon}\norm{\varphi}_{H^{1/2}}(1+\norm{(\X_1)_s}_{H^1}^{10/3}+\norm{(\X_2)_s}_{H^1}^{10/3})\bigg[\norm{(\X_1-\X_2)_s}_{H^1}\times\\
  &\quad\times\bigg(\norm{(\X_1)_{sss}}_{\dot H^{1/2}}^{2/3} + \norm{(\X_2)_{sss}}_{\dot H^{1/2}}^{2/3} + \norm{(\X_1)_{sss}}_{\dot H^{1/2}}^{3/4}\bigg) \\
  %
  &\qquad + \norm{(\X_1-\X_2)_{sss}}_{\dot H^{1/2}}^{1/3}\norm{(\X_1-\X_2)_s}_{H^1}^{2/3}(\norm{(\X_1)_{sss}}_{\dot H^{1/2}}^{1/3}+\norm{(\X_2)_{sss}}_{\dot H^{1/2}}^{1/3}) \\
  &\qquad  + \norm{(\X_1-\X_2)_{sss}}_{\dot H^{1/2}}^{3/4}\norm{(\X_1-\X_2)_s}_{H^1}^{1/4} \bigg]\,.
\end{aligned}
\end{equation}

Altogether, taking $\varphi=\tau_1-\tau_2$ in \eqref{eq:B1B2eqn} and the subsequent bounds, we may estimate 
\begin{equation}
\begin{aligned}
  \mc{B}_{\epsilon,2}&(\tau_1-\tau_2,\tau_1-\tau_2) \le \abs{{\rm LHS}}+\abs{{\rm RHS}}\\
  &\le 
  c\,\epsilon^{-1}\abs{\log\epsilon}\norm{\tau_1-\tau_2}_{H^{1/2}}(1+\norm{(\X_1)_s}_{H^1}^{10/3}+\norm{(\X_2)_s}_{H^1}^{10/3})\bigg[\norm{(\X_1-\X_2)_s}_{H^1}\times\\
  &\quad\times\bigg(\norm{(\X_1)_{sss}}_{\dot H^{1/2}}^{2/3} + \norm{(\X_2)_{sss}}_{\dot H^{1/2}}^{2/3} + \norm{(\X_1)_{sss}}_{\dot H^{1/2}}^{3/4}\bigg)(1+\norm{(\X_1)_s}_{H^1}^5) \\
  &\qquad + \norm{(\X_1-\X_2)_{sss}}_{\dot H^{1/2}}^{1/3}\norm{(\X_1-\X_2)_s}_{H^1}^{2/3}(\norm{(\X_1)_{sss}}_{\dot H^{1/2}}^{1/3}+\norm{(\X_2)_{sss}}_{\dot H^{1/2}}^{1/3}) \\
  &\qquad  + \norm{(\X_1-\X_2)_{sss}}_{\dot H^{1/2}}^{3/4}\norm{(\X_1-\X_2)_s}_{H^1}^{1/4} \bigg] \,.
\end{aligned}
\end{equation}
Applying the lower bound \eqref{eq:tau_Hhalf} and Young's inequality yields the bound \eqref{eq:tau_lip}. 
\end{proof}

\subsection{Main operator decomposition and linear bounds}\label{subsec:mainOpDecomp}
Since the operator $\overline{\mc{L}}_\epsilon\p_s^4$ used for the evolution \eqref{eq:evolution} involves a nonlinear dependence on the curve $\X$, we begin by extracting the principal linear part for use in the local existence theory. Using the inextensibility constraint for $\X$, we show the following. 
\begin{lemma}[Main operator decompositon]\label{lem:op_decomp}
  Given $\X\in H^{7/2}(\T)$ satisfying $\abs{\X_s}=1$ and $\overline{\mc{L}}_\epsilon$ as in Definition \ref{def:Lepsbar}, the operator $\overline{\mc{L}}_\epsilon[\p_s^4\cdot]$ applied to $\X$ may be decomposed as 
  \begin{equation}
  \overline{\mc{L}}_\epsilon[\p_s^4\X] = T_{m_\epsilon^{\rm n}}(\p_s^4\X) + \mc{R}[\X]\,,
  \end{equation}
  where $m_\epsilon^{\rm n}$ is the normal direction multiplier \eqref{eq:multN}--\eqref{eq:multN0}, and the remainder term $\mc{R}[\X]$ satisfies
  \begin{equation}\label{eq:RX_bound}
  \norm{\mc{R}[\X]}_{H^{1/2}(\T)} \le c\,\epsilon^{-3/2}\big(\norm{\X_{sss}}_{\dot H^{1/2}}^{2/3}\norm{\X_s}_{H^1}^{4/3} + \norm{\X_s}_{H^1}^2 \big)(1+\norm{\X_s}_{H^1}^2)\,.
  \end{equation}
  In addition, given two nearby curves $\X_1,\X_2\in H^{7/2}(\T)$ both satisfying $\abs{(\X_j)_s}=1$, we may bound the difference between the corresponding remainder terms as
  \begin{equation}\label{eq:RX_lip}
  \begin{aligned}
  &\norm{\mc{R}[\X_1]-\mc{R}[\X_2]}_{H^{1/2}(\T)}\\
  &\quad \le c\,\epsilon^{-3/2}(1+\norm{(\X_1)_s}_{H^1}^{10/3}+\norm{(\X_2)_s}_{H^1}^{10/3})\times\\
  &\qquad\quad \times \bigg[\norm{(\X_1-\X_2)_s}_{H^1}\bigg(\norm{(\X_1)_{sss}}_{\dot H^{1/2}}^{2/3}+\norm{(\X_2)_{sss}}_{\dot H^{1/2}}^{2/3}+\norm{(\X_1)_s}_{H^1}^{2/3}+1\bigg)\\
  &\qquad\qquad + \norm{(\X_1-\X_2)_s}_{H^1}^{2/3}\norm{(\X_1-\X_2)_{sss}}_{\dot H^{1/2}}^{1/3}(\norm{(\X_1)_{sss}}_{\dot H^{1/2}}^{1/3}+\norm{(\X_2)_{sss}}_{\dot H^{1/2}}^{1/3} )\\
  &\qquad\qquad+\norm{(\X_1-\X_2)_{sss}}_{\dot H^{1/2}}^{2/3}\norm{(\X_1-\X_2)_s}_{H^1}^{1/3}\bigg]\,.
  \end{aligned}
  \end{equation}
\end{lemma}

\begin{proof}
We begin by writing
\begin{equation}
\begin{aligned}
  \overline{\mc{L}}_\epsilon[\p_s^4\X] &= T_{m_\epsilon^{\rm n}}(\p_s^4\X) + \mc{R}[\X]\,,\\
  \mc{R}[\X] &= \underbrace{-\big(\X_s\cdot T_{m_\epsilon^{\rm n}}(\p_s^4\X)\big)\X_s}_{\mc{R}_a} + \underbrace{\P_{\X_s}T_{m_\epsilon^{\rm t}}(\P_{\X_s}\p_s^4\X)}_{\mc{R}_b} - \underbrace{\P_{\X_s}^\perp T_{m_\epsilon^{\rm n}}(\P_{\X_s}\p_s^4\X) }_{\mc{R}_c}\,.
\end{aligned}
\end{equation}

The terms $\mc{R}_b$ and $\mc{R}_c$ are the simplest to estimate. In particular, using the identities \eqref{eq:inex_ids} coming from the inextensibility constraint, we may write
\begin{equation}\label{eq:tang_ID}
  \P_{\X_s}\p_s^4\X = -3(\X_{ss}\cdot\X_{sss})\X_s 
\end{equation}
so that, using Parseval's theorem, we may bound
\begin{equation}
\begin{aligned}
  \norm{T_{m_\epsilon^{\rm t}}(\P_{\X_s}\p_s^4\X)}_{H^{1/2}}^2
  &= \sum_{\abs{k}=0}^\infty m_\epsilon^{\rm t}(k)^2(1+\abs{k}^2)^{1/2}\abs{\mc{F}[\P_{\X_s}\p_s^4\X]}^2 \\
  &\le c\abs{\log\epsilon}^2\sum_{\abs{k}=0}^{\floor{\frac{1}{2\pi\epsilon}}}(1+\abs{k}^2)^{1/2}\abs{\mc{F}[\P_{\X_s}\p_s^4\X]}^2\\
  &\quad + c\,\epsilon^{-2}\sum_{\abs{k}=\floor{\frac{1}{2\pi\epsilon}}+1}^\infty \abs{k}^{-2}(1+\abs{k}^2)^{1/2}\abs{\mc{F}[\P_{\X_s}\p_s^4\X]}^2\\
  &\le c\,\epsilon^{-2}\abs{\log\epsilon}^2\norm{(\X_{ss}\cdot\X_{sss})\X_s}_{H^{-1/2}}^2\\
  &\le c\,\epsilon^{-2}\abs{\log\epsilon}^2\norm{\X_s}_{H^1}^2\norm{\p_s\abs{\X_{ss}}^2}_{H^{-1/2}}^2\,.
\end{aligned}
\end{equation}
Here we have used the multiplier bounds of Lemma \ref{lem:multiplierbds} in the first inequality, and the second bound of Proposition \ref{prop:H_minus_halfbd} in the final inequality. The interpolation inequalities \eqref{eq:interps} then yield 
\begin{equation}
\begin{aligned}
  \norm{T_{m_\epsilon^{\rm t}}(\P_{\X_s}\p_s^4\X)}_{H^{1/2}} &\le c\,\epsilon^{-1}\abs{\log\epsilon}\norm{\X_s}_{H^1}\norm{\abs{\X_{ss}}^2}_{\dot H^{1/2}}\\
  &\le c\,\epsilon^{-1}\abs{\log\epsilon}\norm{\X_s}_{H^1}^{3/2}\norm{\X_{sss}}_{L^2}^{1/2}\norm{\X_{ss}}_{L^\infty}\\
  &\le c\,\epsilon^{-1}\abs{\log\epsilon}\norm{\X_s}_{H^1}^{7/3}\norm{\X_{sss}}_{\dot H^{1/2}}^{2/3}\,.
\end{aligned}
\end{equation}
We may thus estimate $\mc{R}_b$ as 
\begin{equation}
\begin{aligned}
   \norm{\mc{R}_b}_{H^{1/2}}=\norm{\P_{\X_s}T_{m_\epsilon^{\rm t}}(\P_{\X_s}\p_s^4\X)}_{H^{1/2}}&\le c\norm{\X_s}_{H^1}\norm{T_{m_\epsilon^{\rm t}}(\P_{\X_s}\p_s^4\X)}_{H^{1/2}} \\
   &\le c\,\epsilon^{-1}\abs{\log\epsilon}\norm{\X_s}_{H^1}^{10/3}\norm{\X_{sss}}_{\dot H^{1/2}}^{2/3}\,.
\end{aligned}
\end{equation}
A similar calculation for the normal direction term $\mc{R}_c$ yields 
\begin{equation}
  \norm{\mc{R}_c}_{H^{1/2}}=\norm{\P_{\X_s}^\perp T_{m_\epsilon^{\rm n}}(\P_{\X_s}\p_s^4\X)}_{H^{1/2}} \le c\,\epsilon^{-1}\abs{\log\epsilon}\norm{\X_s}_{H^1}^{10/3}\norm{\X_{sss}}_{\dot H^{1/2}}^{2/3}\,.
\end{equation}

To estimate $\mc{R}_a$, we first use that $\p_s$ commutes with $T_{m_\epsilon^{\rm n}}$ to write 
\begin{equation}\label{eq:Tm_rewrite}
  \X_s\cdot T_{m_\epsilon^{\rm n}}(\p_s^4\X) = \X_s\cdot \p_s^3(T_{m_\epsilon^{\rm n}}\X_s)\,.
\end{equation}
We then consider the difference
\begin{equation}
  \mc{R}_{a,\rm diff}= \p_s^3(T_{m_\epsilon^{\rm n}}\X_s) - c_{\rm n}\epsilon^{-1}\p_s^2\X_s \,,
\end{equation}
where $c_{\rm n}$ is the constant from Lemma \ref{lem:highk_refinement}. Then, using the bounds of Lemma \ref{lem:highk_refinement}, we may estimate 
\begin{equation}\label{eq:R1diff}
\begin{aligned}
\norm{\mc{R}_{a,\rm diff}}_{H^{1/2}}^2 &=
  \sum_{\abs{k}=0}^\infty(1+\abs{k}^2)^{1/2}\big(\abs{k}m_\epsilon^{\rm n}(k) -c_{\rm n}\epsilon^{-1} \big)^2\abs{k}^4\abs{\mc{F}[\X_s]}^2\\
  &\le  \sum_{\abs{k}=0}^\infty(1+\abs{k}^2)^{1/2}(c_{\rm n}\epsilon^{-1})^2 (\epsilon\abs{k}+1)^{-2}\abs{k}^4\abs{\mc{F}[\X_s]}^2\\
  &\le  c\,\epsilon^{-3}\sum_{\abs{k}=0}^\infty(1+\abs{k}^2)^{1/2}\abs{k}^3\abs{\mc{F}[\X_s]}^2\\
  &\le  c\,\epsilon^{-3}\norm{\X_{ss}}_{H^1}^2\,.
\end{aligned}
\end{equation}
Using \eqref{eq:Tm_rewrite}, we may write
\begin{equation}
  \X_s\cdot T_{m_\epsilon^{\rm n}}(\p_s^4\X) = c_{\rm n}\epsilon^{-1}\X_s\cdot \X_{sss} + \X_s\cdot\mc{R}_{a,\rm diff}
  = - c_{\rm n}\epsilon^{-1}\abs{\X_{ss}}^2+ \X_s\cdot\mc{R}_{a,\rm diff}\,,
\end{equation}
so that, by \eqref{eq:R1diff} and the interpolation inequalities \eqref{eq:interps}, we may estimate the final remainder term as 
\begin{equation}
\begin{aligned}
  \norm{\mc{R}_a}_{H^{1/2}} &\le c_{\rm n}\epsilon^{-1}\norm{\X_s}_{H^1}\norm{\abs{\X_{ss}}^2}_{H^{1/2}} + \norm{\X_s}_{H^1}\norm{\mc{R}_{a,\rm diff}}_{H^{1/2}} \\
  &\le c\,\epsilon^{-1}\big(\norm{\X_{sss}}_{\dot H^{1/2}}^{2/3}\norm{\X_s}_{H^1}^{7/3}+ \norm{\X_{sss}}_{\dot H^{1/2}}^{1/3}\norm{\X_s}_{H^1}^{8/3}\big) \\
  &\quad + c\,\epsilon^{-3/2}\big(\norm{\X_{sss}}_{\dot H^{1/2}}^{2/3}\norm{\X_s}_{H^1}^{4/3} + \norm{\X_s}_{H^1}^2 \big)\\
  &\le c\,\epsilon^{-3/2}\big(\norm{\X_{sss}}_{\dot H^{1/2}}^{2/3}\norm{\X_s}_{H^1}^{4/3} + \norm{\X_s}_{H^1}^2 \big)(1+\norm{\X_s}_{H^1})\,.
\end{aligned}
\end{equation}
Combining the bounds for $\mc{R}_a$, $\mc{R}_b$, and $\mc{R}_c$, we obtain the estimate \eqref{eq:RX_bound}.

To show the Lipschitz bound \eqref{eq:RX_lip}, we again begin with $\mc{R}_b$ and $\mc{R}_c$. Using \eqref{eq:tang_ID}, we have
\begin{equation}
\begin{aligned}
  &\norm{\mc{R}_b[\X_1]-\mc{R}_b[\X_2]}_{H^{1/2}}\\
  &\quad \le  
  c\bigg[\norm{(\X_1-\X_2)_s}_{H^1}(\norm{(\X_1)_s}_{H^1}+\norm{(\X_2)_s}_{H^1})\norm{T_{m_\epsilon^{\rm t}}\big(\p_s\abs{(\X_1)_{ss}}^2(\X_1)_s\big)}_{H^{1/2}}\\
  &\qquad\quad +\norm{(\X_2)_s}_{H^1} \bigg(\norm{T_{m_\epsilon^{\rm t}}\big(\p_s\abs{(\X_2)_{ss}}^2(\X_1-\X_2)_s\big)}_{H^{1/2}} \\ 
  &\hspace{5.5cm}+\norm{T_{m_\epsilon^{\rm t}}\big(\p_s(\abs{(\X_1)_{ss}}^2-\abs{(\X_2)_{ss}}^2)(\X_1)_s\big)}_{H^{1/2}}\bigg) \bigg] \\
  &\quad \le c\,\epsilon^{-1}\abs{\log\epsilon}\bigg[\norm{(\X_1-\X_2)_s}_{H^1}(1+\norm{(\X_1)_s}_{H^1}^2+\norm{(\X_2)_s}_{H^1}^2)\times\\
  &\hspace{4cm} \times \bigg(\norm{\abs{(\X_1)_{ss}}^2}_{\dot H^{1/2}}+\norm{\abs{(\X_2)_{ss}}^2}_{\dot H^{1/2}}\bigg)\\
  &\hspace{3cm} + \norm{(\X_1)_s}_{H^1}\norm{((\X_1)_{ss}-(\X_2)_{ss})\cdot((\X_1)_{ss}+(\X_2)_{ss})}_{\dot H^{1/2}} \bigg] \\
  &\quad \le c\,\epsilon^{-1}\abs{\log\epsilon}(1+\norm{(\X_1)_s}_{H^1}^{10/3}+\norm{(\X_2)_s}_{H^1}^{10/3})\times\\
  &\qquad\quad \times \bigg[\norm{(\X_1-\X_2)_s}_{H^1}\bigg(\norm{(\X_1)_{sss}}_{\dot H^{1/2}}^{2/3}+\norm{(\X_2)_{sss}}_{\dot H^{1/2}}^{2/3}\bigg)\\
  &\qquad\qquad + \norm{(\X_1-\X_2)_s}_{H^1}^{2/3}\norm{(\X_1-\X_2)_{sss}}_{\dot H^{1/2}}^{1/3}(\norm{(\X_1)_{sss}}_{\dot H^{1/2}}^{1/3}+\norm{(\X_2)_{sss}}_{\dot H^{1/2}}^{1/3} )\bigg] \,.
\end{aligned}
\end{equation}
An identical bound holds for $\norm{\mc{R}_c[\X_1]-\mc{R}_c[\X_2]}_{H^{1/2}}$ by an analogous calculation.

For $\mc{R}_a$, as above, we first note the linear estimate
\begin{equation}
  \norm{\mc{R}_{a,{\rm diff}}[\X_1-\X_2]}_{H^{1/2}}\le c\,\epsilon^{-3/2}\norm{(\X_1-\X_2)_{ss}}_{H^1}\,.
\end{equation}
We then have
\begin{equation}
\begin{aligned}
  &\norm{\mc{R}_a[\X_1]-\mc{R}_a[\X_2]}_{H^{1/2}}
  \le c\bigg[\norm{(\X_1-\X_2)_s}_{H^1}\bigg(\epsilon^{-1}\norm{\abs{(\X_1)_{ss}}^2}_{H^{1/2}} \\
  &\hspace{5cm}
  +(\norm{(\X_1)_s}_{H^1}+\norm{(\X_2)_s}_{H^1})\norm{\mc{R}_{a,{\rm diff}}[\X_1]}_{H^{1/2}}\bigg)\\
  &\quad + (1+\norm{(\X_2)_s}_{H^1}^2)\bigg(\epsilon^{-1}\norm{(\X_1-\X_2)_{ss}\cdot(\X_1+\X_2)_{ss}}_{H^{1/2}} +\norm{\mc{R}_{a,{\rm diff}}[\X_1-\X_2]}_{H^{1/2}}\bigg)\bigg]\\
  &\le c\,\epsilon^{-3/2}(1+\norm{(\X_1)_s}_{H^1}^3+\norm{(\X_2)_s}_{H^1}^3)\bigg[\norm{(\X_1-\X_2)_{sss}}_{\dot H^{1/2}}^{2/3}\norm{(\X_1-\X_2)_s}_{H^1}^{1/3}\\
  &\qquad +\norm{(\X_1-\X_2)_s}_{H^1}\bigg(\norm{(\X_1)_{sss}}_{\dot H^{1/2}}^{2/3}+\norm{(\X_2)_{sss}}_{\dot H^{1/2}}^{2/3} +\norm{(\X_1)_s}_{H^1}^{2/3}+1\bigg) \bigg]\,.
\end{aligned}
\end{equation}
Again combining the bounds for $\mc{R}_a$, $\mc{R}_b$, and $\mc{R}_c$, we obtain the Lipschitz estimate \eqref{eq:RX_lip}. 
\end{proof}

Given the decomposition of Lemma \ref{lem:op_decomp}, we consider the curve evolution \eqref{eq:evolution} as
\begin{equation}\label{eq:new_evol}
  \frac{\p\X}{\p t} = -T_{m_\epsilon^{\rm n}}\p_s^4\X - \mc{R}[\X] + \overline{\mc{L}}_\epsilon[(\tau\X_s)_s]\,, \quad \abs{\X_s}^2=1\,,
\end{equation}
where the principal linear part of the evolution is given by $T_{m_\epsilon^{\rm n}}\p_s^4\sim \epsilon^{-1}\p_s^3$, and the tension term and $\mc{R}$ will be treated as remainders. We will make use of some general estimates for the linear part of \eqref{eq:new_evol}. 
Given a forcing function $f\in L^2_tH^{1/2}_s$, we consider the following linear equation on $\T$:
\begin{equation}\label{eq:u_eqn}
\begin{aligned}
  \p_t u &= -T_{m_\epsilon^{\rm n}}\p_s^4u + f \,, 
  \qquad u\big|_{t=0} = u^{\rm in}\,.
\end{aligned}
\end{equation}
\begin{lemma}[Linear estimates]\label{lem:linear}
For $T>0$ and any $u^{\rm in}\in H^2(\T)$, there exists a unique weak solution $u\in L^\infty_tH^2_s\cap L^2_tH^{7/2}_s((0,T)\times\T)$ to the linear equation \eqref{eq:u_eqn} satisfying 
\begin{equation}
  \norm{u-u^{\rm in}}_{H^2_s}\to 0 \quad \text{as }t\to 0^+\,.
\end{equation}
In addition, $u$ is continuous on $[0,T]$ with values in $H^2_s$ and satisfies the estimate 
\begin{equation}\label{eq:ubd}
\begin{aligned}
   \norm{u}_{L^\infty_tH^2_s}+ \abs{\log\epsilon}^{1/2}\norm{\p_s^4u}_{L^2_tH^{-1/2}_s}
   &\le c\big(\norm{u^{\rm in}}_{H^2_s} + (T^{1/2}+\abs{\log\epsilon}^{-1/2})\norm{f}_{L^2_tH^{1/2}_s} \big) 
\end{aligned}
\end{equation}
on $(0,T)\times\T$.
\end{lemma}

\begin{proof}
The mapping properties of the solution operator for $u$ are immediate from the explicit solution formula using the Fourier multiplier behavior of Lemma \ref{lem:multiplierbds}, so it remains to show the bound \eqref{eq:ubd}.
We begin with a bound for $\p_s^2u$. Upon differentiating \eqref{eq:u_eqn} in $s$, we have that $\p_s^2u$ satisfies
\begin{equation}
\begin{aligned}
  \p_t (\p_s^2u) &= -T_{m_\epsilon^{\rm n}}\p_s^4(\p_s^2u) + \p_s^2f \\
  \p_s^2u\big|_{t=0} &= \p_s^2u^{\rm in}\,.
\end{aligned}
\end{equation}
Multiplying by $\p_s^2u$ and integrating by parts, we have that $\p_s^2u$ satisfies the energy identity
\begin{equation}\label{eq:u_energy}
   \frac{1}{2}\p_t\int_\T (\p_s^2u)^2\,ds+\int_\T\p_s^4u\,T_{m_\epsilon^{\rm n}}\p_s^4u\,ds = \int_\T \p_s^4uf\,ds\,,
\end{equation}
where we have used that $\p_s$ commutes with $T_{m_\epsilon^{\rm n}}$. We may bound the term involving $T_{m_\epsilon^{\rm n}}$ as 
\begin{equation}
\begin{aligned}
  \int_\T\p_s^4u\,T_{m_\epsilon^{\rm n}}\p_s^4u\,ds &= 
  \sum_{\abs{k}=0}^\infty \abs{\mc{F}[u](k)}^2\abs{k}^8m_\epsilon^{\rm n}(k) \\
  &\ge c\abs{\log\epsilon}\sum_{\abs{k}=0}^{\floor{\frac{1}{2\pi\epsilon}}} \abs{\mc{F}[u](k)}^2\abs{k}^8 + c\,\epsilon^{-1}\sum_{\abs{k}=\floor{\frac{1}{2\pi\epsilon}}+1}^{\infty} \abs{k}^7\abs{\mc{F}[u](k)}^2 \\
  &\ge c\abs{\log\epsilon}\norm{\p_s^4u}_{H^{-1/2}(\T)}^2\,.
\end{aligned}
\end{equation}
Integrating \eqref{eq:u_energy} in $t$, we then obtain the bound
\begin{equation}
\begin{aligned}
   &\norm{\p_s^2u}_{L^\infty_tL^2_s}^2+ \abs{\log\epsilon}\norm{\p_s^4u}_{L^2_tH^{-1/2}_s}^2
   \le c\big(\norm{\p_s^2u^{\rm in}}_{L^\infty_tL^2_s}^2 +\norm{\p_s^4u}_{L^2_tH^{-1/2}_s}\norm{f}_{L^2_tH^{1/2}_s} \big)\,.
\end{aligned}
\end{equation}
An application of Young's inequality yields \eqref{eq:ubd} with $L^\infty_t\dot H^2_s$ in place of $L^\infty_tH^2_s$.

It remains to bound the mean of $u$, which we denote by $u_0$. We have that $u_0$ satisfies
\begin{equation}
  \p_tu_0 = \int_\T f\,ds\,, \qquad u_0\big|_{t=0}=u_0^{\rm in}\,,
\end{equation}
so that
\begin{equation}
  \norm{u_0}_{L^\infty_t} \le \abs{u_0^{\rm in}} + T^{1/2}\norm{f}_{L^2_tL^2_s}\,.
\end{equation}
\end{proof}

\subsection{Local well-posedness}\label{subsec:LWP}
Here we establish local-in-time well-posedness for the evolution \eqref{eq:evolution}, rewritten as \eqref{eq:new_evol}.
\begin{lemma}[Local well-posedness]\label{lem:LWP}
Given $\epsilon>0$ and $\X^{\rm in}\in H^2_s(\T)$, there exists $T>0$ and a unique solution $\X\in C([0,T];H^2_s(\T))\cap L^2_t\dot H^{7/2}_s$ to \eqref{eq:new_evol} with initial condition $\X^{\rm in}$.\\
Given $T_{\rm f}>0$, if $\X\in C([0,T];H^2_s(\T))\cap L^2_t\dot H^{7/2}_s$ is a solution for all $T\in (0,T_{\rm f})$ and 
\begin{equation}\label{eq:continuation}
  \norm{\X}_{L^\infty_tH^2_s\cap L^2_t\dot H^{7/2}_s([0,T]\times \T)}<+\infty\,, 
\end{equation} 
then the solution may be extended continuously in time past $T_{\rm f}$, i.e. there exists $\delta>0$ such that $\X\in C([0,T_{\rm f}+\delta];H^2_s(\T))\cap L^2_t\dot H^{7/2}_s$. 
\end{lemma}

\begin{proof}
For $T>0$, let $\Phi^T(\X)$ denote the time-$T$ solution map for the evolution \eqref{eq:new_evol}. We proceed by a fixed point argument in the space $\mc{Y}_0\cap\mc{Y}_1$, where 
  \begin{equation}
  \begin{aligned}
    \mc{Y}_0(T) &= C([0,T];H^2_s)\,, \qquad \norm{\cdot}_{\mc{Y}_0} = \sup_{t\in[0,T]}\norm{\cdot}_{H^2_s(\T)}\\
    \mc{Y}_1(T) &= L^2_t\dot H^{7/2}_s\,, \qquad \norm{\cdot}_{\mc{Y}_1} = \bigg(\int_0^T\norm{\cdot}_{\dot H^{7/2}_s(\T)}^2\,dt'\bigg)^{1/2}\,.
  \end{aligned}
  \end{equation}
For some $M_0,M_1>0$, we consider $\X$ belonging to the closed ball
  \begin{equation}
    {\bm B}_{0,1} = \{ \X\in \mc{Y}_0(T)\cap\mc{Y}_1(T) \,:\, \norm{\X}_{\mc{Y}_0}\le M_0\,, \; \norm{\X}_{\mc{Y}_1}\le M_1 \}\,, 
  \end{equation}
and show that for some choice of $T$, the solution map $\Phi^T(\cdot)$ maps ${\bm B}_{0,1}$ to itself.

We begin with the tension terms. First, we may note the following general bounds. Using Parseval's theorem, the form \eqref{eq:Tmh} of $T_{m_\epsilon^{\rm t}}$, and the multiplier bounds of Lemma \ref{lem:multiplierbds}, we have that for any $\bm{f}\in H^{-1/2}(\T)$,
 \begin{equation}
 \begin{aligned}
   \norm{T_{m_\epsilon^{\rm t}}(\P_{\X_s}\bm{f})}_{H^{1/2}}^2 &= \sum_{\abs{k}=0}^\infty m_\epsilon^{\rm t}(k)^2(1+\abs{k}^2)^{1/2}\abs{\mc{F}[\P_{\X_s}\bm{f}]}^2\\
   &\le c\abs{\log\epsilon}^2\sum_{\abs{k}=0}^{\floor{\frac{1}{2\pi\epsilon}}} (1+\abs{k}^2)^{1/2}\abs{\mc{F}[\P_{\X_s}\bm{f}]}^2 \\
   &\quad + c\,\epsilon^{-2}\sum_{\abs{k}=\floor{\frac{1}{2\pi\epsilon}}+1}^\infty \abs{k}^{-2}(1+\abs{k}^2)^{1/2}\abs{\mc{F}[\P_{\X_s}\bm{f}]}^2\\
   &\le c\,\epsilon^{-2}\abs{\log\epsilon}^2\norm{\P_{\X_s}\bm{f}}_{H^{-1/2}}^2\,.
 \end{aligned}
 \end{equation}
 Similarly,
  \begin{equation}
   \norm{T_{m_\epsilon^{\rm n}}(\P_{\X_s}^\perp\bm{f})}_{H^{1/2}}^2 
   \le c\,\epsilon^{-2}\abs{\log\epsilon}^2\norm{\P_{\X_s}^\perp\bm{f}}_{H^{-1/2}}^2\,.
 \end{equation}
We may thus estimate
 \begin{equation}
\begin{aligned}
   \norm{\overline{\mc{L}}_\epsilon[\bm{f}]}_{H^{1/2}}
   &\le c(1+\norm{\X_s}_{H^1})\big(\norm{T_{m_\epsilon^{\rm t}}(\P_{\X_s}\bm{f})}_{H^{1/2}} +\norm{T_{m_\epsilon^{\rm n}}(\P_{\X_s}^\perp\bm{f})}_{H^{1/2}} \big)\\
   &\le c\,\epsilon^{-1}\abs{\log\epsilon}(1+\norm{\X_s}_{H^1})\big(\norm{\P_{\X_s}\bm{f}}_{H^{-1/2}}+\norm{\P_{\X_s}^\perp\bm{f}}_{H^{-1/2}}\big)\\
   &\le c\,\epsilon^{-1}\abs{\log\epsilon}(1+\norm{\X_s}_{H^1}^2)\norm{\bm{f}}_{H^{-1/2}}\,.
  \end{aligned}
 \end{equation}  
Applying this bound to the tension term $(\tau\X_s)_s$, we have 
\begin{equation}
\begin{aligned}
   &\norm{\overline{\mc{L}}_\epsilon[(\tau\X_s)_s]}_{H^{1/2}_s}
   \le c\,\epsilon^{-1}\abs{\log\epsilon}(1+\norm{\X_s}_{H^1_s}^2)\norm{\tau\X_s}_{H^{1/2}_s}\\
   &\qquad\le c\,\epsilon^{-2}\abs{\log\epsilon}(1+\norm{\X_s}_{H^1_s}^6)\big(\norm{\X_{sss}}_{\dot H^{1/2}_s}^{3/4}\norm{\X_s}_{H^1_s}^{9/4} + \norm{\X_{sss}}_{\dot H^{1/2}_s}^{2/3}\norm{\X_s}_{H^1_s}^{13/3} \big)\,,
  \end{aligned}
 \end{equation}
by Lemma \ref{lem:TDP}.
 Integrating in time from 0 to $T$, for $\X\in \bm{B}_{0,1}$, we have 
\begin{equation}
\begin{aligned}
 \norm{\overline{\mc{L}}_\epsilon[(\tau\X_s)_s]}_{L^2_tH^{1/2}_s} & 
 \le c\,\epsilon^{-2}\abs{\log\epsilon}(1+\norm{\X_s}_{L^\infty_tH^1_s}^6)\big(T^{1/8}\norm{\X_{sss}}_{L^2_t\dot H^{1/2}_s}^{3/4}\norm{\X_s}_{L^\infty_tH^1_s}^{9/4} \\
 &\quad + T^{1/6}\norm{\X_{sss}}_{L^2_t\dot H^{1/2}_s}^{2/3}\norm{\X_s}_{L^\infty_tH^1_s}^{13/3} \big)\\
 &\le c\,\epsilon^{-2}\abs{\log\epsilon}(1+M_0^6)\big(T^{1/8}M_1^{3/4}M_0^{9/4} + T^{1/6}M_1^{2/3}M_0^{13/3} \big)\,.
\end{aligned}
\end{equation}

For the remainder terms $\mc{R}[\X]$, by Lemma \ref{lem:op_decomp}, we additionally have
\begin{equation}
\begin{aligned}
  \norm{\mc{R}[\X]}_{L^2_tH^{1/2}_s} &\le 
  c\,\epsilon^{-3/2}\big(T^{1/6}\norm{\X_{sss}}_{L^2_t\dot H^{1/2}_s}^{2/3}\norm{\X_s}_{L^\infty_tH^1_s}^{4/3}\\
  &\qquad\quad  + T^{1/2}\norm{\X_s}_{L^\infty_tH^1_s}^2 \big)(1+\norm{\X_s}_{L^\infty_tH^1_s}^2)\\
  &\le c\,\epsilon^{-3/2}\big(T^{1/6}M_1^{2/3}M_0^{4/3} + T^{1/2}M_0^2 \big)(1+M_0^2)\,.
\end{aligned}
\end{equation}

Altogether, using Lemma \ref{lem:linear}, we may bound 
\begin{equation}
\begin{aligned}
  \norm{\X}_{\mc{Y}_0} &\le \norm{\X^{\rm in}}_{H^2_s} + c\,\epsilon^{-3/2}(T^{1/2}+\abs{\log\epsilon}^{-1/2})\big(T^{1/6}M_1^{2/3}M_0^{4/3} + T^{1/2}M_0^2 \big)(1+M_0^2)\\
  &\quad + c\,\epsilon^{-2}\abs{\log\epsilon}(T^{1/2}+\abs{\log\epsilon}^{-1/2})(1+M_0^6)\big(T^{1/8}M_1^{3/4}M_0^{9/4} + T^{1/6}M_1^{2/3}M_0^{13/3} \big) \\
  &\le \norm{\X^{\rm in}}_{H^2_s} + c_0(M_0,M_1)\,\epsilon^{-2}\abs{\log\epsilon}\,T^{1/8}\,,\\
  \norm{\X}_{\mc{Y}_1} &\le \abs{\log\epsilon}^{-1/2}\norm{e^{-t\,T_{m_\epsilon^{\rm n}}\p_s^4}\X^{\rm in}}_{\mc{Y}_1}\\
  &\quad  + c\,\epsilon^{-3/2}\abs{\log\epsilon}^{-1/2}(T^{1/2}+\abs{\log\epsilon}^{-1/2})\big(T^{1/6}M_1^{2/3}M_0^{4/3} + T^{1/2}M_0^2 \big)(1+M_0^2)\\
  &\quad + c\,\epsilon^{-2}\abs{\log\epsilon}^{1/2}(T^{1/2}+\abs{\log\epsilon}^{-1/2})(1+M_0^6)\big(T^{1/8}M_1^{3/4}M_0^{9/4} + T^{1/6}M_1^{2/3}M_0^{13/3} \big)\\
  &\le \abs{\log\epsilon}^{-1/2}M_1^{\rm in}(T) + c_1(M_0,M_1)\,\epsilon^{-2}\abs{\log\epsilon}^{1/2}\,T^{1/8} \,,
\end{aligned}
\end{equation}
where $M_1^{\rm in}(T):= \norm{e^{-t\,T_{m_\epsilon^{\rm n}}\p_s^4}\X^{\rm in}}_{\mc{Y}_1}\to 0$ as $T\to 0^+$ since $\norm{\cdot}_{\mc{Y}_1}=\norm{\cdot}_{L^2_t\dot H^{7/2}_s}$ is an integral quantity in $t$. 
Taking $M_0 = 4\norm{\X^{\rm in}}_{H^2_s}$ and then choosing $T$ small enough that  
\begin{equation}\label{eq:Tcond1}
\begin{aligned}
  c_0(M_0,M_1)\,\epsilon^{-2}\abs{\log\epsilon}\,T^{1/8}&\le \frac{M_0}{2}
  \qquad \text{and}\\ 
  \abs{\log\epsilon}^{-1/2}M_1^{\rm in}(T) + c_1(M_0,M_1)\,\epsilon^{-2}\abs{\log\epsilon}^{1/2}\,T^{1/8}&\le \frac{3M_1}{4}\,,
\end{aligned}
\end{equation}
we obtain that the time-$T$ solution map $\Phi^T(\X)$ maps $\bm{B}_{0,1}$ to itself.

We next show that the solution map $\Phi^T$ is a contraction on $\bm{B}_{0,1}$. Again, we begin with the tension term. 
First, for any $\bm{f}_1$, $\bm{f}_2:\T\to \R^3$ along corresponding nearby curves $\X_1$, $\X_2$, we may write
\begin{equation}
\begin{aligned}
  \overline{\mc{L}}_\epsilon[\bm{f}_1]-\overline{\mc{L}}_\epsilon[\bm{f}_2] &=  \P_{(\X_1)_s}T_{m_\epsilon^{\rm t}}(\P_{(\X_1)_s}\bm{f}_1) - \P_{(\X_2)_s}T_{m_\epsilon^{\rm t}}(\P_{(\X_2)_s}\bm{f}_2)\\
  &\quad + \P_{(\X_1)_s}^\perp T_{m_\epsilon^{\rm n}}(\P_{(\X_1)_s}^\perp \bm{f}_1) - \P_{(\X_2)_s}^\perp T_{m_\epsilon^{\rm n}}(\P_{(\X_2)_s}^\perp \bm{f}_2)\,.
\end{aligned}
\end{equation}
We may then estimate
\begin{equation}
\begin{aligned}
  &\norm{\overline{\mc{L}}_\epsilon[\bm{f}_1]-\overline{\mc{L}}_\epsilon[\bm{f}_2]}_{H^{1/2}} \le  
  \norm{\P_{(\X_1)_s}- \P_{(\X_2)_s}}_{H^1}\norm{T_{m_\epsilon^{\rm t}}(\P_{(\X_1)_s}\bm{f}_1)}_{H^{1/2}} \\
  &\qquad +
  \norm{\P_{(\X_2)_s}}_{H^1}\big(\norm{T_{m_\epsilon^{\rm t}}\big((\P_{(\X_1)_s}-\P_{(\X_2)_s})\bm{f}_1\big)}_{H^{1/2}}+
  \norm{T_{m_\epsilon^{\rm t}}(\P_{(\X_2)_s}(\bm{f}_1-\bm{f}_2))}_{H^{1/2}}\big)\\
  &\quad\qquad  + \norm{\P_{(\X_1)_s}^\perp- \P_{(\X_2)_s}^\perp}_{H^1}\norm{T_{m_\epsilon^{\rm n}}(\P_{(\X_1)_s}^\perp\bm{f}_1)}_{H^{1/2}} \\
  &\qquad +
  \norm{\P_{(\X_2)_s}^\perp}_{H^1}\big(\norm{T_{m_\epsilon^{\rm n}}\big((\P_{(\X_1)_s}^\perp-\P_{(\X_2)_s}^\perp)\bm{f}_1\big)}_{H^{1/2}}+
  \norm{T_{m_\epsilon^{\rm n}}(\P_{(\X_2)_s}^\perp(\bm{f}_1-\bm{f}_2))}_{H^{1/2}}\big)\\
  &\le c\,\epsilon^{-1}\abs{\log\epsilon}\norm{(\X_1-\X_2)_s}_{H^1}\big(\norm{\P_{(\X_1)_s}\bm{f}_1}_{H^{-1/2}}+\norm{\P_{(\X_1)_s}^\perp\bm{f}_1}_{H^{-1/2}}\big)\\
  &\qquad + c\,\epsilon^{-1}\abs{\log\epsilon}(1+\norm{(\X_2)_s}_{H^1}) \norm{(\P_{(\X_1)_s}-\P_{(\X_2)_s})\bm{f}_1}_{H^{-1/2}}  \\
  &\qquad + c\,\epsilon^{-1}\abs{\log\epsilon}(1+\norm{(\X_2)_s}_{H^1})\big( \norm{\P_{(\X_2)_s}(\bm{f}_1-\bm{f}_2)}_{H^{-1/2}} + \norm{\P_{(\X_2)_s}^\perp(\bm{f}_1-\bm{f}_2)}_{H^{-1/2}}\big)\\
  &\le c\,\epsilon^{-1}\abs{\log\epsilon}\norm{(\X_1-\X_2)_s}_{H^1}(1+\norm{(\X_1)_s}_{H^1}+\norm{(\X_2)_s}_{H^1})\norm{\bm{f}_1}_{H^{-1/2}}\\
  &\qquad + c\,\epsilon^{-1}\abs{\log\epsilon}(1+\norm{(\X_2)_s}_{H^1})^2\norm{\bm{f}_1-\bm{f}_2}_{H^{-1/2}}\,. 
\end{aligned}
\end{equation}
Plugging in $\bm{f}_1=(\tau_1(\X_1)_s)_s$ and $\bm{f}_2=(\tau_2(\X_2)_s)_s$ for $\tau_1$, $\tau_2$ satisfying the tension equations \eqref{eq:TDP} corresponding to $\X_1$ and $\X_2$, respectively, we have
\begin{equation}
\begin{aligned}
  &\norm{\overline{\mc{L}}_\epsilon[(\tau_1(\X_1)_s)_s]-\overline{\mc{L}}_\epsilon[(\tau_2(\X_2)_s)_s]}_{H^{1/2}_s} \\
  &\le c\,\epsilon^{-1}\abs{\log\epsilon}\norm{(\X_1-\X_2)_s}_{H^1_s}(1+\norm{(\X_1)_s}_{H^1_s}+\norm{(\X_2)_s}_{H^1_s})\norm{\tau_1(\X_1)_s}_{\dot H^{1/2}_s}\\
  &\qquad + c\,\epsilon^{-1}\abs{\log\epsilon}(1+\norm{(\X_2)_s}_{H^1_s})^2\norm{\tau_1(\X_1)_s-\tau_2(\X_2)_s}_{\dot H^{1/2}_s}\\
  &\le c\,\epsilon^{-1}\abs{\log\epsilon}\norm{(\X_1-\X_2)_s}_{H^1_s}(1+\norm{(\X_1)_s}_{H^1_s}+\norm{(\X_2)_s}_{H^1_s})^2(\norm{\tau_1}_{H^{1/2}_s}+\norm{\tau_2}_{H^{1/2}_s})\\
  &\qquad + c\,\epsilon^{-1}\abs{\log\epsilon}(1+\norm{(\X_2)_s}_{H^1_s})^2\norm{(\X_1)_s}_{H^1_s}\norm{\tau_1-\tau_2}_{H^{1/2}_s}\\
  &\le c\,\epsilon^{-2}\abs{\log\epsilon}(1+\norm{(\X_1)_s}_{H^1_s}^{12}+\norm{(\X_2)_s}_{H^1_s}^{12})\times\\
  &\quad \times \bigg[\norm{(\X_1-\X_2)_{sss}}_{\dot H^{1/2}_s}^{1/3}\norm{(\X_1-\X_2)_s}_{H^1_s}^{2/3}(\norm{(\X_1)_{sss}}_{\dot H^{1/2}_s}^{1/3}+\norm{(\X_2)_{sss}}_{\dot H^{1/2}_s}^{1/3}) \\
  &+ \norm{(\X_1-\X_2)_s}_{H^1_s}\bigg(\norm{(\X_1)_{sss}}_{\dot H^{1/2}_s}^{2/3} + \norm{(\X_2)_{sss}}_{\dot H^{1/2}_s}^{2/3} + \norm{(\X_1)_{sss}}_{\dot H^{1/2}_s}^{3/4}+\norm{(\X_2)_{sss}}_{\dot H^{1/2}_s}^{3/4}\bigg) \\
  &\qquad  + \norm{(\X_1-\X_2)_{sss}}_{\dot H^{1/2}_s}^{3/4}\norm{(\X_1-\X_2)_s}_{H^1_s}^{1/4} \bigg] \,.
\end{aligned}
\end{equation}
Here we have used the tension bounds of Lemma \ref{lem:TDP}. 
Integrating in time from $0$ to $T$, we obtain the $L^2_t$ bound 
\begin{equation}
\begin{aligned}
  &\norm{\overline{\mc{L}}_\epsilon[(\tau_1(\X_1)_s)_s]-\overline{\mc{L}}_\epsilon[(\tau_2(\X_2)_s)_s]}_{L^2_tH^{1/2}_s} \\
  &\quad\le c\,\epsilon^{-2}\abs{\log\epsilon}(1+M_0^{12})\bigg[T^{1/6}\norm{\X_1-\X_2}_{\mc{Y}_1}^{1/3}\norm{\X_1-\X_2}_{\mc{Y}_0}^{2/3}M_1^{1/3} \\
  &\quad\qquad + T^{1/8}\norm{\X_1-\X_2}_{\mc{Y}_1}^{3/4}\norm{\X_1-\X_2}_{\mc{Y}_0}^{1/4} + \norm{\X_1-\X_2}_{\mc{Y}_0}\big(T^{1/6}M_1^{2/3} + T^{1/8}M_1^{3/4}\big)  \bigg]\,.
\end{aligned}
\end{equation}

For the remainder term $\mc{R}[\X]$, integrating the Lipschitz bound of Lemma \ref{lem:op_decomp} in time from 0 to $T$, we have
  \begin{equation}
  \begin{aligned}
  &\norm{\mc{R}[\X_1]-\mc{R}[\X_2]}_{L^2_tH^{1/2}_s}\\
  &\quad \le c\,\epsilon^{-3/2}(1+M_0^{10/3})\,T^{1/6}\bigg[\norm{\X_1-\X_2}_{\mc{Y}_0}\big(M_1 + T^{1/3}(M_0^{2/3}+1)\big)\\
  &\qquad\quad + \norm{\X_1-\X_2}_{\mc{Y}_0}^{2/3}\norm{\X_1-\X_2}_{\mc{Y}_1}^{1/3}M_1^{1/3} 
  + \norm{\X_1-\X_2}_{\mc{Y}_1}^{2/3}\norm{\X_1-\X_2}_{\mc{Y}_0}^{1/3}\bigg]\,.
  \end{aligned}
  \end{equation}

Altogether, using Lemma \ref{lem:linear}, we may estimate 
\begin{equation}
\begin{aligned}
  \norm{\Phi^T(\X_1)-\Phi^T(\X_2)}_{\mc{Y}_0\cap\mc{Y}_1}
  &\le \norm{\overline{\mc{L}}_\epsilon[(\tau_1(\X_1)_s)_s]-\overline{\mc{L}}_\epsilon[(\tau_2(\X_2)_s)_s]}_{L^2_tH^{1/2}_s}\\
  &\qquad+ \norm{\mc{R}[\X_1]-\mc{R}[\X_2]}_{L^2_tH^{1/2}_s}\\
  &\le c(M_0,M_1)\,\epsilon^{-2}\abs{\log\epsilon}\,T^{1/8}\norm{\X_1-\X_2}_{\mc{Y}_0\cap\mc{Y}_1}\,.
\end{aligned}
\end{equation}
Taking $T$ small enough that $c(M_0,M_1)\,\epsilon^{-2}\abs{\log\epsilon}\,T^{1/8}\le \frac{1}{2}$ in addition to the above bounds \eqref{eq:Tcond1}, we obtain that $\Phi^T$ is a contraction on $\bm{B}_{0,1}$ and hence admits a unique fixed point.
\end{proof}

\subsection{Global well-posedness}\label{subsec:GWP}
Given the local well-posedness result of Lemma \ref{lem:LWP}, we may proceed to the proof of Theorem \eqref{thm:GWP} regarding global well-posedness for the evolution \eqref{eq:evolution}.
Given the form \eqref{eq:energy_ineq} of the energy inequality satisfied by the filament and following \cite{koiso1996motion,oelz2011curve,albritton2025rods}, it will be convenient to define the quantity 
\begin{equation}\label{eq:zee}
  \bm{Z} := \X_{sss}-\tau\X_s\,.
\end{equation}
We define the energy and dissipation quantities $E$ and $D$ as 
\begin{equation}\label{eq:EandD}
  E(t) := \frac{1}{2}\norm{\X_{ss}}_{L^2_s(\T)}^2\,, \qquad
  D(t) := \norm{\bm{Z}}_{\dot H^{1/2}_s(\T)}^2\,,
\end{equation}
so that, by \eqref{eq:energy_ineq},
\begin{equation}\label{eq:EandD_ineq}
  E(t) \le -c\,\abs{\log\epsilon} D(t)\,.
\end{equation}
We aim to extract control on $\norm{\X_{sss}}_{H^{1/2}_s}$, i.e. on the curve alone, in terms of $E$ and $D$. We show the following.
\begin{lemma}[Control on $\X_{sss}$]\label{lem:Xsss_control}
  A curve $\X$ satisfying \eqref{eq:evolution} may be controlled by the energy and dissipation quantities as follows:
  \begin{equation}\label{eq:Xsss_control}
    \norm{\X_{sss}}_{H^{1/2}_s(\T)} \le c\big( D(t)^{1/2}(1+E(t)^{1/2}) + E(t)^{1/2}+ E(t)^4\big)\,.
\end{equation} 
\end{lemma}

\begin{proof}
We begin by noting that the quantities $D$ and $E$ control the full $H^{1/2}_s$ norm of $\bm{Z}$ in the following way. Using \eqref{eq:inex_ids}, we may calculate that the mean of $\bm{Z}$ satisfies
\begin{equation}
\begin{aligned}
  \bm{Z}_0 &:= \int_\T \bm{Z}\,ds = \int_\T (\X_{sss}-\tau\X_s)\,ds = \int_\T \tau\X_s\,ds \\
  &= -\int_\T \big((\bm{Z}-\bm{Z}_0)\cdot\X_s + \bm{Z}_0\cdot\X_s+\abs{\X_{ss}}^2\big)\X_s\,ds\,.
\end{aligned}
\end{equation}
From this we obtain 
\begin{equation}
  \bigg(\int_\T({\bf I}+ \X_s\otimes\X_s)\,ds\bigg) \bm{Z}_0 = -\int_\T \big((\bm{Z}-\bm{Z}_0)\cdot\X_s +\abs{\X_{ss}}^2\big)\X_s\,ds\,.
\end{equation}
Using that the left hand side integrand is bounded below pointwise and that, on $\T$, we have
\begin{align*}
  \norm{\bm{Z}-\bm{Z}_0}_{L^2_s(\T)}^2 = \sum_{\abs{k}=1}^\infty\mc{F}[{\bm Z}]\overline{\mc{F}[{\bm Z}]} \le \sum_{\abs{k}=1}^\infty\abs{k}\mc{F}[{\bm Z}]\overline{\mc{F}[{\bm Z}]}\le \norm{\bm{Z}}_{\dot H^{1/2}_s(\T)}^2\,,
\end{align*}
we may then estimate
\begin{equation}
\begin{aligned}
  \abs{\bm{Z}_0} &\le c\big(\norm{\bm{Z}}_{\dot H^{1/2}_s} + \norm{\X_{ss}}_{L^2_s}^2 \big) 
  \le c(D(t)^{1/2}+E(t))\,.
\end{aligned}
\end{equation}
We thus have
\begin{equation}\label{eq:Z_Hhalf}
  \norm{\bm{Z}}_{H^{1/2}_s(\T)} \le c(D(t)^{1/2}+E(t))\,.
\end{equation}

We now proceed to estimate $\norm{\X_{sss}}_{H^{1/2}_s}$. First note that we may write the normal components of $\X_{sss}$ as
\begin{equation}
  \P_{\X_s}^\perp \X_{sss} = \P_{\X_s}^\perp \bm{Z} = \bm{Z} - \P_{\X_s}\bm{Z}\,.
\end{equation}
Using \eqref{eq:Z_Hhalf}, we may estimate
\begin{equation}
  \norm{\P_{\X_s}\bm{Z}}_{\dot H^{1/2}_s} \le \norm{\bm{Z}}_{H^{1/2}_s}\norm{\X_s\otimes\X_s}_{H^1_s}
  \le c\,(D(t)^{1/2}+E(t))(1+E(t)^{1/2}) \,,
\end{equation}
so that
\begin{equation}\label{eq:normal_est}
  \norm{\P_{\X_s}^\perp \X_{sss}}_{\dot H^{1/2}_s} \le c\,(D(t)^{1/2}+E(t))(1+E(t)^{1/2})\,.
\end{equation}
Next, using \eqref{eq:inex_ids}, we may write the tangential component of $\X_{sss}$ as 
\begin{equation}\label{eq:tang1}
  \P_{\X_s}\X_{sss} = (\X_s\cdot\X_{sss})\X_s =-\abs{\X_{ss}}^2\X_s\,.
\end{equation}
To estimate this in $\dot H^{1/2}_s$, we will make use of the interpolation inequalities \eqref{eq:interps} on $\T$ to obtain
\begin{equation}
\begin{aligned}
  \norm{\abs{\X_{ss}}^2\X_s}_{\dot H^{1/2}_s} &\le c\norm{(\abs{\X_{ss}}^2\X_s)_s}_{L^2_s}^{1/2}\norm{\abs{\X_{ss}}^2\X_s}_{L^2_s}^{1/2}\\
  &\le c\norm{(\abs{\X_{ss}}^2\X_s)_s}_{L^2_s}^{1/2}\norm{\X_{ss}}_{L^2_s}^{5/6}\norm{\X_{sss}}_{\dot H^{1/2}_s}^{1/6}\,.
\end{aligned}
\end{equation}
Expanding 
\begin{equation}
  (\abs{\X_{ss}}^2\X_s)_s = 2(\X_{ss}\cdot\X_{sss})\X_s+ \abs{\X_{ss}}^2\X_{ss}\,,
\end{equation}
we may further estimate 
\begin{equation}
  \begin{aligned}
  \norm{(\abs{\X_{ss}}^2\X_s)_s}_{L^2_s} &\le 2\norm{\X_{ss}}_{L^\infty_s}\norm{\X_{sss}}_{L^2_s} + \norm{\X_{ss}}_{L^\infty_s}^2\norm{\X_{ss}}_{L^2_s}\\
  &\le \norm{\X_{sss}}_{\dot H^{1/2}_s}\norm{\X_{ss}}_{L^2_s}
  + \norm{\X_{sss}}_{\dot H^{1/2}_s}^{2/3}\norm{\X_{ss}}_{L^2_s}^{7/3}\,.
\end{aligned}
\end{equation}
In total, the right hand side of \eqref{eq:tang1} satisfies 
\begin{equation}\label{eq:tang_est}
  \norm{\abs{\X_{ss}}^2\X_s}_{\dot H^{1/2}_s} 
  \le c\big(\norm{\X_{sss}}_{\dot H^{1/2}_s}^{2/3}\norm{\X_{ss}}_{L^2_s}^{4/3}
  + \norm{\X_{sss}}_{\dot H^{1/2}_s}^{1/2}\norm{\X_{ss}}_{L^2_s}^2\big)\,.
\end{equation}
Using \eqref{eq:tang_est} and \eqref{eq:normal_est}, we may thus bound
\begin{equation}
\begin{aligned}
  &\norm{\X_{sss}}_{\dot H^{1/2}_s} \le \norm{\P_{\X_s}^\perp\X_{sss}}_{\dot H^{1/2}_s} + \norm{\P_{\X_s}\X_{sss}}_{\dot H^{1/2}_s}\\
  &\qquad\le c(D(t)^{1/2}+E(t))(1+E(t)^{1/2}) + c\big(\norm{\X_{sss}}_{\dot H^{1/2}_s}^{2/3}E(t)^{4/3}
  + \norm{\X_{sss}}_{\dot H^{1/2}_s}^{1/2}E(t)^2\big)\\
  &\qquad \le c\big( D(t)^{1/2}(1+E(t)^{1/2}) + E(t)+ E(t)^4\big)\,,
\end{aligned}
\end{equation}
where we have used Young's inequality to split the quantities in the second line. Using \eqref{eq:interps}, we may further estimate
\begin{equation}
    \norm{\X_{sss}}_{L^2_s} \le c(\norm{\X_{sss}}_{\dot H^{1/2}_s}+\norm{\X_{ss}}_{L^2_s})
    \le c\big( D(t)^{1/2}(1+E(t)^{1/2}) + E(t)^{1/2}+ E(t)^4\big)\,,
\end{equation}  
from which we obtain \eqref{eq:Xsss_control}.
\end{proof}

Given Lemma \ref{lem:Xsss_control} and the continuation criterion \eqref{eq:continuation}, the local solution of Lemma \ref{lem:LWP} may be extended globally in time to yield Theorem \ref{thm:GWP}.

\section{Convergence to resistive force theory dynamics}\label{sec:RFTconverge}
We may next turn to the limiting dynamics of \eqref{eq:evolution} as $\epsilon\to 0$.

\subsection{Setup}
Given an initial curve\footnote{Here $\X^{\rm in}\in H^{5/2}(\T)$ should actually be sufficient regularity for our arguments, but in order to use the results of \cite[Theorem 1.1]{albritton2025rods} as a black box, we will take $\X^{\rm in}\in H^4(\T)$.} $\X^{\rm in}\in H^4(\T)$, let $\Y\in C([0,T];H^4_s(\T))\cap L^2_{t,{\rm loc}}H^6_s$ be the resulting unique global solution to the resistive force theory evolution \eqref{eq:RFT_FBP}, guaranteed by \cite[Theorem~1.1]{albritton2025rods}. In particular, $\Y$ satisfies  
\begin{equation}\label{eq:RFT_Yeqn}
\begin{aligned}
   \frac{\p\Y}{\p t} &= -\frac{\abs{\log\epsilon}}{4\pi}({\bf I}+\Y_s\otimes\Y_s)\big(\Y_{sss}-\tau_Y\Y_s \big)_s\\
   &=-\frac{\abs{\log\epsilon}}{2\pi}\P_{\Y_s}\big(\Y_{sss}-\tau_Y\Y_s \big)_s -\frac{\abs{\log\epsilon}}{4\pi}\P_{\Y_s}^\perp\big(\Y_{sss}-\tau_Y\Y_s \big)_s \,,\\
   \abs{\Y_s}&=1\,,\qquad \Y\big|_{t=0}=\X^{\rm in}\,.
\end{aligned}
 \end{equation}

Letting $\X(s,t)\in C([0,T];H^2_s(\T))\cap L^2_{t,{\rm loc}}H^{7/2}_s$ denote the global solution to \eqref{eq:evolution} starting from $\X^{\rm in}$, we consider the difference $\bm{W}= \X-\Y$, whose evolution then satisfies 
\begin{equation}\label{eq:Wevolution1}
\begin{aligned}
  \frac{\p\bm{W}}{\p t} = -\overline{\mc{L}}_\epsilon(\X)[(\X_{sss}-\tau_X\X_s)_s]+ \frac{\abs{\log\epsilon}}{4\pi}({\bf I}+\Y_s\otimes\Y_s)(\Y_{sss}-\tau_Y\Y_s)_s\,, \quad \bm{W}\big|_{t=0}=0\,.
\end{aligned}
\end{equation}
Now, noting that both the map $\overline{\mc{L}}_\epsilon(\X)$ and the resistive force theory operator $\mc{L}_{\epsilon,\rm RFT}(\Y)=\frac{\abs{\log\epsilon}}{4\pi}({\bf I}+\Y_s\otimes\Y_s)$ blow up logarithmically as $\epsilon\to 0$, to obtain a convergence result for their dynamics, we consider \eqref{eq:Wevolution1} under the following rescaling of time:
\begin{equation}\label{eq:t_rescale}
  \underline{t} = \abs{\log\epsilon} t\,, \quad \p_{\underline t} = \abs{\log\epsilon}^{-1} \p_t\,.
\end{equation}
The left hand side of the evolution \eqref{eq:Wevolution1} may then be written
\begin{equation}
  \frac{\p\bm{W}}{\p t} = \abs{\log\epsilon}\frac{\p\bm{W}}{\p \underline t}\,.
\end{equation}
We may split the right hand side of \eqref{eq:Wevolution1} as
\begin{equation}
\begin{aligned}
  -\P_{\X_s}T_{m_\epsilon^{\rm t}}\P_{\X_s}&\big(\X_{sss}-\tau_X\X_s \big)_s  + \P_{\Y_s}\frac{\abs{\log\epsilon}}{2\pi}\P_{\Y_s}\big(\Y_{sss}-\tau_Y\Y_s \big)_s \\
   - \P_{\X_s}^\perp T_{m_\epsilon^{\rm n}}&\P_{\X_s}^\perp\big(\X_{sss}-\tau_X\X_s \big)_s+ \P_{\Y_s}^\perp\frac{\abs{\log\epsilon}}{4\pi}\P_{\Y_s}^\perp\big(\Y_{sss}-\tau_Y\Y_s \big)_s\\
  &= -\P_{\X_s}T_{m_\epsilon^{\rm t}}\P_{\X_s}\big(\bm{W}_{sss}-(\tau_X\X_s-\tau_Y\Y_s) \big)_s  \\
  &\qquad - \bigg(\P_{\X_s}T_{m_\epsilon^{\rm t}}\P_{\X_s}-\P_{\Y_s}\frac{\abs{\log\epsilon}}{2\pi}\P_{\Y_s}\bigg)\big(\Y_{sss}-\tau_Y\Y_s \big)_s\\
  &\quad - \P_{\X_s}^\perp T_{m_\epsilon^{\rm n}}\P_{\X_s}^\perp\big(\bm{W}_{sss}-(\tau_X\X_s-\tau_Y\Y_s) \big)_s\\
  &\qquad - \bigg(\P_{\X_s}^\perp T_{m_\epsilon^{\rm n}}\P_{\X_s}^\perp-\P_{\Y_s}^\perp\frac{\abs{\log\epsilon}}{4\pi}\P_{\Y_s}^\perp\bigg)\big(\Y_{sss}-\tau_Y\Y_s \big)_s\,,
\end{aligned}
\end{equation}
so that, upon multiplying \eqref{eq:Wevolution1} by $\bm{W}_{ssss}$ and integrating over $\T$, we obtain the energy identity 
\begin{equation}\label{eq:Wenergy1}
\begin{aligned}
  \abs{\log\epsilon}&\frac{\p}{\p \underline{t}}\int_\T\abs{\bm{W}_{ss}}^2\,ds +\int_\T (\P_{\X_s}\bm{W}_{ssss})\cdot T_{m_\epsilon^{\rm t}}(\P_{\X_s}\bm{W}_{ssss})\,ds \\
  &\qquad\quad + \int_\T (\P_{\X_s}^\perp\bm{W}_{ssss})\cdot T_{m_\epsilon^{\rm n}}(\P_{\X_s}^\perp\bm{W}_{ssss})\,ds = \bm{A}^\parallel +\bm{A}^\perp \,, \\
  & \bm{A}^\parallel = \int_\T(\P_{\X_s}\bm{W}_{ssss})\cdot T_{m_\epsilon^{\rm t}}\P_{\X_s}\big(\tau_W\X_s+\tau_Y\bm{W}_s \big)_s \,ds \\
  &\qquad\quad - \int_\T \bm{W}_{ssss}\cdot\bigg(\P_{\X_s}T_{m_\epsilon^{\rm t}}\P_{\X_s}-\P_{\Y_s}\frac{\abs{\log\epsilon}}{2\pi}\P_{\Y_s}\bigg)\big(\Y_{sss}-\tau_Y\Y_s \big)_s\,ds\\
  & \bm{A}^\perp = \int_\T (\P_{\X_s}^\perp\bm{W}_{ssss})\cdot T_{m_\epsilon^{\rm n}}\P_{\X_s}^\perp\big(\tau_W\X_s+\tau_Y\bm{W}_s \big)_s\,ds \\
  &\qquad\quad - \int_\T\bm{W}_{ssss}\cdot\bigg(\P_{\X_s}^\perp T_{m_\epsilon^{\rm n}}\P_{\X_s}^\perp-\P_{\Y_s}^\perp\frac{\abs{\log\epsilon}}{4\pi}\P_{\Y_s}^\perp\bigg)\big(\Y_{sss}-\tau_Y\Y_s \big)_s\,ds\,.
\end{aligned}
\end{equation}

Our goal will be to bound $\bm{A}^\parallel$ and $\bm{A}^\perp$ in terms of the energy and dissipation quantities appearing on the left hand side of \eqref{eq:Wenergy1}. Since convergence is expected only at a very slow logarithmic rate as $\epsilon\to 0$ (see Theorem \ref{thm:RFTconv} and introductory discussion), these estimates are a bit delicate. In particular, we see from the local well-posedness theory that because of the difference in $\epsilon$-scaling for $\overline{\mc{L}}_\epsilon$ at high versus low wavenumbers, the `coarse' upper and lower bounds differ by a factor of $\epsilon^{-1}$. This will be far too lossy for convergence, so whenever possible, we will aim to avoid actually evaluating the operators $T_{m_\epsilon^{\rm t}}$ and $T_{m_\epsilon^{\rm n}}$. Instead , we will rely on the fact that both $T_{m_\epsilon^{\rm t}}$ and $T_{m_\epsilon^{\rm n}}$ are positive operators, and, in particular, we may apply $T_{m_\epsilon^{\rm t}}^{1/2}$ and $T_{m_\epsilon^{\rm n}}^{1/2}$. Furthermore, since $m_\epsilon^{\rm n}(k)\lesssim m_\epsilon^{\rm t}(k)\lesssim m_\epsilon^{\rm n}(k)$ independent of $\epsilon$, by Lemma \ref{lem:multiplierbds}, we may interchange $T_{m_\epsilon^{\rm t}}$ and $T_{m_\epsilon^{\rm n}}$ in upper and lower bounds as needed.

To proceed, we will need a few additional ingredients. The first is a bound for powers of the inverse operators $T_{m_\epsilon^{\rm t}}^{-1}=T_{1/m_\epsilon^{\rm t}}$, $T_{m_\epsilon^{\rm n}}^{-1}=T_{1/m_\epsilon^{\rm n}}$.
\begin{lemma}[Inverse multiplier operators]\label{lem:Tinverse}
  For $\theta\ge 0$, the inverse operators 
  \begin{equation}
    T_{m_\epsilon^{\rm t}}^{-\theta}(\cdot) = \mc{F}^{-1}\big[m_\epsilon^{\rm t}(k)^{-\theta}\mc{F}[\cdot]\big]\,, \quad
    T_{m_\epsilon^{\rm n}}^{-\theta}(\cdot) = \mc{F}^{-1}\big[m_\epsilon^{\rm n}(k)^{-\theta}\mc{F}[\cdot]\big]
  \end{equation}
satisfy, for $j={\rm t},{\rm n}$,
  \begin{equation}
    \norm{T_{m_\epsilon^j}^{-\theta}\bm{f}}_{H^\ell} \le c\abs{\log\epsilon}^{-\theta}\norm{\bm{f}}_{H^{\ell+\theta}}\,.
  \end{equation}
\end{lemma}
The proof follows directly from the multiplier bounds of Lemma \ref{lem:multiplierbds}.
We will use Lemma \ref{lem:Tinverse} along the way in estimating $\bm{A}^\parallel$ and $\bm{A}^\perp$.

\subsection{Operator difference}\label{subsec:OpDiff}
We will require bounds for the terms in $\bm{A}^\parallel$ and $\bm{A}^\perp$ involving the difference between the multiplier operators $T_{m_\epsilon^{\rm t}}$, $T_{m_\epsilon^{\rm n}}$ and their resistive force theory counterparts, the constants $\frac{\abs{\log\epsilon}}{2\pi}$, $\frac{\abs{\log\epsilon}}{4\pi}$. Here we exploit Lemma \ref{lem:lowk_refine}: the difference between the multipliers and the resistive force theory constants is bounded independent of $\epsilon$ at low wavenumber. 
\begin{lemma}[Multiplier operator and RFT difference]\label{lem:mult_RFT_diff}
Given $\bm{\varphi}\in H^{-1/2}(\T)$ and $\bm{f}\in L^2(\T)$, we may estimate 
 \begin{equation}
  \begin{aligned}
  \abs{\int_\T\bm{\varphi}\cdot\bigg(T_{m_\epsilon^{\rm t}}-\frac{\abs{\log\epsilon}}{2\pi}\bigg)[\bm{f}]\,ds} 
  &\le c\norm{T_{m_\epsilon^{\rm t}}^{1/2}\bm{\varphi}}_{L^2}\norm{\bm{f}}_{L^2}\\
   \abs{\int_\T\bm{\varphi}\cdot\bigg(T_{m_\epsilon^{\rm n}}-\frac{\abs{\log\epsilon}}{4\pi}\bigg)[\bm{f}]\,ds} 
   &\le c\norm{T_{m_\epsilon^{\rm n}}^{1/2}\bm{\varphi}}_{L^2}\norm{\bm{f}}_{L^2}\,.
\end{aligned}
\end{equation}
\end{lemma}
As a corollary, given a bit more regularity on $\bm{f}$, we may obtain the following bound: 
\begin{corollary}\label{cor:mult_RFT_diff}
  Given $\bm{\varphi}\in H^{-1/2}(\T)$ and $\bm{f}\in H^{1/2}(\T)$, we have 
 \begin{equation}
  \begin{aligned}
  &\abs{\int_\T\bm{\varphi}\cdot\bigg(\P_{\X_s} T_{m_\epsilon^{\rm t}}\P_{\X_s} -\P_{\Y_s}\frac{\abs{\log\epsilon}}{2\pi}\P_{\Y_s}\bigg)[\bm{f}]\,ds} 
  \le c\norm{T_{m_\epsilon^{\rm t}}^{1/2}(\P_{\X_s}\bm{\varphi})}_{L^2}\norm{\P_{\X_s}\bm{f}}_{L^2}\\
  &\qquad\qquad\qquad\quad + c\abs{\log\epsilon}^{1/2}\norm{\bm{W}_s}_{H^1}(\norm{\X_s}_{H^1}^2+\norm{\Y_s}_{H^1}^2)\norm{T_{m_\epsilon^{\rm t}}^{1/2}\bm{\varphi}}_{L^2}\norm{\bm{f}}_{H^{1/2}}\\
   &\abs{\int_\T\bm{\varphi}\cdot\bigg(\P_{\X_s}^\perp T_{m_\epsilon^{\rm n}}\P_{\X_s}^\perp-\P_{\Y_s}^\perp\frac{\abs{\log\epsilon}}{4\pi}\P_{\Y_s}^\perp\bigg)[\bm{f}]\,ds} 
   \le c\norm{T_{m_\epsilon^{\rm n}}^{1/2}(\P_{\X_s}^\perp\bm{\varphi})}_{L^2}\norm{\P_{\X_s}^\perp\bm{f}}_{L^2}\\
  &\qquad\qquad\qquad\quad + c\abs{\log\epsilon}^{1/2}\norm{\bm{W}_s}_{H^1}(1+\norm{\X_s}_{H^1}^2+\norm{\Y_s}_{H^1}^2)\norm{T_{m_\epsilon^{\rm n}}^{1/2}\bm{\varphi}}_{L^2}\norm{\bm{f}}_{H^{1/2}}\,.
\end{aligned}
\end{equation}
\end{corollary}
Note that the use of this corollary will be the only place requiring additional regularity on the resistive force theory curve $\Y$ beyond the natural energy space. 
We proceed to prove Lemma \ref{lem:mult_RFT_diff} and Corollary \ref{cor:mult_RFT_diff} together.
\begin{proof}
We begin with the tangential direction of Lemma \ref{lem:mult_RFT_diff}. Using Parseval's theorem and Lemma \ref{lem:multiplierbds} with the low wavenumber correction of Lemma \ref{lem:lowk_refine}, we have
\begin{equation}
\begin{aligned}
    &\abs{\int_\T\bm{\varphi}\cdot\bigg(T_{m_\epsilon^{\rm t}}-\frac{\abs{\log\epsilon}}{2\pi}\bigg)\bm{f}\,ds}
    =\abs{\sum_{\abs{k}=0}^\infty \bigg(m_\epsilon^{\rm t}(k)-\frac{\abs{\log\epsilon}}{2\pi}\bigg)\mc{F}[\bm{\varphi}]\cdot\overline{\mc{F}[\bm{f}]} }\\
    &\quad \le c\sum_{\abs{k}=1}^{\floor{\frac{1}{2\pi\epsilon}}}(1+\log\abs{k}) \abs{\mc{F}[\bm{\varphi}]\cdot\overline{\mc{F}[\bm{f}]}} + c\sum_{\abs{k}=\floor{\frac{1}{2\pi\epsilon}}}^\infty (\epsilon \abs{k})^{-1/2}\abs{m_\epsilon^{\rm t}(k)^{1/2}\mc{F}[\bm{\varphi}]\cdot\overline{\mc{F}[\bm{f}]}} \\
    &\quad \le c\norm{T_{m_\epsilon^{\rm t}}^{1/2}\bm{\varphi}}_{L^2}\norm{\bm{f}}_{L^2}\,.
\end{aligned}
\end{equation}
Here in the sum over high wavenumbers, we have used that $\abs{k}^{-1}\lesssim\epsilon$, in particular, we may trade the extra regularity of $\bm{f}$ to avoid losing factors of $\epsilon$. 
The normal direction bound follows by a nearly identical calculation using the low wavenumber correction of Lemma \ref{lem:lowk_refine} for $m_\epsilon^{\rm n}(k)$.

To show Corollary \ref{cor:mult_RFT_diff}, we again treat the tangential direction explicitly and note that the normal direction follows similarly. We may write
\begin{equation} 
\begin{aligned}
  \int_\T\bm{\varphi}\cdot\bigg(\P_{\X_s} T_{m_\epsilon^{\rm t}}\P_{\X_s}& -\P_{\Y_s}\frac{\abs{\log\epsilon}}{2\pi}\P_{\Y_s}\bigg)[\bm{f}]\,ds = I_1+I_2+I_3\,,\\
  I_1 &= \int_\T \P_{\X_s}\bm{\varphi}\cdot\bigg(T_{m_\epsilon^{\rm t}}-\frac{\abs{\log\epsilon}}{2\pi}\bigg)[\P_{\X_s}\bm{f}]\,ds\\
  I_2 &= \frac{\abs{\log\epsilon}}{2\pi}\int_\T\P_{\X_s}\bm{\varphi}\cdot (\P_{\X_s}-\P_{\Y_s})\bm{f}\,ds \\
  I_3 &= \frac{\abs{\log\epsilon}}{2\pi}\int_\T\bm{\varphi}\cdot(\P_{\X_s}-\P_{\Y_s})\P_{\X_s}\bm{f}\,ds \,.
\end{aligned}
\end{equation}
By Lemma \ref{lem:mult_RFT_diff}, we immediately have
\begin{equation}
  \abs{I_1}\le c\norm{T_{m_\epsilon^{\rm t}}^{1/2}(\P_{\X_s}\bm{\varphi})}_{L^2}\norm{\P_{\X_s}\bm{f}}_{L^2}\,.
\end{equation}
For $I_2$ and $I_3$, writing $\bm{\varphi}=T_{m_\epsilon^{\rm t}}^{-1/2}T_{m_\epsilon^{\rm t}}^{1/2}\bm{\varphi}$, we may use Lemma \ref{lem:Tinverse} and the additional regularity of $\bm{f}$ to estimate  
\begin{equation}
\begin{aligned}
  \abs{I_2} 
  &\le c\abs{\log\epsilon}\abs{\int_\T \big((T_{m_\epsilon^{\rm t}}^{-1/2}T_{m_\epsilon^{\rm t}}^{1/2}\bm{\varphi})\cdot\X_s\big)\X_s\cdot\big((\bm{W}_s\otimes\X_s+\Y_s\otimes\bm{W}_s)\bm{f}\big)\,ds } \\
  &\le c\abs{\log\epsilon}\norm{T_{m_\epsilon^{\rm t}}^{1/2}\bm{\varphi}}_{L^2}\norm{T_{m_\epsilon^{\rm t}}^{-1/2}\big(\X_s\otimes\X_s(\bm{W}_s\otimes\X_s+\Y_s\otimes\bm{W}_s)\bm{f}\big)}_{L^2}\\
  &\le c\abs{\log\epsilon}^{1/2}\norm{T_{m_\epsilon^{\rm t}}^{1/2}\bm{\varphi}}_{L^2}\norm{\X_s\otimes\X_s(\bm{W}_s\otimes\X_s+\Y_s\otimes\bm{W}_s)}_{H^1}\norm{\bm{f}}_{H^{1/2}}\\
  &\le c\abs{\log\epsilon}^{1/2}\norm{\bm{W}_s}_{H^1}(\norm{\X_s}_{H^1}^2+\norm{\Y_s}_{H^1}^2)\norm{T_{m_\epsilon^{\rm t}}^{1/2}\bm{\varphi}}_{L^2}\norm{\bm{f}}_{H^{1/2}}\,,
\end{aligned}
\end{equation}
as well as 
\begin{equation}
\begin{aligned}
  \abs{I_3} &\le c\abs{\log\epsilon}\abs{\int_\T (T_{m_\epsilon^{\rm t}}^{-1/2}T_{m_\epsilon^{\rm t}}^{1/2}\bm{\varphi})\cdot(\bm{W}_s\otimes\X_s+\Y_s\otimes\bm{W}_s)\X_s(\X_s\cdot\bm{f})\,ds} \\
  &\le c\abs{\log\epsilon}\norm{T_{m_\epsilon^{\rm t}}^{1/2}\bm{\varphi}}_{L^2}\norm{T_{m_\epsilon^{\rm t}}^{-1/2}\big((\bm{W}_s\otimes\X_s+\Y_s\otimes\bm{W}_s)\X_s(\X_s\cdot\bm{f})\big)}_{L^2}\\
  &\le c\abs{\log\epsilon}^{1/2}\norm{T_{m_\epsilon^{\rm t}}^{1/2}\bm{\varphi}}_{L^2}\norm{(\bm{W}_s\otimes\X_s+\Y_s\otimes\bm{W}_s)\X_s\otimes\X_s}_{H^1}\norm{\bm{f}}_{H^{1/2}}\\
  &\le c\abs{\log\epsilon}^{1/2}\norm{\bm{W}_s}_{H^1}(\norm{\X_s}_{H^1}^2+\norm{\Y_s}_{H^1}^2)\norm{T_{m_\epsilon^{\rm t}}^{1/2}\bm{\varphi}}_{L^2}\norm{\bm{f}}_{H^{1/2}}\,.
\end{aligned}
\end{equation}
In total, we may bound 
\begin{equation}
\begin{aligned}
  \abs{I_1}+\abs{I_2}+\abs{I_3} &\le c\norm{T_{m_\epsilon^{\rm t}}^{1/2}(\P_{\X_s}\bm{\varphi})}_{L^2}\norm{\P_{\X_s}\bm{f}}_{L^2}\\
  &\quad + c\abs{\log\epsilon}^{1/2}\norm{\bm{W}_s}_{H^1}(\norm{\X_s}_{H^1}^2+\norm{\Y_s}_{H^1}^2)\norm{T_{m_\epsilon^{\rm t}}^{1/2}\bm{\varphi}}_{L^2}\norm{\bm{f}}_{H^{1/2}}
\end{aligned}
\end{equation}
to obtain Corollary \ref{cor:mult_RFT_diff}.
\end{proof}

\subsection{Tension bounds and difference}\label{subsec:tauXYDiff}
In addition to the operator difference estimates from section \ref{subsec:OpDiff}, we will require bounds for terms resulting from the difference in tension behavior between the $\overline{\mc{L}}_\epsilon(\X)$ curve $\X$ and the resistive force theory curve $\Y$.

We first note that the resistive force theory equation for the tension $\tau_Y$ may be obtained in an analogous manner to \eqref{eq:TDP}. Differentiating \eqref{eq:RFT_Yeqn} in $s$ and dotting with $\Y_s$, the inextensibility constraint yields the equation
\begin{equation}\label{eq:tauY_eqn}
  \big(({\bf I}+\Y_s\otimes\Y_s)(\tau_Y\Y_s)_s \big)_s\cdot\Y_s =\big(({\bf I}+\Y_s\otimes\Y_s)\Y_{ssss}\big)_s \cdot\Y_s\,.
\end{equation}
We further note that, by \cite[Lemma 3.3]{albritton2025rods}, $\tau_Y$ satisfying \eqref{eq:tauY_eqn} obeys the bounds
\begin{equation}\label{eq:tauYbd}
\begin{aligned}
  \norm{\tau_Y}_{H^1} &\le c\norm{\Y_{ss}}_{H^1}^2 \,, \\
  \norm{\tau_Y}_{H^2}&\le c(\norm{\Y_s}_{H^2}^4 + \norm{\Y_s}_{H^2}\norm{\Y_s}_{H^3})\,.
\end{aligned}
\end{equation}
We may immediately estimate the terms involving $\tau_Y$ in $\bm{A}^\parallel$ and $\bm{A}^\perp$ of \eqref{eq:Wenergy1} as   
\begin{equation}\label{eq:tauYterm_est}
\begin{aligned}
  \norm{T_{m_\epsilon^{\rm t}}^{1/2}\big(\P_{\X_s}(\tau_Y\bm{W}_s)_s\big)}_{L^2} &\le c\abs{\log\epsilon}^{1/2}\norm{\P_{\X_s}(\tau_Y\bm{W}_s)_s}_{L^2}
  \le c\abs{\log\epsilon}^{1/2}\norm{(\tau_Y)_s}_{L^\infty}\norm{\bm{W}_s}_{H^1}\\
  &\le c\abs{\log\epsilon}^{1/2}(\norm{\Y_s}_{H^2}^4 + \norm{\Y_s}_{H^2}\norm{\Y_s}_{H^3})\norm{\bm{W}_s}_{H^1}\,,
\end{aligned}
\end{equation}
with an identical bound for the normal direction term $\norm{T_{m_\epsilon^{\rm n}}^{1/2}\big(\P_{\X_s}^\perp(\tau_Y\bm{W}_s)_s\big)}_{L^2}$.

We will additionally require a bound for the terms in $\bm{A}^\parallel$ and $\bm{A}^\perp$ involving the difference $\tau_W=\tau_X-\tau_Y$. We show the following.
\begin{lemma}[Tension difference]\label{lem:tensionXY_diff}
  The difference $\tau_W=\tau_X-\tau_Y$ between the tensions corresponding to curves $\X$ and $\Y$, respectively, satisfies the bound 
\begin{equation}\label{eq:tauWbd}
\begin{aligned}
  \norm{T_{m_\epsilon^{\rm t}}^{1/2}((\tau_W)_s\X_s)}_{L^2} &+ 
  \norm{T_{m_\epsilon^{\rm n}}^{1/2}(\tau_W\X_{ss})}_{L^2}
  \le c\bigg(\abs{\log\epsilon}^{1/2}\norm{\bm{W}_s}_{H^1}(1+\norm{\Y_s}_{H^3}^2)\\
  & + \bigg(\norm{T_{m_\epsilon^{\rm t}}^{1/2}\bm{W}_{ssss}}_{L^2}+ 1+\norm{\Y_s}_{H^3}^2\bigg)\big(1+\norm{\X_s}_{H^1}+\norm{\Y_s}_{H^3}^2\big)\bigg)\,,
\end{aligned}
\end{equation}
where $\bm{W}=\X-\Y$.
\end{lemma}

\begin{proof}
For any $\varphi\in H^{1/2}(\T)$, defining the bilinear forms
\begin{equation}
\begin{aligned}
  \mc{B}_\epsilon^X(\tau_X,\varphi) &= \int_\T \overline{\mc{L}}_\epsilon(\X)[(\tau_X\X_s)_s] \cdot(\varphi\X_s)_s\,ds \,,\\
  \mc{B}_\epsilon^Y(\tau_Y,\varphi) &= \frac{\abs{\log\epsilon}}{4\pi}\int_\T \big(({\bf I}+\Y_s\otimes\Y_s)(\tau_Y\Y_s)_s\big) \cdot(\varphi\Y_s)_s\,ds\,,
\end{aligned}
\end{equation}
we may write the difference between the weak form of the equations satisfied by $\tau_X$ and $\tau_Y$ as 
\begin{equation}\label{eq:diffXY_weak}
\begin{aligned}
  &\mc{B}_\epsilon^X(\tau_X,\varphi) - \mc{B}_\epsilon^Y(\tau_Y,\varphi) \\
  &\qquad = \underbrace{\int_\T \bigg(\overline{\mc{L}}_\epsilon(\X)[\X_{ssss}]\cdot(\varphi\X_s)_s - \frac{\abs{\log\epsilon}}{4\pi}\big(({\bf I}+\Y_s\otimes\Y_s)\Y_{ssss}\big)\cdot(\varphi\Y_s)_s\bigg)\,ds }_{{\rm RHS}}\,.
\end{aligned}
\end{equation}

The left hand side of \eqref{eq:diffXY_weak} may be written as 
\begin{equation}
\begin{aligned}
  \mc{B}_\epsilon^X(\tau_X,\varphi) &- \mc{B}_\epsilon^Y(\tau_Y,\varphi)= B^\parallel +B^\perp \,,\\
  B^\parallel &= \int_\T \bigg(\big(\P_{\X_s}T_{m_\epsilon^{\rm t}}((\tau_X)_s\X_s)\big)\cdot(\varphi\X_s)_s -\frac{\abs{\log\epsilon}}{2\pi}(\tau_Y)_s\Y_s\cdot(\varphi\Y_s)_s \bigg)\,ds \\
  B^\perp &= \int_\T \bigg( \big(\P_{\X_s}^\perp T_{m_\epsilon^{\rm n}}(\tau_X\X_{ss})\big)\cdot(\varphi\X_s)_s  -\frac{\abs{\log\epsilon}}{4\pi}\tau_Y\Y_{ss} \cdot(\varphi\Y_s)_s \bigg)\,ds\,.
\end{aligned}
\end{equation}
We may further write 
\begin{equation}
\begin{aligned}
  B^\parallel &= \int_\T \bigg(\big(\X_s\cdot T_{m_\epsilon^{\rm t}}((\tau_X)_s\X_s)\big) -\frac{\abs{\log\epsilon}}{2\pi}(\tau_Y)_s \bigg)\varphi_s\,ds \\
  &= \int_\T \bigg(\big(\X_s\cdot T_{m_\epsilon^{\rm t}}((\tau_W)_s\X_s)\big) + \X_s\cdot\bigg(T_{m_\epsilon^{\rm t}}-\frac{\abs{\log\epsilon}}{2\pi}\bigg)\big[(\tau_Y)_s\X_s\big] \bigg)\varphi_s\,ds
\end{aligned}
\end{equation}
as well as 
\begin{equation}
\begin{aligned}
 B^\perp &= \int_\T \bigg( \X_{ss}\cdot T_{m_\epsilon^{\rm n}}(\tau_X\X_{ss})  -\frac{\abs{\log\epsilon}}{4\pi}\tau_Y\abs{\Y_{ss}}^2 \bigg)\varphi\,ds \\
  &= \int_\T \bigg( \X_{ss}\cdot T_{m_\epsilon^{\rm n}}(\tau_W\X_{ss}) + \X_{ss}\cdot\bigg(T_{m_\epsilon^{\rm n}} -\frac{\abs{\log\epsilon}}{4\pi}\bigg)[\tau_Y\X_{ss}] \\
  &\qquad\quad + \frac{\abs{\log\epsilon}}{4\pi}\tau_Y(\abs{\X_{ss}}^2-\abs{\Y_{ss}}^2) \bigg)\varphi\,ds \,.
\end{aligned}
\end{equation}
We may thus rewrite the left hand side of \eqref{eq:diffXY_weak} as 
\begin{equation}
\begin{aligned}
  \mc{B}_\epsilon^X(\tau_X,\varphi) &- \mc{B}_\epsilon^Y(\tau_Y,\varphi)= \mc{B}_\epsilon^X(\tau_W,\varphi) + {\rm LHS}\,,\\
  {\rm LHS}&= \int_\T \bigg( \varphi_s\X_s\cdot\bigg(T_{m_\epsilon^{\rm t}}-\frac{\abs{\log\epsilon}}{2\pi}\bigg)\big[(\tau_Y)_s\X_s\big] \\
  &\quad+ \varphi\X_{ss}\cdot\bigg(T_{m_\epsilon^{\rm n}} -\frac{\abs{\log\epsilon}}{4\pi}\bigg)[\tau_Y\X_{ss}]  + \frac{\abs{\log\epsilon}}{4\pi}\tau_Y(\abs{\X_{ss}}^2-\abs{\Y_{ss}}^2)\varphi \bigg)\,ds \,.
\end{aligned}
\end{equation}
Using Lemma \ref{lem:mult_RFT_diff} and the $\tau_Y$ bound \eqref{eq:tauYbd}, we may bound the terms appearing in LHS as 
\begin{equation}\label{eq:LHS_diffbd}
\begin{aligned}
  \abs{{\rm LHS}} &\le c\bigg( \norm{T_{m_\epsilon^{\rm t}}^{1/2}(\varphi_s\X_s)}_{L^2}\norm{(\tau_Y)_s\X_s}_{L^2} + 
  \norm{T_{m_\epsilon^{\rm n}}^{1/2}(\varphi\X_{ss})}_{L^2}\norm{\tau_Y\X_{ss}}_{L^2}\\
  &\qquad + \abs{\log\epsilon}\norm{\bm{W}_s}_{H^1}\norm{\tau_Y}_{L^\infty}(\norm{\bm{W}_{ss}}_{L^\infty}+\norm{\Y_{ss}}_{L^\infty})\norm{\varphi}_{L^2} \bigg) \\
  &\le c\bigg( \norm{T_{m_\epsilon^{\rm t}}^{1/2}(\varphi_s\X_s)}_{L^2}\norm{\Y_{ss}}_{H^1}^2 + 
  \norm{T_{m_\epsilon^{\rm n}}^{1/2}(\varphi\X_{ss})}_{L^2}\norm{\Y_{ss}}_{H^1}^2\norm{\X_s}_{H^1}\\
  &\qquad + \abs{\log\epsilon}\norm{\bm{W}_s}_{H^1}\norm{\Y_{ss}}_{H^1}^2(\norm{\bm{W}_{sss}}_{\dot H^{1/2}}^{2/3}\norm{\bm{W}_s}_{H^1}^{1/3}+\norm{\Y_{ss}}_{H^1})\times\\
  &\qquad\qquad\times\norm{T_{m_\epsilon^{\rm n}}^{-1}T_{m_\epsilon^{\rm n}}(\varphi \X_s)\cdot\X_s}_{L^2} \bigg)  \\
  &\le c\norm{\Y_{ss}}_{H^1}^2\bigg( \norm{T_{m_\epsilon^{\rm t}}^{1/2}(\varphi_s\X_s)}_{L^2} + 
  \norm{T_{m_\epsilon^{\rm n}}^{1/2}(\varphi\X_{ss})}_{L^2}\norm{\X_s}_{H^1}\\
  &\qquad + \norm{\bm{W}_s}_{H^1}(\norm{\bm{W}_{sss}}_{\dot H^{1/2}}^{2/3}\norm{\bm{W}_s}_{H^1}^{1/3}+\norm{\Y_{ss}}_{H^1})\norm{T_{m_\epsilon^{\rm n}}(\varphi \X_s)_s}_{L^2} \bigg) \\
  &\le c\bigg( \norm{T_{m_\epsilon^{\rm t}}^{1/2}(\varphi_s\X_s)}_{L^2} + 
  \norm{T_{m_\epsilon^{\rm n}}^{1/2}(\varphi\X_{ss})}_{L^2}\bigg)\norm{\Y_{ss}}_{H^1}^2\big(1 + \norm{\X_s}_{H^1}^2 + \norm{\Y_{ss}}_{H^1}^2 \\
  &\qquad\qquad + \abs{\log\epsilon}^{1/2}\norm{\bm{W}_s}_{H^1}^{4/3}\norm{\bm{W}_{sss}}_{\dot H^{1/2}}^{2/3}\big)\,.
\end{aligned}
\end{equation}
Here we have also exploited Lemma \ref{lem:Tinverse} to bound $T_{m_\epsilon^{\rm n}}^{-1}$, and have used that we may opt not to use the regularity gain from $T_{m_\epsilon^{\rm n}}^{1/2}$ to avoid an additional factor of $\epsilon^{-1/2}$ from the high wavenumber scaling of $m_\epsilon^{\rm n}(k)$. We have further used that $m_\epsilon^{\rm n}(k)\lesssim m_\epsilon^{\rm t}(k)\lesssim m_\epsilon^{\rm n}(k)$, so we may interchange $T_{m_\epsilon^{\rm t}}$ and $T_{m_\epsilon^{\rm n}}$ in the above bounds.

We next turn to the right hand side terms RHS of \eqref{eq:diffXY_weak}, which we may write as 
\begin{equation}
\begin{aligned}
  {\rm RHS} &= R^\parallel + R^\perp \,,\\
  R^\parallel &= \int_\T \bigg(\X_s\cdot T_{m_\epsilon^{\rm t}}(\X_s\cdot\X_{ssss}\X_s) - \frac{\abs{\log\epsilon}}{2\pi}\Y_s\cdot\Y_{ssss}\bigg)\varphi_s\,ds \\
  R^\perp &= \int_\T \bigg(\X_{ss}\cdot T_{m_\epsilon^{\rm n}}(\P_{\X_s}^\perp\X_{ssss}) - \frac{\abs{\log\epsilon}}{4\pi}\Y_{ssss}\cdot\Y_{ss} \bigg)\varphi\,ds \,.
\end{aligned}
\end{equation}
Using the inextensibility identities \eqref{eq:inex_ids}, we may write
\begin{equation}
\begin{aligned}
  R^\parallel &= -3\int_\T \varphi_s\X_s\cdot T_{m_\epsilon^{\rm t}}\big( (\X_{ss}\cdot\bm{W}_{sss}+\bm{W}_{ss}\cdot\Y_{sss})\X_s\big) \,ds\\
  &\qquad -3\int_\T \varphi_s\X_s\cdot \bigg(T_{m_\epsilon^{\rm t}}- \frac{\abs{\log\epsilon}}{2\pi}\bigg)\big[\X_s \Y_{ss}\cdot\Y_{sss}\big]\,ds\\
  &= -3\int_\T \varphi_s\X_s\cdot T_{m_\epsilon^{\rm t}}\bigg( \frac{1}{2}\p_s\abs{\bm{W}_{ss}}^2\X_s+\p_s(\Y_{ss}\cdot\bm{W}_{ss})\X_s\bigg) \,ds\\
  &\qquad -\frac{3}{2}\int_\T \varphi_s\X_s\cdot \bigg(T_{m_\epsilon^{\rm t}}- \frac{\abs{\log\epsilon}}{2\pi}\bigg)\big[\p_s\abs{\Y_{ss}}^2\X_s\big]\,ds\,.
\end{aligned}
\end{equation}
Using Lemma \ref{lem:mult_RFT_diff} and the interpolation inequalities \eqref{eq:interps}, we may estimate $R^\parallel$ as 
\begin{equation}\label{eq:Rpar_bd}
\begin{aligned}
  \abs{R^\parallel} &\le c \norm{T_{m_\epsilon^{\rm t}}^{1/2}(\varphi_s\X_s)}_{L^2}
  \bigg(\norm{T_{m_\epsilon^{\rm t}}^{1/2}(\p_s\abs{\bm{W}_{ss}}^2\X_s)}_{L^2}+\norm{T_{m_\epsilon^{\rm t}}^{1/2}(\p_s(\bm{W}_{ss}\cdot\Y_{ss})\X_s)}_{L^2}\\
  &\hspace{5cm} + \norm{\X_s \Y_{ss}\cdot\Y_{sss}}_{L^2}\bigg)\\
  &\le c\norm{T_{m_\epsilon^{\rm t}}^{1/2}(\varphi_s\X_s)}_{L^2}
  \bigg(\abs{\log\epsilon}^{1/2}\norm{\abs{\bm{W}_{ss}}^2}_{\dot H^1}
  +\abs{\log\epsilon}^{1/2}\norm{\bm{W}_{ss}\cdot\Y_{ss}}_{\dot H^1}+ \norm{\abs{\Y_{ss}}^2}_{\dot H^1}\bigg)\\
  &\le c\norm{T_{m_\epsilon^{\rm t}}^{1/2}(\varphi_s\X_s)}_{L^2}
  \bigg(\abs{\log\epsilon}^{1/2}\norm{\bm{W}_{sss}}_{\dot H^{1/2}}(\norm{\bm{W}_s}_{H^1}+\norm{\Y_s}_{H^3})\\
  &\hspace{4cm}
  +\abs{\log\epsilon}^{1/2}\norm{\bm{W}_s}_{H^1}\norm{\Y_s}_{H^3}
  + \norm{\Y_{ss}}_{H^1}^2\bigg)\\
  &\le c\norm{T_{m_\epsilon^{\rm t}}^{1/2}(\varphi_s\X_s)}_{L^2}
  \bigg(\norm{T_{m_\epsilon^{\rm t}}^{1/2}\bm{W}_{ssss}}_{L^2}(\norm{\bm{W}_s}_{H^1}+\norm{\Y_s}_{H^3}) + \norm{\Y_{ss}}_{H^1}^2\\
  &\hspace{4cm}
  +\abs{\log\epsilon}^{1/2}\norm{\bm{W}_s}_{H^1}\norm{\Y_s}_{H^3}
  \bigg)\,.
\end{aligned}
\end{equation}
Here, as in the estimate \eqref{eq:LHS_diffbd} for LHS, we have opted to forego the regularity gain from $T_{m_\epsilon^{\rm t}}^{1/2}$ in order to avoid an additional factor of $\epsilon^{-1/2}$ from the high wavenumber behavior of $m_\epsilon^{\rm t}(k)$. Furthermore, in the last line, we have used that 
\begin{equation}
\norm{\bm{W}_{sss}}_{\dot H^{1/2}}=\|T_{m_\epsilon^{\rm t}}^{-1/2}T_{m_\epsilon^{\rm t}}^{1/2}\bm{W}_{sss}\|_{\dot H^{1/2}}\le c\abs{\log\epsilon}^{-1/2}\|T_{m_\epsilon^{\rm t}}^{1/2}\bm{W}_{sss}\|_{\dot H^1}\,,
\end{equation}
by Lemma \ref{lem:Tinverse}. 

We may similarly write $R^\perp$ as 
\begin{equation}
\begin{aligned}
  R^\perp &= \int_\T \bigg(\X_{ss}\cdot T_{m_\epsilon^{\rm n}}(\X_{ssss}+3\X_{ss}\cdot\X_{sss}\X_s) - \frac{\abs{\log\epsilon}}{4\pi}(\Y_{ssss}+3\Y_{ss}\cdot\Y_{sss}\Y_s)\cdot\Y_{ss} \bigg)\varphi\,ds \\
  &= \int_\T \varphi\X_{ss}\cdot T_{m_\epsilon^{\rm n}}\big(\bm{W}_{ssss}+3\X_{ss}\cdot\bm{W}_{sss}\X_s +3\bm{W}_{ss}\cdot\Y_{sss}\X_s +3\Y_{ss}\cdot\Y_{sss}\bm{W}_s\big)\,ds\\
  &\quad +\int_\T \varphi\bigg(\X_{ss}\cdot\bigg(T_{m_\epsilon^{\rm n}}-\frac{\abs{\log\epsilon}}{4\pi}\bigg)+ \frac{\abs{\log\epsilon}}{4\pi}\bm{W}_{ss}\cdot \bigg)[\Y_{ssss}+3\Y_{ss}\cdot\Y_{sss}\Y_s] \,ds\,.
\end{aligned}
\end{equation}
Again using Lemma \ref{lem:mult_RFT_diff} and the inequalities \eqref{eq:interps}, we may estimate 
\begin{equation}
\begin{aligned}
  \abs{R^\perp} &\le c\norm{T_{m_\epsilon^{\rm n}}^{1/2}(\varphi\X_{ss})}_{L^2}\bigg(\norm{T_{m_\epsilon^{\rm n}}^{1/2}\bm{W}_{ssss}}_{L^2}+\norm{T_{m_\epsilon^{\rm n}}^{1/2}(\p_s\abs{\bm{W}_{ss}}^2\X_s)}_{L^2}\\
  &\qquad\qquad  +\norm{T_{m_\epsilon^{\rm n}}^{1/2}(\p_s(\bm{W}_{ss}\cdot\Y_{ss})\X_s)}_{L^2} +\norm{T_{m_\epsilon^{\rm n}}^{1/2}(\p_s\abs{\Y_{ss}}^2\bm{W}_s)}_{L^2}\bigg) \\
  &\qquad +c\bigg(\norm{T_{m_\epsilon^{\rm n}}^{1/2}(\varphi\X_{ss})}_{L^2}+ \abs{\log\epsilon}\norm{\varphi}_{L^2}\norm{\bm{W}_{ss}}_{L^\infty}\bigg)\big( \norm{\Y_s}_{H^3}+ \norm{\Y_{ss}}_{H^1}^2\big) \\
  &\le c\norm{T_{m_\epsilon^{\rm n}}^{1/2}(\varphi\X_{ss})}_{L^2}\bigg(\norm{T_{m_\epsilon^{\rm n}}^{1/2}\bm{W}_{ssss}}_{L^2}(1+\norm{\bm{W}_s}_{H^1}+\norm{\Y_s}_{H^3}) \\
  &\qquad\qquad +(1+\abs{\log\epsilon}^{1/2}\norm{\bm{W}_s}_{H^1})(\norm{\Y_s}_{H^3}+\norm{\Y_{ss}}_{H^1}^2) \bigg) \\
  &\qquad + c\abs{\log\epsilon}\norm{T_{m_\epsilon^{\rm n}}^{-1}T_{m_\epsilon^{\rm n}}(\varphi\X_s)\cdot\X_s}_{L^2}\norm{\bm{W}_{sss}}_{\dot H^{1/2}}^{2/3}\norm{\bm{W}_s}_{H^1}^{1/3}\big( \norm{\Y_s}_{H^3}+ \norm{\Y_{ss}}_{H^1}^2\big)\\
   &\le c\norm{T_{m_\epsilon^{\rm n}}^{1/2}(\varphi\X_{ss})}_{L^2}\bigg(\norm{T_{m_\epsilon^{\rm n}}^{1/2}\bm{W}_{ssss}}_{L^2}(1+\norm{\bm{W}_s}_{H^1}+\norm{\Y_s}_{H^3}) \\
  &\qquad\qquad +(1+\abs{\log\epsilon}^{1/2}\norm{\bm{W}_s}_{H^1})(\norm{\Y_s}_{H^3}+\norm{\Y_{ss}}_{H^1}^2) \bigg) \\
  &\qquad + c\abs{\log\epsilon}^{1/2}\norm{T_{m_\epsilon^{\rm n}}^{1/2}(\varphi\X_s)_s}_{L^2}\norm{\bm{W}_{sss}}_{\dot H^{1/2}}^{2/3}\norm{\bm{W}_s}_{H^1}^{1/3}\big( \norm{\Y_s}_{H^3}+ \norm{\Y_{ss}}_{H^1}^2\big)\,.
\end{aligned}
\end{equation}
Here we have reused some of the estimates from the LHS bound~\eqref{eq:LHS_diffbd} and the $R^\parallel$ bound~\eqref{eq:Rpar_bd}.

Altogether, taking $\varphi=\tau_W$ in \eqref{eq:diffXY_weak} and combining the bounds for LHS and RHS, we may estimate 
\begin{equation}\label{eq:BXtautau}
\begin{aligned}
  &\abs{\bm{B}_\epsilon^X(\tau_W,\tau_W)}\le \abs{\rm LHS} + \abs{\rm RHS}\\
  &\le c\bigg( \norm{T_{m_\epsilon^{\rm t}}^{1/2}((\tau_W)_s\X_s)}_{L^2} + 
  \norm{T_{m_\epsilon^{\rm n}}^{1/2}(\tau_W\X_{ss})}_{L^2}\bigg)\bigg(\abs{\log\epsilon}^{1/2}\norm{\bm{W}_s}_{H^1}(\norm{\Y_s}_{H^3}+\norm{\Y_s}_{H^2}^2)\\
  &\qquad  + \abs{\log\epsilon}^{1/2}\norm{\bm{W}_s}_{H^1}^{4/3}\norm{\bm{W}_{sss}}_{\dot H^{1/2}}^{2/3}(\norm{\Y_s}_{H^3}+\norm{\Y_{ss}}_{H^1}^2)\\
  &\qquad\quad + \norm{T_{m_\epsilon^{\rm t}}^{1/2}\bm{W}_{ssss}}_{L^2}(1+\norm{\bm{W}_s}_{H^1}+\norm{\Y_s}_{H^3})\\
  &\qquad\qquad + (\norm{\Y_s}_{H^3}+\norm{\Y_{ss}}_{H^1}^2)\big(1 + \norm{\X_s}_{H^1}^2 + \norm{\Y_{ss}}_{H^1}^2\big)\bigg)\\
  &\le c\bigg( \norm{T_{m_\epsilon^{\rm t}}^{1/2}((\tau_W)_s\X_s)}_{L^2} + 
  \norm{T_{m_\epsilon^{\rm n}}^{1/2}(\tau_W\X_{ss})}_{L^2}\bigg)\bigg(\abs{\log\epsilon}^{1/2}\norm{\bm{W}_s}_{H^1}(\norm{\Y_s}_{H^3}+\norm{\Y_s}_{H^2}^2)\\
  &\qquad\qquad  + \norm{T_{m_\epsilon^{\rm t}}^{1/2}\bm{W}_{ssss}}_{L^2}(1+\norm{\X_s}_{H^1}+\norm{\Y_s}_{H^3}+\norm{\Y_{ss}}_{H^1}^2)\\
  &\qquad\qquad\quad + (\norm{\Y_s}_{H^3}+\norm{\Y_{ss}}_{H^1}^2)\big(1 + \norm{\X_s}_{H^1}^2 + \norm{\Y_{ss}}_{H^1}^2\big)\bigg)\,.
\end{aligned}
\end{equation}
Here, to simplify, we have used that $m_\epsilon^{\rm t}(k)\lesssim m_\epsilon^{\rm n}(k) \lesssim m_\epsilon^{\rm t}(k)$, so we may interchange $T_{m_\epsilon^{\rm t}}$ and $T_{m_\epsilon^{\rm n}}$. In addition, we have replaced one instance of $\norm{\bm{W}_s}_{H^1}\le \norm{\X_s}_{H^1}+\norm{\Y_s}_{H^1}$, and used that
\begin{equation}
\begin{aligned}
  \norm{\bm{W}_{sss}}_{\dot H^{1/2}}^{2/3}\norm{\bm{W}_s}_{H^1}^{1/3}&\le c\abs{\log\epsilon}^{-1/3}\norm{T_{m_\epsilon^{\rm t}}\bm{W}_{ssss}}_{L^2}^{2/3}\norm{\bm{W}_s}_{H^1}^{1/3}\\
  &\le c(\abs{\log\epsilon}^{-1/2}\norm{T_{m_\epsilon^{\rm t}}\bm{W}_{ssss}}_{L^2} + \norm{\bm{W}_s}_{H^1})\,.
\end{aligned}
\end{equation} 
Finally, using that 
\begin{equation}
  \bm{B}_\epsilon^X(\tau_W,\tau_W) = \norm{T_{m_\epsilon^{\rm t}}^{1/2}((\tau_W)_s\X_s)}_{L^2}^2 + 
  \norm{T_{m_\epsilon^{\rm n}}^{1/2}(\tau_W\X_{ss})}_{L^2}^2\,,
\end{equation}
applying Young's inequality to \eqref{eq:BXtautau} then yields \eqref{eq:tauWbd}.
\end{proof}

\subsection{Energy bound for curve difference}
Using the results of sections \ref{subsec:OpDiff} and \ref{subsec:tauXYDiff}, we may proceed to estimate the terms $\bm{A}^\parallel$ and $\bm{A}^\perp$ appearing on the right hand side of the energy identity~\eqref{eq:Wenergy1}.

We may first apply Corollary \ref{cor:mult_RFT_diff} to estimate $\bm{A}^\parallel$ as 
\begin{equation}
\begin{aligned}
  \abs{\bm{A}^\parallel} &\le \norm{T_{m_\epsilon^{\rm t}}^{1/2}(\P_{\X_s}\bm{W}_{ssss})}_{L^2}\bigg(\norm{T_{m_\epsilon^{\rm t}}^{1/2}\big((\tau_W)_s\X_s\big)}_{L^2}+\norm{T_{m_\epsilon^{\rm t}}^{1/2}\big(\P_{\X_s}(\tau_Y\bm{W}_s)_s\big)}_{L^2} \bigg) \\
  &\quad + c\norm{T_{m_\epsilon^{\rm t}}^{1/2}(\P_{\X_s}\bm{W}_{ssss})}_{L^2}\norm{\P_{\X_s}\big(\Y_{sss}-\tau_Y\Y_s \big)_s}_{L^2}\\
  &\qquad + c\abs{\log\epsilon}^{1/2}\norm{\bm{W}_s}_{H^1}(\norm{\X_s}_{H^1}^2+\norm{\Y_s}_{H^1}^2)\norm{T_{m_\epsilon^{\rm t}}^{1/2}\bm{W}_{ssss}}_{L^2}\norm{\Y_{sss}-\tau_Y\Y_s }_{\dot H^{3/2}}\,,
\end{aligned}
\end{equation}
and $\bm{A}^\perp$ as 
\begin{equation}
\begin{aligned}
  \abs{\bm{A}^\perp} &\le \norm{T_{m_\epsilon^{\rm n}}^{1/2}(\P_{\X_s}^\perp\bm{W}_{ssss})}_{L^2}\bigg( \norm{T_{m_\epsilon^{\rm n}}^{1/2}(\tau_W\X_{ss})}_{L^2}
  + \norm{T_{m_\epsilon^{\rm n}}^{1/2}\big(\P_{\X_s}^\perp(\tau_Y\bm{W}_s)_s \big)}_{L^2}\bigg) \\
  &\quad + c\norm{T_{m_\epsilon^{\rm n}}^{1/2}(\P_{\X_s}^\perp\bm{W}_{ssss})}_{L^2}\norm{\P_{\X_s}^\perp\big(\Y_{sss}-\tau_Y\Y_s \big)_s}_{L^2}\\
  &\quad + c\abs{\log\epsilon}^{1/2}\norm{\bm{W}_s}_{H^1}(1+\norm{\X_s}_{H^1}^2+\norm{\Y_s}_{H^1}^2)\norm{T_{m_\epsilon^{\rm n}}^{1/2}\bm{W}_{ssss}}_{L^2}\norm{\Y_{sss}-\tau_Y\Y_s}_{\dot H^{3/2}}\,.
\end{aligned}
\end{equation}
Together, using Lemma \ref{lem:tensionXY_diff} and the bound \eqref{eq:tauYterm_est} for the tension terms involving $\tau_Y$, we obtain the bound
\begin{equation}\label{eq:AparAperp}
\begin{aligned}
  \abs{\bm{A}^\parallel}+\abs{\bm{A}^\perp} &\le c\bigg(\norm{T_{m_\epsilon^{\rm t}}^{1/2}(\P_{\X_s}\bm{W}_{ssss})}_{L^2}+\norm{T_{m_\epsilon^{\rm n}}^{1/2}(\P_{\X_s}^\perp\bm{W}_{ssss})}_{L^2}\bigg)\times\\
  &\quad\times \bigg(\abs{\log\epsilon}^{1/2}\norm{\bm{W}_s}_{H^1}(1+\norm{\X_s}_{H^1}^2+\norm{\Y_s}_{H^1}^2)(1+\norm{\Y_s}_{H^3}^5+ \norm{\Y_s}_{H^{7/2}}) \\
  &\qquad + \bigg(\norm{T_{m_\epsilon^{\rm t}}^{1/2}\bm{W}_{ssss}}_{L^2}+ 1+\norm{\Y_s}_{H^3}^2\bigg)\big(1+\norm{\X_s}_{H^1}+\norm{\Y_s}_{H^3}^2\big)\bigg)\,.
\end{aligned}
\end{equation}
Now, recalling the form of the energy identity \eqref{eq:Wenergy1} under the time rescaling \eqref{eq:t_rescale}, we have
\begin{equation}\label{eq:Wenergy2}
\begin{aligned}
  \abs{\log\epsilon}\frac{\p}{\p \underline{t}}\int_\T\abs{\bm{W}_{ss}}^2\,ds+\norm{T_{m_\epsilon^{\rm t}}^{1/2}(\P_{\X_s}\bm{W}_{ssss})}_{L^2_s}^2  
  +\norm{T_{m_\epsilon^{\rm n}}^{1/2}(\P_{\X_s}^\perp\bm{W}_{ssss})}_{L^2_s}^2 \le \abs{\bm{A}^\parallel} +\abs{\bm{A}^\perp} \,.
\end{aligned}
\end{equation}
Applying Young's inequality to \eqref{eq:AparAperp} thus yields the energy bound 
\begin{equation}\label{eq:Wenergy3}
\begin{aligned}
  &\frac{\p}{\p \underline{t}}\int_\T\abs{\bm{W}_{ss}}^2\,ds+\norm{T_{m_\epsilon^{\rm t}}^{1/2}(\P_{\X_s}\bm{W}_{ssss})}_{L^2_s}^2  
  +\norm{T_{m_\epsilon^{\rm n}}^{1/2}(\P_{\X_s}^\perp\bm{W}_{ssss})}_{L^2_s}^2\\
  &\quad \le c\bigg(\norm{\bm{W}_s}_{H^1_s}^2(1+\norm{\X_s}_{H^1_s}^2+\norm{\Y_s}_{H^1_s}^2)^2(1+\norm{\Y_s}_{H^3_s}^5+ \norm{\Y_s}_{H^{7/2}_s})^2 \\
  &\qquad + \abs{\log\epsilon}^{-1}\bigg(\norm{T_{m_\epsilon^{\rm t}}^{1/2}\bm{W}_{ssss}}_{L^2_s}^2+ (1+\norm{\Y_s}_{H^3_s}^2)^2\bigg)\big(1+\norm{\X_s}_{H^1_s}+\norm{\Y_s}_{H^3_s}^2\big)^2\bigg)\,.
\end{aligned}
\end{equation}

We may define the energy and dissipation quantities
\begin{equation}
  E_W(\underline{t}) = \norm{\bm{W}_{ss}}_{L^2_s}^2\,, \quad 
  D_W(\underline{t}) = \norm{T_{m_\epsilon^{\rm t}}^{1/2}(\P_{\X_s}\bm{W}_{ssss})}_{L^2_s}^2  
  +\norm{T_{m_\epsilon^{\rm n}}^{1/2}(\P_{\X_s}^\perp\bm{W}_{ssss})}_{L^2_s}^2\,,
\end{equation}
and note that $\norm{\bm{W}_s}_{H^1_s}^2\lesssim E_W(\underline{t})$ since $\bm{W}_s$ has zero mean on $\T$. We further define the quantities 
\begin{equation}
\begin{aligned}
  C_{XY1}(\underline{t}) &= c(1+\norm{\X_s}_{H^1_s}^2+\norm{\Y_s}_{H^1_s}^2)^2(1+\norm{\Y_s}_{H^3_s}^5+ \norm{\Y_s}_{H^{7/2}_s})^2\,, \\
  C_{XY2}(\underline{t}) &= c\big(1+\norm{\X_s}_{H^1_s}+\norm{\Y_s}_{H^3_s}^2\big)^2\,, \\
  F_{XY}(\underline{t}) &= c(1+\norm{\Y_s}_{H^3_s}^2)^2\big(1+\norm{\X_s}_{H^1_s}+\norm{\Y_s}_{H^3_s}^2\big)^2\,.
\end{aligned}
\end{equation}
By \eqref{eq:energy_ineq}, we have that $\norm{\X_s}_{H^1_s}$ is uniformly bounded in time, independent of $\epsilon$. Furthermore, since $\Y\in C_tH^4_s\cap L^2_{t,{\rm loc}}H^6_s$, we have $C_{XY1}\in L^1_t$ while $C_{XY2}$, $F_{XY}\in L^\infty_t$ with $\epsilon$-independent bounds, as the time-rescaled evolution for $\bm{Y}$ no longer depends on $\epsilon$.

We may thus rewrite \eqref{eq:Wenergy3} as 
\begin{equation}
  \p_{\underline{t}} E_W(\underline{t}) + D_W(\underline{t})
  \le C_{XY1}(\underline{t}) E_W(\underline{t}) + \abs{\log\epsilon}^{-1}C_{XY2}(\underline{t})D_W(\underline{t}) + \abs{\log\epsilon}^{-1}F_{XY}(\underline{t})\,.
\end{equation}
Then, given a fixed time interval $[0,T]$, for $\epsilon$ sufficiently small, we may absorb $\abs{\log\epsilon}^{-1}C_{XY2}D_W(\underline{t})$ into the left hand side terms to obtain 
\begin{equation}\label{eq:diff_ineq}
  \p_{\underline{t}} E_W(\underline{t}) + \frac{1}{2}D_W(\underline{t})
  \le C_{XY1}(\underline{t}) E_W(\underline{t})  + \abs{\log\epsilon}^{-1}F_{XY}(\underline{t})\,,
  \qquad E_W(0)=0\,.
\end{equation}
For every $T<\infty$, a Gr\"onwall inequality applied to \eqref{eq:diff_ineq} then yields that there exists $C_W=C_W(T,\norm{\X}_{L^\infty_tH^2_s},\norm{\Y}_{L^\infty_tH^4_s\cap L^2_tH^6_s})$ independent of $\epsilon$ such that  
\begin{equation}\label{eq:EW_DW}
\begin{aligned}
  E_W(\underline{t})+ \int_0^{\underline{t}}D_W(\underline{t}')d\underline{t}' &\le \abs{\log\epsilon}^{-1} \,C_W(T,\norm{\X}_{L^\infty_tH^2_s},\norm{\Y}_{L^\infty_tH^4_s\cap L^2_tH^6_s}) \,.
\end{aligned}
\end{equation}
In particular, on any fixed time interval, the left hand side goes to zero as $\epsilon\to 0$.

It remains to control the mean of $\bm{W}$, which we denote by $\bm{W}_0$. Using equation \eqref{eq:Wevolution1} under the time rescaling \eqref{eq:t_rescale}, we have
\begin{equation}
\begin{aligned}
  \frac{\p\bm{W}_0}{\p \underline{t}} &= -\abs{\log\epsilon}^{-1}\int_\T\overline{\mc{L}}_\epsilon(\X)[(\X_{sss}-\tau_X\X_s)_s]\,ds
  + \frac{1}{4\pi}\int_\T ({\bf I}+\Y_s\otimes\Y_s)(\Y_{sss}-\tau_Y\Y_s)_s\,ds\\
  &= -\abs{\log\epsilon}^{-1}\underbrace{\int_\T\overline{\mc{L}}_\epsilon(\X)[(\bm{W}_{sss}-\tau_W\X_s-\tau_Y\bm{W}_s)_s]\,ds}_{\bm{J}_1}\\
  &\quad -\abs{\log\epsilon}^{-1}\underbrace{\int_\T\bigg(\overline{\mc{L}}_\epsilon(\X)-\frac{1}{4\pi}({\bf I}+\Y_s\otimes\Y_s)\bigg)[(\Y_{sss} -\tau_Y\Y_s)_s]\,ds}_{\bm{J}_2}\,,
\end{aligned}
\end{equation}
with $\bm{W}_0\big|_{\underline{t}=0}=0$. Using equation \eqref{eq:tauYterm_est} and Lemma \ref{lem:tensionXY_diff}, we may estimate $\bm{J}_1$ as 
\begin{equation}
\begin{aligned}
  \abs{\bm{J}_1} &\le c\big(\norm{T_{m_\epsilon^{\rm t}}(\P_{\X_s}\bm{W}_{ssss})}_{L^2_s} + \norm{T_{m_\epsilon^{\rm n}}(\P_{\X_s}^\perp\bm{W}_{ssss})}_{L^2_s}\big)\\
  &\quad + c\big(\norm{T_{m_\epsilon^{\rm t}}((\tau_W)_s\X_s)}_{L^2_s} + \norm{T_{m_\epsilon^{\rm n}}(\tau_W\X_{ss})}_{L^2_s}\big)\\
  &\quad + c\big(\norm{T_{m_\epsilon^{\rm t}}(\P_{\X_s}(\tau_Y\bm{W}_s)_s)}_{L^2_s} + \norm{T_{m_\epsilon^{\rm n}}(\P_{\X_s}^\perp(\tau_Y\bm{W}_s)_s)}_{L^2_s}\big)\\
  &\le c\abs{\log\epsilon}^{1/2}D_W(\underline{t})^{1/2}
  + c\abs{\log\epsilon}^{1/2}\bigg(\abs{\log\epsilon}^{1/2}\norm{\bm{W}_s}_{H^1}(1+\norm{\Y_s}_{H^3}^2)\\
  &\quad + \big(D_W(\underline{t})^{1/2} + 1 +\norm{\Y_s}_{H^3}^2\big)(1+\norm{\X_s}_{H^1}+\norm{\Y_s}_{H^3}^2)\bigg)\\
  &\quad + c\abs{\log\epsilon}\norm{\bm{W}_s}_{H^1}(\norm{\Y_s}_{H^2}^4+\norm{\Y_s}_{H^2}\norm{\Y_s}_{H^3})\\
  &\le c\abs{\log\epsilon}^{1/2}D_W(\underline{t})^{1/2}
  + c\abs{\log\epsilon}E_W(\underline{t})^{1/2}(1+\norm{\Y_s}_{H^3}^4)\\
  &\quad + c\abs{\log\epsilon}^{1/2}\big(1 +\norm{\Y_s}_{H^3}^2\big)(1+\norm{\X_s}_{H^1}+\norm{\Y_s}_{H^3}^2)\\
  &\le \abs{\log\epsilon}^{1/2}\big(cD_W(\underline{t})^{1/2}+C(T,\norm{\X}_{L^\infty_tH^2_s},\norm{\Y}_{L^\infty_tH^4_s\cap L^2_tH^6_s}) \big)\,,
\end{aligned}
\end{equation}
using \eqref{eq:EW_DW}.
Furthermore, using Corollary \ref{cor:mult_RFT_diff}, we may bound $\bm{J}_2$ as
\begin{equation}
\begin{aligned}
  \abs{\bm{J}_2} &\le c\bigg(\norm{T_{m_\epsilon^{\rm t}}^{1/2}(\X_s\otimes\X_s)}_{L^2_s} +\norm{T_{m_\epsilon^{\rm n}}^{1/2}({\bf I}-\X_s\otimes\X_s)}_{L^2_s} \bigg)\times\\
  &\qquad\times \bigg(\norm{\P_{\X_s}(\Y_{sss} -\tau_Y\Y_s)_s}_{L^2_s} + \norm{\P_{\X_s}^\perp(\Y_{sss} -\tau_Y\Y_s)_s}_{L^2_s} \\
  &\qquad\quad +\abs{\log\epsilon}^{1/2}\norm{\bm{W}_s}_{H^1_s}(1+\norm{\X_s}_{H^1_s}^2+\norm{\Y_s}_{H^1_s}^2)\norm{(\Y_{sss} -\tau_Y\Y_s)_s}_{H^{1/2}_s}\bigg)\\
  &\le c\abs{\log\epsilon}^{1/2}\big(1+\norm{\Y_s}_{H^3_s}^3\big)\\
  &\quad +c\abs{\log\epsilon}E_W^{1/2}(\underline{t})(1+\norm{\X_s}_{H^1_s}^2+\norm{\Y_s}_{H^1_s}^2)(1+\norm{\Y_s}_{H^3_s}^5+\norm{\Y_s}_{H^{7/2}_s})\\
  &\le \abs{\log\epsilon}^{1/2}C(T,\norm{\X}_{L^\infty_tH^2_s},\norm{\Y}_{L^\infty_tH^4_s\cap L^2_tH^6_s})\,,
\end{aligned}
\end{equation}
again by \eqref{eq:EW_DW}.
In total, we have that $\bm{W}_0$ satisfies
\begin{equation}
  \frac{\p}{\p\underline{t}}\abs{\bm{W}_0}^2 \le \abs{\log\epsilon}^{-1}\big(cD_W(\underline{t})+C(T,\norm{\X}_{L^\infty_tH^2_s},\norm{\Y}_{L^\infty_tH^4_s\cap L^2_tH^6_s}) \big)\,, 
\end{equation}
with $\abs{\bm{W}_0}^2\big|_{\underline{t}=0}=0$. Integrating in time from 0 to $T$, we obtain 
\begin{equation}
  \abs{\bm{W}_0}^2 \le \abs{\log\epsilon}^{-1}C(T,\norm{\X}_{L^\infty_tH^2_s},\norm{\Y}_{L^\infty_tH^4_s\cap L^2_tH^6_s})\,.
\end{equation}

Combining the bounds for $E_W$, $D_W$, and $\abs{\bm{W}_0}^2$, we obtain
\begin{equation}
\begin{aligned}
  &\norm{\bm{W}}_{L^\infty_t H^2_s} + \norm{T_{m_\epsilon^{\rm t}}^{1/2}(\P_{\X_s}\bm{W}_{ssss})}_{L^2_tL^2_s} + \norm{T_{m_\epsilon^{\rm n}}^{1/2}(\P_{\X_s}\bm{W}_{ssss})}_{L^2_tL^2_s}\\
  &\qquad \le \abs{\log\epsilon}^{-1/2}C(T,\norm{\X}_{L^\infty_tH^2_s},\norm{\Y}_{L^\infty_tH^4_s\cap L^2_tH^6_s})\,.
\end{aligned}
\end{equation}
Using the lower bound of Lemma \ref{lem:multiplierbds} for $T_{m_\epsilon^{\rm t}}$ and $T_{m_\epsilon^{\rm n}}$ yields Theorem \ref{thm:RFTconv}.
\hfill $\square$


\subsubsection*{Acknowledgments} LO acknowledges support from NSF grant DMS-2406003 and thanks Dallas Albritton for helpful comments on a draft of this paper.


\bibliographystyle{abbrv} 
\bibliography{StokesBib}

\end{document}